\documentclass[11pt,reqno]{amsart}
\usepackage{xifthen,setspace,bm}
\usepackage[utf8]{inputenc}
\usepackage[T1]{fontenc} 
\usepackage{amsmath,amsthm,amsfonts,amssymb,mathrsfs,stmaryrd,bigints}
\usepackage{dsfont}
\usepackage{mathtools}
\usepackage[usenames]{color}

\usepackage[colorlinks=true,linkcolor=blue]{hyperref}
\usepackage{epstopdf} 
\usepackage{amssymb}
\usepackage{setspace}
\usepackage{enumerate}
\usepackage{bigstrut}
\usepackage{multirow}
\usepackage{mathtools,xparse}
\usepackage{mathrsfs}
\usepackage{hyperref}
\usepackage{xcolor}
\usepackage [autostyle, english = american]{csquotes}
\usepackage{pgfplots}
\newtheorem*{lemma*}{Lemma}
\newtheorem{theorem}{Theorem}[section]
\newtheorem{corollary}[theorem]{Corollary}
\newtheorem{lemma}[theorem]{Lemma}
\newtheorem{proposition}[theorem]{Proposition}
\theoremstyle{definition}

\newtheorem{remark}[theorem]{Remark}

\DeclareMathOperator*{\argmax}{arg\,max}
\allowdisplaybreaks
\usepackage{chngcntr}

\def\it{\textit}

\newcommand{\ri}{\rightarrow \infty}
\newcommand{\ro}{\rightarrow 0}

\newcommand{\E}{{\mathbb{E}}}
\newcommand{\1}{\mathds{1}}

\newcommand{\R}{\mathbb{R}}
\newcommand{\norm}[1]{\left\lVert#1\right\rVert}

\newcommand{\e}{\varepsilon}

\newcommand{\N}{\mathbb{N}}

	\renewcommand{\P}{\mathbb{P}}

\newcommand{\cA}{\mathcal{A}}
\newcommand{\cB}{\mathcal{B}}
\newcommand{\cC}{\mathcal{C}}
\newcommand{\cD}{\mathcal{D}}
\newcommand{\cE}{\mathcal{E}}
\newcommand{\cF}{\mathcal{F}}
\newcommand{\cG}{\mathcal{G}}

\newcommand{\cM}{\mathcal{M}}

\newcommand{\cS}{\mathcal{S}}
\newcommand{\cT}{\mathcal{T}}
\newcommand{\cU}{\mathcal{U}}

\newcommand{\sG}{\mathscr{G}}

\DeclareMathOperator{\argmin}{arg\,min}

\renewcommand{\setminus}{\backslash}

\def\ba{\begin{align}}
\def\ea{\end{align}}
\def\bs{\begin{split}}
\def\es{\end{split}}

\begin{document}

\title[Largest eigenvalue of sparse Gaussian networks]{Large deviations for the largest eigenvalue of Gaussian networks with constant average degree}

\author{Shirshendu Ganguly and Kyeongsik Nam}

\begin{abstract} 
Large deviation behavior of the largest eigenvalue $\lambda_1$ of Gaussian networks (Erd\H{o}s-R\'enyi random graphs $\cG_{n,p}$ with i.i.d. Gaussian weights on the edges) has been the topic of considerable interest. In the recent works \cite{guionnet1,guionnet2}, a powerful approach was introduced based on tilting measures by suitable spherical integrals to prove a large deviation principle, particularly establishing a non-universal behavior for a fixed $p<1$ compared to the standard Gaussian ($p=1$) case.  The case when $p\to 0$ was however completely left open {with one expecting the dense behavior to hold only until the average degree is logarithmic in $n$.} In this article we focus on the case of constant average degree i.e., $p=\frac{d}{n}$ for some fixed $d>0$. Results in \cite{van1} on general non-homogeneous Gaussian matrices imply that in this regime $\lambda_1$ scales like $\sqrt{\log n}.$ We prove the following results towards a precise understanding of the large deviation behavior in this setting.

\begin{enumerate}
\item (Upper tail probabilities and structure theorem): For $\delta>0,$ we pin down the exact exponent $\psi(\delta)$ such that $$\P(\lambda_1\ge \sqrt{2(1+\delta)\log n})=n^{-\psi(\delta)+o(1)}.$$
Further, we show that conditioned on the upper tail event, with high probability, a \emph{unique maximal clique} emerges with a very precise $\delta$ dependent size (takes either \emph {one} or \emph{two} possible values) and the Gaussian weights are uniformly high in absolute value on the edges in the clique. Finally, we also prove an optimal localization result for the leading eigenvector, showing that it allocates most of its mass on the aforementioned clique which is spread uniformly across its vertices. 
\item (Lower tail probabilities):  The exact stretched exponential behavior $$\P(\lambda_1\le \sqrt{2(1-\delta)\log n})=\exp\left(-n^{\ell(\delta)+o(1)}\right)$$ is also established.
\end{enumerate}
As an immediate corollary, one obtains that $\lambda_1$ is typically $(1+o(1))\sqrt{2\log n}$, a result which surprisingly appears to be new.  A key ingredient in our proofs is an extremal spectral theory for weighted graphs obtained by an $\ell_1-$reduction of the standard $\ell_2-$variational formulation of the largest eigenvalue via the classical Motzkin-Straus theorem \cite{turan}, which could be of independent interest. 
\end{abstract}

\address{ Department of Statistics, Evans Hall, University of California, Berkeley, CA
94720, USA} 

\email{sganguly@berkeley.edu }

\address{ Department of Mathematics,  University of California, Los Angeles, CA
94720, USA} 

\email{ksnam@math.ucla.edu }


 
\maketitle

{\small\tableofcontents}

\section{Introduction}
Spectral statistics arising from random matrices and their asymptotic properties have  %
 been the subject of major investigations for several years. 
Fundamental observables of interest include the %
 empirical spectral measure as well as edge/extreme eigenvalues. %
The study of such quantities began in the classical 
setting of the Gaussian unitary and orthogonal ensembles (GUE and GOE) 
where the entries are complex or real i.i.d. Gaussians up 
to symmetry constraints. These exactly solvable examples admit complicated but explicit joint densities 
for the eigenvalues which can be analyzed, albeit involving a lot of work,  to pin down the precise behavior 
of several observables of interest.

The central phenomenon driving this article is the atypical behavior  
of the largest eigenvalue of a random matrix. This falls within the framework of large deviations which has attracted immense interest over the past two decades.

Perhaps not surprisingly, this was first investigated in the above mentioned exactly solvable cases 
\cite{freeentropy,adg}. Subsequently, 
Bordenave and Caputo  \cite{bc} considered empirical distributions in Wigner matrices with entries with heavier 
tails where large deviations is dictated by  a relatively small number of large entries. This phenomenon was 
shown for
the largest eigenvalue as well in \cite{augeri2}.

Another set of random matrix models arise from random graphs, particularly the Erd\H{o}s-R\'enyi graph $\cG_{n,p}$ on $n$ %
vertices %
with edge probability $p \in (0, 1)$. The literature on the
study of such graphs is massive with a significant 
fraction devoted to the study of spectral properties. A 
long series of works  established universality results for the bulk and edge of the spectrum in random graphs of 
average degree at least logarithmic in the graph size  drawing similarities to the Gaussian counterparts 
(cf. \cite{erdHos2012spectral, erdHos2013spectral} and the references therein). For sparser graphs, however, 
including the case of constant average degree which is the focus of this article, progress has been relatively limited. Nonetheless, some notable accomplishments 
include the results in \cite{adk, bbk_dense, bbk, KS} about the  edge of the spectrum, as well as the results 
of \cite{bordenave2017mean} and \cite{empirical_neighborhood}, which studied %
continuity properties of the
limiting spectral measure 
and a large deviation theory of the 
related local limits, respectively.  

While large deviations theory for linear functions of independent random variables is by now classical (see \cite{DZ}),  recently a powerful theory of  \emph{non-linear} large deviations has been put forth, developed over several articles (some of which are reviewed below), which treats non-linear functions such as the spectral norm of a random matrix with i.i.d. entries.

Among the recent explosion of results around this, 
a series of works investigated spectral large deviations 
for $\cG_{n,p}$, beginning with Chatterjee and 
Varadhan \cite{CV12}, where the authors proved a large
 deviation principle for the entire spectrum of $\cG_{n, p}$ at scale  $np$,  building on their seminal work
 \cite{CV11}, in the case where $p$ is fixed and does not depend on $n$ (dense case).
  However, the sparse case where $p=p(n) \rightarrow 0$ was left completely open until a major breakthrough 
 was made by Chatterjee and Dembo \cite{CD16}.  
 This led to considerable progress in developing the theory of 
 large deviations for various functionals of interest for sparse random graphs \cite{augeri1,austin,BM17,eldan,yan}. 
Via a refined understanding of cycle counts in $\cG_{n,p}$ which was obtained in 
\cite{augeri1,ab_rb,BGLZ,CV11,CD18,upper_tail_localization,LZ-dense,LZ-sparse}, one can deduce large 
deviation properties for eigenvalues using the trace method and this was carried out in \cite{uppertail_eigenvalue}.   
However such arguments only extended to $p$ going to zero at a rate slower than $1/\sqrt n$,  
since cycle statistics fail to encode information about the spectral norm for sparser graphs.
Such sparser graphs were treated more recently in \cite{ganguly1}, where the first named author along
 with Bhaswar Bhattacharya and Sohom Bhattacharya analyzed the large deviations behavior for the spectral 
edge for sparse $\cG_{n,p}$ in the  \emph{entire} ``localized regime'' when 
\begin{equation}\label{eq:rangep}
\log n \gg \log (1/np) \quad \text{ and } \quad n p \ll \sqrt{\frac{\log n}{\log \log n}},
\end{equation} 
where the extreme eigenvalues are governed by high 
degree vertices. This notably includes the well studied 
example of constant average degree.

In a related direction, which merges the above two classical settings, a series of works \cite{guionnet1,guionnet2} have explored universality of large deviations behavior for the largest eigenvalue for a Wigner matrix with i.i.d sub-Gaussian entries. First in \cite{guionnet1},
 it was shown that 
if the 
Laplace transform
 of the entries is pointwise 
 bounded by those of a standard real or complex Gaussian, then
a universal large deviation principle (LDP) same as in the  Gaussian case holds. 
 Examples of this include Rademacher variables and uniform variables. 
 However the situation changes  when the sub-Gaussian tails are not sharp. Perhaps 
 the most interesting examples in this class are sparse Gaussian matrices whose entries are obtained
by multiplying a Gaussian variable with an independent Bernoulli random variable with mean $p$.
 In \cite{guionnet2}, this more general setting is investigated and it is shown that the rate function for the
LDP can indeed be different from the Gaussian case. 
The approach in both \cite{guionnet1} and \cite{guionnet2} broadly relies on considering appropriate tilts of the original measures and analyzing the associated spherical integrals.
However, the above approach has been shown to work only in  the `dense' case of constant $p$ where the typical behavior is still the same as when $p=1,$ leaving the sparse regime $p\ll 1$ completely open. 

In the case of the random graph $\cG_{n,p}$, as established in \cite{KS, erdHos2012spectral,erdHos2013spectral}, at the law of large numbers level,  $\lambda_1=(1+o(1))\max(d_1, np)$
where $d_1$ denotes the maximum degree of the 
random graph. Consequently $\lambda_1$ exhibits a transition at $np=\sqrt{\frac{{\log n}}{\log \log n}}$, where 
the largest eigenvalue begins to be governed by the largest degree.
A similar phenomenon reflecting this transition for large deviations was established across the papers \cite{uppertail_eigenvalue, ganguly1}.
{In the case of Gaussian ensembles, although a precise result does not appear in the literature to the best of the 
authors' knowledge, it is expected that the dense behavior extends to the case of the average degree 
being logarithmic in $n$ (an analogous result for Wigner matrices with bounded entries, which is more 
comparable to the setting of random graphs, was established in \cite{tikhomirov}).  Beyond this, as the graph becomes sparser, a different behavior is expected to emerge.}

This motivates the present work where we obtain a very precise understanding of the  case of constant average 
degree, i.e., $p=\frac{d}{n},$ arguably the most 
interesting sparse case because of its connections to various models of statistical mechanics.

Also relevant to this paper is a  different line of research, which, motivated by viewing a random matrix as a random linear operator, considers `non-homogeneous' matrices. The most well studied example is a Gaussian matrix where the variance varies from entry to entry. In this general setting, even the leading order behavior for the spectral norm is far from obvious and requires a much more refined understanding beyond the concentration of measure bounds obtained as a consequence of the non-commutative Khintchine inequality. A beautiful conjecture posed by Latala \cite{latala} related to an earlier result of Seginer \cite{seginer} states that the expected spectral norm for such non-homogeneous Gaussian matrices, is up to constants the expectation of the {maximum $\ell_2$ norm of a row and a column}, and after a series of impressive accomplishments \cite{van1,van2}, the conjecture was finally settled in the beautiful work \cite{van3}.

Note that sparse Wigner matrices, quenching on the sparsity, falls in the above framework where the variance of each entry is $0$ or $1.$ It is worth mentioning that while the dependence on $n$ in the leading order behavior is pinned down in the above mentioned works, the techniques are not sharp enough to unearth finer properties such as the exact constant multiplicative pre-factor.\\

We now move on to the statements of the main theorems after setting up some basic notations. 
\subsection{Setup and main results:}
We will denote  by $\sG_n$ denote the set of all simple, undirected networks on $n$ vertices labelled $[n]:=\{1, 2, \ldots, n\}$ i.e., simple graphs with a conductance value on each edge. 
For $G \in \sG_n,$ denote by $A(G)=(a_{ij})_{1 \leq i, j \leq n},$ the adjacency matrix of $G$, that is $a_{ij}$ is
the conductance associated to the edge $(i,j)$  if the latter is an edge in $G$, and 0 otherwise.
Thus graphs are trivially encoded as networks where the entries of $A$ are $0$ or $1.$
For $F \in \sG_n$, since $A(F)$ is a self-adjoint matrix, denote by $\lambda_1(F) \geq \lambda_2(F) \geq \cdots \geq \lambda_n(F)$
its eigenvalues in non-increasing order, and let $\|F\|_{\rm op}:=\|A(F)\|_{\rm op}=\max\{|\lambda_1(F)|, |\lambda_n(F)|\}$ be the operator norm of $A$. 
Throughout most of the paper we will be concerned with $\lambda_1(F)$ and for notational brevity we will often drop the subscript to denote the same.

{In this paper we are interested in the sparse 
Erd\H os-R\'enyi random graph $\cG_{n, p}$, where $p=\frac{d}{n}$ for some $d>0$ which does not depend on $n$.
 We will denote by $X$ the random adjacency matrix associated to it. 
Thus for all $1\le i<j\le n$, $X_{i,j}$ is an independent 
Bernoulli random variable with mean $p,$ and $X_{ii}=0$ for all $i.$}
Let $Y$ be a standard GOE matrix, i.e. $Y_{ij}\sim N(0,1)$ for $i\le j$. The matrix of interest for us is $Z=X\odot Y,$ i.e.,  $Z_{ij}=X_{ij}Y_{ij}.$

Let $\lambda_1\geq \lambda_2 \geq \cdots \geq \lambda_n$ be eigenvalues of the matrix $Z$.
As a consequence of the already referred to work on the 
behavior of the spectral norm of general 
inhomogeneous Gaussian matrices \cite{van1},  it 
follows that \begin{align}\label{lln}
\E(\lambda_1) \approx \sqrt{\log n}.
\end{align}
One also obtains concentration around $\E(\lambda_1)$ using standard Gaussian techniques, see e.g. \cite[Corollary 3.9]{van1}. However so far, the methods have not been able to obtain a sharper understanding including the precise constant in front of $\sqrt{\log n}$ which we deduce as a simple corollary of our main theorems.
We now move on to the exact statements of the results in this paper. 
%
%
%
%
%
%
\begin{theorem}[Upper tail probabilities]\label{main theorem 1}
For $\delta>0$, define a function $\phi_\delta : \mathbb{N}_{\geq 2} \rightarrow \R$\footnote{
$\N$ will be used to denote the set of natural numbers, and $\N_{\ge k}$ to denote all the natural numbers bigger equal to $k.$}
 by
\begin{align}\label{keydef}
{\phi_\delta (k):= \frac{k(k-3)}{2} + \frac{1+\delta}{2} \frac{k}{k-1}}
\end{align}
and $\psi(\delta) := \min_{k\in\mathbb{N}_{\geq 2}} \phi_\delta (k)$. Then,
\begin{align} \label{112}
\lim_{n\ri}   -\frac{1}{\log n} \log  \mathbb{P}(\lambda_1 \geq \sqrt{2(1+\delta) \log n}  )=    \psi(\delta).
\end{align}
\end{theorem}

\begin{remark}[Infinite phase transition in upper tail] \label{remark 1.2}
The rate function given by  \eqref{112} is a continuous piecewise linear function with infinitely many pieces which we now describe in detail. 
Since we will only be concerned about the $\argmin$ restricted to integers larger than $1,$ we consider momentarily $\phi_\delta(x)  = \frac{x(x-3)}{2} + \frac{1+\delta}{2} \frac{x}{x-1} $ as a function of real numbers greater than one and notice that,
\begin{align} \label{104}
\phi'_\delta(x)= x-\frac{3}{2} - \frac{1+\delta}{2}\frac{1}{(x-1)^2}.
\end{align}
Thus $\phi_\delta(x)$ is a strictly convex function.
Let
$\cM(\delta) = \{\argmin_{k\geq 2} \phi_\delta(k)\}$ be the set of minimizers of $\phi_\delta(\cdot)$. By the strict convexity of $\phi_\delta(\cdot)$, $\cM(\delta)$ is at most of size $2$ containing either a single element or two consecutive integers. Precisely, denoting by $x(\delta)>1$, the unique solution to $\phi_\delta'(x)=0$, any element in $\cM(\delta)$ is either $\left \lfloor{x(\delta)}\right \rfloor $ or $\left \lceil{x(\delta)}\right \rceil $.
Now the values of $\delta$ for which $\cM(\delta)$ is of size two forms a discrete set. That is, there exists $0=\delta_1<\delta_2<\delta_3<\cdots$ such that the following holds: for any positive integer $k\geq 2$,  $ (\delta_{k-1}, \delta_{k})$ is the collection of $\delta$ such that $ \cM(\delta) = \{k\}$ and $\delta_{k}$ is the unique $\delta$ such that    $ \cM(\delta) = \{k,k+1\}$.   To see this, since $\delta \mapsto x(\delta)$ is strictly increasing, it suffices to verify  that the situation $\delta_1<\delta_2$, $\phi_{\delta_1}(k+1)\leq \phi_{\delta_1}(k)$ and  $\phi_{\delta_2}(k)\leq \phi_{\delta_2}(k+1)$ never occurs. Observe that the contrary implies 
\begin{align*}
\phi_{\delta_1}(k+1)\leq \phi_{\delta_1}(k) \leq \phi_{\delta_2}(k)\leq \phi_{\delta_2}(k+1).
\end{align*}
By \eqref{104},  $\phi_{\delta_1}'(x) > \phi_{\delta_2}'(x)$,  which contradicts the above.

Hence, for  $\delta \in [\delta_{k-1}, \delta_{k}]$,
\begin{align*}
\psi(\delta) = \frac{1+\delta}{2}\frac{k}{k-1} +  \frac{k(k-3)}{2},
\end{align*}
which is a linear function in $\delta \in [\delta_{k-1}, \delta_{k}]$ for any  fixed $k\geq 2$. This implies that $\psi(\delta)$ is a continuous piecewise linear function.

Also by a simple algebra, it follows from \eqref{104} that $$( \frac{1+\delta}{2})^{1/3} + 1 < x(\delta) < ( \frac{1+\delta}{2})^{1/3} + \frac{3}{2}.$$ Since $$\phi_\delta \Big ( \Big(\frac{1+\delta}{2}\Big)^{1/3}\Big)  =  \frac{1}{2}\delta + \frac{3}{2^{5/3}} \delta^{2/3}+ O(\delta^{1/3}),$$ we obtain 
\begin{align} \label{103}
\psi(\delta) = \frac{1}{2}\delta + \frac{3}{2^{5/3}} \delta^{2/3} +  O(\delta^{1/3}) \qquad \text{as} \quad \delta\ri,
\end{align}
where 
$O(\delta^{1/3})$ is a quantity bounded by $C \delta^{1/3}$ for some absolute constant $C>0$. 
Plugging this into \eqref{112}, one thus obtains the following asymptotic behavior of the upper tail probabilities\footnote{Throughout the paper, $o(1)$ will be used to denote functions of $n$ that tend to $0$ as $n$ tends to infinity. However we will also need to deal with quantities that go to zero as $\delta$ converges to infinity, which would be denoted by $o_\delta(1).$}
{\begin{align}
\label{largedelta}
 \mathbb{P}(\lambda_1 \geq \sqrt{2(1+\delta) \log n}  ) &= n^{- ( \frac{1}{2}\delta + \frac{3}{2^{5/3}} \delta^{2/3}+O(\delta^{1/3}))}
\text{ for large $\delta>0$, and,}\\
\label{smalldelta}
 \mathbb{P}(\lambda_1 \geq \sqrt{2(1+\delta) \log n}  ) &= n^{-\delta+o(1) } \text{ for small $\delta>0$.}
\end{align}
}
\end{remark}

\begin{remark}[Comparison with maximum of i.i.d. Gaussians] As the reader possibly already notices, for small $\delta$, the behavior in \eqref{smalldelta} is the same as that for the maximum of $n$ many standard Gaussian variables. {The reason for this will be discussed in the idea of proofs section}. 
\end{remark}

Having established the sharp order of the tail probabilities, we now state three results establishing a sharp structural behavior conditioned on the upper tail event $\cU_{\delta}:=\{\lambda_1  \geq   \sqrt{2(1+\delta) \log n}\},$ thus unearthing the dominant mechanism dictating upper tail large deviations. 
The first result shows the existence of a clique of a very precise $\delta$ dependent size establishing a sharp concentration for the maximal clique size conditioned on $\cU_\delta$.
For any graph $G\in \cG_n,$ let $k_G$ be the size of a maximal clique $K_G$ in $G$. {Recall the definition of $\cM(\delta)$ and let 
$h(\delta)$ be the smallest element of $\cM(\delta)$.
By Remark \ref{remark 1.2},
{\begin{align}\label{argmaxloc}
\Big\vert h(\delta) - \Big( \frac{1+\delta}{2}\Big)^{1/3} - 1 \Big\vert \leq 2.
\end{align}
}}
\begin{theorem} [Structure theorem] \label{theorem structure}
For any $\delta$ with $h(\delta) \geq 3,$ i.e., $\delta>\delta_2$ (see Remark  \ref{remark 1.2} for the definition of $\delta_k$),
\begin{align} \label{150}
\lim_{n\ri} \mathbb{P} \Big( k_X \in \cM(\delta) \mid \cU_\delta   \Big)  = 1.
\end{align}
Furthermore, with conditional probability tending to one, $K_X$ is unique and any clique of size at least $4$ is a subset of $K_X.$
\end{theorem}

{Note that the above statement in particular implies that the largest clique outside $K_X$ is a triangle whose occurrence has constant probability.}
Thus the above result proves a two point concentration for the maximal clique size and for values of $\delta$ such that $\cM(\delta)$ only contains $h(\delta),$ it implies a one point concentration.

Our next result asserts that the most of the contribution to the spectral norm comes from $K_X,$ with the Gaussians along the edges of the latter being uniformly high in absolute value. 
 
\begin{theorem}[Uniformly high Gaussian values] \label{uniform gaussian}
There exists    $\zeta = \zeta(\kappa)>0$ with $\lim_{\kappa \rightarrow 0} \zeta = 0 $ such that  the following holds. 
For $\kappa>0$, {for $\delta$ large enough}, with probability (conditional on $\cU_{\delta}$) going to $1,$ there exists $T\subset K_X$ such that $|T|\ge (1-\kappa)h(\delta)$ and
\begin{align}
   \frac{1}{h(\delta)^2}   \sum_{i\neq j, i,j\in T} \Big\vert |Z_{ij}|  -  \frac{1}{h(\delta)}  \sqrt{2(1 + \delta)\log n}  \Big\vert  \leq \frac{\zeta}{h(\delta)}  \sqrt{2(1 + \delta)\log n} .
\end{align}
\end{theorem}

Even though in the statement $\delta$ is chosen large enough as a function of $\kappa,$ the proof will in fact give a quantitative, albeit technical, bound for all large $\delta$ and small $\kappa$ which can then be simplified into the form of the statement of the theorem by choosing $\delta$ dependent on $\kappa.$

{
Since the maximal clique $K_X$ has  size $h(\delta)$ or $h(\delta)+1$ with probability going to 1 (conditional on $\cU_\delta$), the above theorem shows that the Gaussian values $Z_{ij}$ on $K_X$ are uniformly high in absolute value and close to $\frac{1}{h(\delta)}  \sqrt{2(1 + \delta)\log n}  $ in the $\ell_1$ sense.
}

Our final structural result is an optimal localization statement about the leading eigenvector.

\begin{theorem}[Optimal localization of eigenvector]\label{eigenvectorloc}
Let $\textbf{v} = (v_1,\cdots,v_n)$ be the top eigenvector with $\norm{\textbf{v}}_2 = 1$ and consider the unique maximal clique $K_X$ and its size $k_X$ from Theorem \ref{theorem structure}.  For   $\kappa>0$, define the events
\begin{align*}
\cA_1 := \Big \{\sum_{i\in K_X} v_i^2 \geq 1-\kappa\Big\}
\end{align*}
and
\begin{align*}
\mathcal{A}_2= \Big\{ \frac{1}{k_X}\sum_{i\in K_X} \Big(v_i^2  - \frac{1}{k_X} \Big)^2    \leq     \frac{40\kappa}{k^2_X}  \Big\}.
\end{align*} 
Then, for sufficiently large $\delta>0
$,
\begin{align} \label{1500}
\lim_{n\ri} \mathbb{P}( \cA_1\cap \cA_2 \mid  \,\cU_{\delta} )=1.
\end{align}
\end{theorem}

Thus the above theorem says, for any $\kappa>0,$ for all large enough $n,$ conditioned on $\cU_\delta,$  the leading eigenvector distributes at least $1-\kappa$ mass on $K_X$ almost uniformly.

Note that the last two theorems do not claim anything about the sign of the entries of the eigenvector or the Gaussian values. This is since switching the signs of the entries of the largest eigenvector arbitrarily and accordingly changing the signs of the Gaussians yields the same quadratic form.

Having stated our results concerning upper tail deviations, the next result pins down the lower tail large deviation probability.

\begin{theorem}[Lower tail probabilities]\label{main theorem 2}
For any $0<\delta<1$,
\begin{align}
\lim_{n\ri}   \frac{1}{\log n}  \Big( \log \log \frac{1}{ \mathbb{P}(\lambda_1 \leq \sqrt{2(1-\delta) \log n}  )}  \Big) = \delta  .
\end{align}

\end{theorem}

As an immediate corollary of Theorems \ref{main theorem 1} and \ref{main theorem 2}, one obtains the following \emph{`law of large numbers'} behavior which we were  surprised to not be able to locate in the literature.
\begin{corollary}
We have
\begin{align*}
\lim_{n\ri}  \frac{\lambda_1}{\sqrt{\log n}} = \sqrt{2}
\end{align*}
in probability.
\end{corollary}

We conclude this discussion by remarking that
although in principle our techniques may be used to analyze a wider subset of the parameter space, we have, for concreteness and aesthetic considerations, chosen to simply focus on the case of constant average degree.

\subsection{Organization of the article}
In Section \ref{s:iop} we provide a detailed account of the keys ideas driving the proofs. In Section \ref{swg}, we state and prove the key Proposition \ref{spectral bound} obtaining a bound on the spectral norm in terms of the Frobenius norm for weighted graphs.
The rest of the paper focuses on the proofs of  Theorem \ref{main theorem 1} (Sections \ref{section 3}, \ref{section 3.2}), Theorem \ref{theorem structure} in Section \ref{structureproof},  Theorem \ref{eigenvectorloc} in  Section \ref{section 7}, Theorem \ref{uniform gaussian} in Section  
\ref{section 8} and Theorem \ref{main theorem 2} in Section \ref{section 4} respectively. Certain straightforward but technical estimates are proved in the appendix.

\subsection{Acknowledgement} The authors thank Noga Alon for pointing out the classical reference \cite{turan}. 
SG is partially supported by NSF grant DMS-1855688, NSF CAREER Award DMS-1945172 and a Sloan Research Fellowship. KN is supported by UCLA Mathematics department. This work was initiated when SG was participating in the  Probability, Geometry, and Computation in High Dimensions program at the Simons Institute in Fall 2020. 

\section{Key ideas of the proofs}\label{s:iop}
In this section we provide a sketch of the arguments in the proofs of our main results. \\

\noindent
\textit{Upper tail lower bound:} This is straightforward.
The strategy is to plant a clique of an appropriate size $(\argmax_{k} \phi_{\delta}(k))$ and have high valued Gaussians on all the clique edges, i.e., at least ${\frac{\sqrt{2(1+\delta)\log(n)}}{k-1}}$.
The probability of a clique of size $k\ge 3$ appearing is up to constants $n^{k-{k\choose 2}}$ (the proof follows by a second moment argument) while the probability of having high Gaussians is
\begin{align*}
\mathbb{P}\Big({Y_{ij}} \geq  {\frac{\sqrt{2(1+\delta)\log(n)}}{k-1}},\ \forall \,\,1\leq i<j\leq k \Big)  
&\geq {\Big( \frac{C}{\sqrt{\log n}} n^{-\frac{1+\delta }{(k-1)^2}}\Big)^{{k \choose 2}}},
\end{align*}
where the right hand side follows from standard Gaussian tail bounds (see \eqref{tail} later). 
Thus the total cost at the polynomial scale is $n^{k-{{k\choose 2}}}n^{-\frac{1+\delta }{(k-1)^2}{{k \choose 2}}}.$ 
Observe that the exponent is precisely $-\phi_{\delta}(k).$
When $k=2$, one should view it slightly differently however, since $k-{k \choose 2} =1>0$. Namely, there are order $n^{{k-{{k\choose 2}}}}=n$ many edges and hence the probability that there exists a Gaussian of value at $\sqrt{2(1+\delta)\log n}$ is $nn^{-(1+\delta)}=n^{-\delta}=n^{-\phi_\delta(2)}.$
Finally, optimizing over $k$ yields the bound $n^{-\psi(\delta)}.$

It is worth noticing the contrasting behavior in the absence of the Gaussian variables, where in \cite{ganguly1} it was shown that large deviations for the largest eigenvalue is guided by the large deviations for the maximum degree and not by appearance of a clique.\\

\noindent
\textit{Upper tail upper bound:} This is the most difficult among the four bounds and a significant part of the work goes into proving this.
The first step is to make the underlying graph sparser by only focusing on the Gaussians with a large enough value. 
As will be apparent soon, the reason for this is two-fold.
a) It is much harder for the graph restricted to small Gaussian values to have a high spectral norm, and so for our purposes we will treat that component as \emph{spectrally negligible}, b) The graph restricted to high Gaussian values is much sparser and hence admits greater shattering into smaller components whose sizes we can control; since eigenvalues of different components do not interact with each other, this will be particularly convenient. 

Proceeding to implement this strategy,
decompose the Gaussian random variables $Y_{ij}$ as
\begin{align*} 
Y_{ij} = Y^{(1)}_{ij} + Y^{(2)}_{ij},
\end{align*}
where $Y^{(1)}_{ij}=  Y_{ij}\1_{|Y_{ij}| > \sqrt{\varepsilon \log \log n}}$ and similarly $Y^{(2)}_{ij}=  Y_{ij}\1_{|Y_{ij}| \leq \sqrt{\varepsilon \log \log n}}.$
Thus, we can write the matrix $Z$ as $Z^{(1)}+Z^{(2)}$ with 
\begin{align}\label{decomposition1}
Z^{(1)}_{ij} = X_{ij} Y^{(1)}_{ij}, \quad  Z^{(2)}_{ij} = X_{ij} Y^{(2)}_{ij},
\end{align} 
and similarly $ X=X^{(1)}+X^{(2)}$ i.e., $X^{(1)}_{ij}=X_{ij}\1_{|Y_{ij}| > \sqrt{\varepsilon \log \log n}}.$
We next prove an upper bound on the probability that $Z^{(2)}$ has high spectral norm which is much smaller than that for $Z$ which implies that the spectral behavior of $Z$ even under large deviations is dictated by that of $Z^{(1)}.$
The choice of the truncation threshold is governed by the fact that the typical spectral norm of $\cG_{n,\frac{d}{n}}$ is of order $\sqrt{\frac{\log n}{\log \log n}}$ which in itself is a consequence of the fact that the maximum degree is of order ${\frac{\log n}{\log \log n}}$. Sharp large deviations behavior for eigenvalues of sparse random graphs was recently established in the already mentioned work \cite{ganguly1} which we use to make this step precise.

This allows one to focus simply on $Z^{(1)}$ or the underlying graph $X^{(1)},$ conditioning on which makes the spectral behavior of the individual connected components independent guided by the Gaussian variables each of which are conditioned to be at least $\sqrt{\e} \sqrt {\log\log n}.$

Let $C_1,\cdots,C_k$ be its connected components.
At this point denoting the network $Z$ restricted to $C_{\ell}$ by $Z_{\ell}$, we relate $\|Z_{\ell}\|_{\rm op}$ to its Frobenius norm $\|Z_{\ell}\|_{F}$. 
The trivial bound $\|Z_{\ell}\|_{\rm op}\le \|Z_{\ell}\|_{F}$ is easy to see.  
The next idea which is the key one in this paper relies on the following sharp improvement over the above. Namely we show that if $k_{\ell}$ is the size of the maximal clique in $Z_{\ell},$ then  
\begin{equation}\label{keyineq}
\|Z_{\ell}\|^2_{\rm op}\le \frac{k_{\ell}-1}{k_{\ell}}\|Z_{\ell}\|^2_{F}.
\end{equation}
The proof of the above relies on reducing the standard $\ell_2$ variational problem for the spectral norm to an $\ell_1$ version which can be solved by `mass transportation' techniques.  And the above leads us to a bound of the form
\begin{equation}\label{crucin1234}
\P(\|Z_{\ell}\|^2_{\rm op} \ge 2(1+\delta)\log n)\le \P\Big(\|Z_{\ell}\|^2_{F} \ge \frac{k_{\ell}}{k_{\ell}-1}2(1+\delta)\log n\Big).
\end{equation}

Now quenching the graph $X,$ the random variable $\|Z_{\ell}\|^2_{F}$ can be viewed at first glance as a chi-squared random variable with degrees of freedom given by the component size $|E(C_{\ell})|$.
Now as long as $|E(C_{\ell})|$ is $o(\log n),$ the degree of freedom does not affect the latter probability in its leading order behavior and it behaves as the square of a single Gaussian. This is what justifies the sparsification step mentioned at the outset which ensures that $|C_{\ell}|=O_{\e}(\frac{\log n}{\log\log n})$ which along with the tree like behavior of $C_{\ell}$ implies $|E(C_{\ell})|=O_{\e}(\frac{\log n}{\log\log n})$ as well (Here $O_\e(\cdot)$ is the standard notation denoting  that the implicit constant is a function of $\e.$) 

However there is one crucial subtlety that we have overlooked so far. Namely, $\|Z_{\ell}\|^2_{F}$ is not simply a chi-squared random variable but instead is a sum of squares of independent Gaussian variables each conditioned to {have an absolute value} at least $\sqrt{\e \log\log n}.$ 
This makes the tail heavier by the exact amount which on interacting with the $\e$ dependence in the size of $C_{\ell}$ begins to affect the leading order probability. Thus unfortunately the  above strategy ends up not quite working.

To address this we further rely on the fact that $C_{\ell}$ is almost tree-like and has a bounded number of `tree-excess edges' with high probability and revise our strategy in the following way. Consider the eigenvector $v$ corresponding to the largest eigenvalue $\lambda(\ell):=\lambda_1(C_{\ell}).$  Thus we know $v^{\top}Z_{\ell}v=\lambda(\ell).$

The key idea now is to split the vertices of $C_{\ell},$ according to high and low values of $v.$ We first show that it is much more costly for the Frobenius norm to be high on the subgraph induced by the low values of $v.$ This is where the tree like property is crucially used as well.

Thus we focus only on the $O(1)$ vertices supporting high $v$ values and since the maximum degree is $O(\frac{\log n}{\log\log n})$ (\emph{without} an $\e$ dependence in the constant), the strategy originally outlined can be made to work for the subgraph induced by these vertices.\\

While the next three proofs are rather technically involved, here we simply review the high level strategies involved. \\

\noindent
\textit{Emergence of a unique maximal clique:} The above proofs imply that the graph $X^{(1)}$ under $\cU_\delta$ contains a clique $K_{X^{(1}}$ whose size is sharply concentrated on  $\cM(\delta)$ (where the latter appearing in the statement of Theorem \ref{theorem structure} denotes the set of minimizers of $\phi_{\delta}(\cdot)$). It also follows that $K_{X^{(1)}}$ is unique. We then show that on account of sparsity, superimposing $X^{(2)}$ on $X^{(1)}$ does not alter this. Particularly convenient is the fact that conditional on $X^{(1)},$ the spectral behavior of $Z^{(1)}$ and the random graph $X^{(2)}$ are independent. However making this precise is delicate and is one of the most technical parts of the paper, relying on a rather refined understanding of the graph $X^{(1)}$ under the large deviation behavior of $\lambda(Z^{(1)})$ Such understanding also allows us to show that there does not exist any other clique in $X$ of size at least $4$ which is not contained in $K_X$.\\

\noindent
\textit{Localization of the leading eigenvector:}
The proof of this is reliant on the fact that \eqref{keyineq} is sharp only when the leading eigenvector is supported on the maximal clique $K_X$. We prove a quantitative version of this fact showing that significant mass away from the clique results in a deteriorated form of \eqref{keyineq} which then makes $\cU_{\delta}$ much more costly than the already proven lower bound for its probability. Further a similar approach is used to prove the desired flatness of the vector on $K_X.$\\

\noindent
\textit{Flatness of the Gaussian values on the maximal clique.}
Using the previous structural result about the leading eigenvector $v=(v_1,v_2,\ldots, v_n)$, we consider the set $T\subset K_X$ such that ${|v_i|}\approx \frac{1}{k_X}$ for all $i\in T$ (we don't make the meaning of $\approx$ precise) The previous results guarantee that, conditional on $\cU_\delta,$ 
$|T|\ge (1-\kappa) k_X$ and $|k_X-h(\delta)|\le 1.$ 
Firstly showing that the spectral contribution from the edges incident on $T^c$ is negligible, it follows that the quadratic form $v^{\top}Zv \approx v_T^{\top}Z_Tv_T$ where $v_T$ and $Z_T$ are the restrictions to the subgraph induced on $T.$ Now owing to the flatness of $v$ on $T$ (and this is why we work with $T$ and not $K_X$), it follows that $$v_T^{\top}Z_Tv_T\le 2\frac{(1+o_\delta(1))}{h(\delta)}\|Z_T\|_1,$$
{ where the $\ell_p$ norm $Z_T$ is defined by
\begin{align*}
\norm{Z_T}_p:=\Big  ( \sum_{i<j, i,j\in T} |Z_{ij}|^p\Big  )^{1/p}.
\end{align*}
}
Using this and the fact that $|k_X-h(\delta)|\le 1$ with high probability, we obtain the bound 
$$\|Z_T\|_1\approx  \frac{1}{2}  h(\delta)\sqrt{2(1 + \delta')\log n}.$$ In fact the above argument only implies a lower bound, while the upper bound follows from the following sharp bound on the $\ell_2$ norm which is a consequence of previous arguments (e.g. \eqref{crucin1234}). $$\|Z_T\|^2_2 \approx  (1+o_\delta(1)) (1+\delta)  \log n   .$$

Using the above two bounds, one can conclude the statement of the theorem in a straightforward fashion.\\

\noindent
\textit{Lower tail:} The upper bound can be obtained simply by a comparison with the maximum of $O(n)$ many independent Gaussians.

For the lower bound, $Z^{(2)}$ can still be considered spectrally negligible, while for $Z^{(1)},$ conditioning on $X^{(1)}$ being `nice', with none of the components being too large while also having at most bounded tree excess we use the results about the upper tail to upper bound the probability that for any connected component $C_{\ell}$, $\lambda(\ell)\ge \sqrt{2(1-\delta)\log n}$ or in other words lower bound $\P(\lambda(\ell)\le \sqrt{2(1-\delta)\log n})$ where $\lambda(\ell):=\lambda_1(C_\ell)$. Since $\lambda_1(Z)= \max_{\ell} (\lambda(\ell))$ and, conditioning on the graph makes $\lambda(\ell)$ across  different values of $\ell$ independent, the result follows in a straightforward fashion.

\section{Spectral theory of weighted graphs}\label{swg}
As outlined in Section \ref{s:iop}, a key ingredient in our proofs is a new deterministic bound on the spectral norm in terms of the Frobenius norm by an $\ell_2\to \ell_1$ reduction. Though this is independently interesting, the proofs are somewhat technical and the reader only interested in the large deviations aspect, at first read can simply treat this result as an input in the proof of Theorems \ref{main theorem 1}.
\subsection{Spectral norm and Frobenius norm}
{  For a Hermitian matrix $A$ of size $n\times n$,
let $\lambda_1\geq \cdots \geq \lambda_n$ be the eigenvalues in a non-increasing order.}
Then,  we have
\begin{align*}
\text{tr}(A^k)  = \lambda_1^k + \cdots + \lambda_n^k,
\end{align*}
which immediately implies that for any even positive integer $k$,
\begin{align}
\lambda_1^k \leq  \text{tr}(A^k) \leq n \lambda_1^k.
\end{align}
We denote by $\norm{A}_F$, the Frobenius norm of the matrix $A$:
\begin{align*}
\norm{A}_F :=  (\text{tr}(A^2)   )^{1/2}=  \Big(\sum_{1\leq i,j\leq n} a_{ij}^2\Big)^{1/2},
\end{align*}
Then, taking $k=2$ above, we record the following trivial  bound 
\begin{align} \label{fro}
\lambda_1^2 \leq \norm{A}_F^2.
\end{align}

\subsection{Refined bound on spectral norms for weighted graphs}

We now move on to a sharp bound on the spectral norm in terms of the Frobenius bound for networks improving the above.

Before stating the result let us discuss a situation where one already obtains an improvement over \eqref{fro}, namely for bipartite graphs. This is because of the underlying symmetry in the spectrum, as a consequence of which we get $\lambda_1=-\lambda_n$ and hence \begin{align*}
\lambda_1(A)^2 \leq \frac{1}{2}\norm{A}_F^2.
\end{align*}
 The main result of this section is a new and sharp generalization of this inequality.

\begin{proposition} \label{spectral bound}
Let $k$ be the maximal size of clique contained in $G$. Then, for any conductance $a:E\rightarrow \R$, we have
\begin{align}
\lambda_1(A)^2 \leq  \frac{k-1}{k} \norm{A}_F^2
\end{align}
\end{proposition}

\begin{remark} \label{clique}
For $G$ a clique  of size $k$, with adjacency matrix $A$, it is straightforward to see that 
\begin{align} \label{220}
\lambda_1(A)^2  =  \frac{k-1}{k} \norm{A}_F^2.
\end{align}
This follows from the fact that a $k\times k$ matrix  whose off-diagonal entries are 1 and on-diagonal entries are 0 has the largest eigenvalue $k-1$ and the Frobenius norm  $ \sqrt{k^2-k}$.
\end{remark}

The proof of the proposition will rely crucially on the following bound which goes back to the seminal work of Motzkin and Straus \cite{turan} whose proof we include for completeness. 

\begin{lemma} \label{lemma 02}
Suppose that $k$ is the maximal size of clique contained in the graph $G$ with  vertex set $[n]$. Let   $f=(f_1,\cdots,f_n)$ be a vector with $\sum_{i=1}^n  f_i=s$ and $f_i\geq 0$.  Then,
\begin{align} \label{020}
\sum_{i<j, i\sim j} f_if_j  \leq \frac{k-1}{2k}s^2.
\end{align}
\end{lemma}

We first furnish the proof of the proposition before proving the above lemma.
\begin{proof}[Proof of Proposition \ref{spectral bound}]
By the variational characterization of the largest eigenvalue,
\begin{align*}
\lambda_1(A) = \sup_{ \norm{f}_2  =1} \sum_{i\sim j}  a_{ij} f_if_j.
\end{align*}
Thus, {for any conductance $a:E\rightarrow \R$,}
\begin{align*}
  \frac{\lambda_1(A)}{\norm{A}_F} &=  \sup_{ \norm{f}_2  =1}  \frac{ \sum_{i\sim j}  a_{ij} f_if_j }{\norm{A}_F} \\
  &\leq    \sup_{ \norm{f}_2  =1}   \frac{  (\sum_{i\sim j} a_{ij}^2)^{1/2} (\sum_{i\sim j} f_i^2f_j^2)^{1/2}}{\norm{A}_F}  \\
  & =  \sup_{ \norm{f}_2  =1}  \Big (\sum_{i\sim j} f_i^2f_j^2\Big)^{1/2} = \sup_{ \norm{w}_1  =1, w_i\geq 0}  \Big (\sum_{i\sim j} w_iw_j\Big)^{1/2},
\end{align*}
where the second line follows by Cauchy-Schwarz inequality and the final equality witnesses the $\ell_2\to \ell_1$ reduction.
 By Lemma \ref{lemma 02}, we have
 \begin{align*}
  \sup_{ \norm{w}_1  =1, w_i\geq 0}  \Big (\sum_{i\sim j} w_iw_j\Big)^{1/2}  \leq \Big(\frac{k-1}{k}\Big)^{1/2},
 \end{align*}
which finishes the proof.
\end{proof}

 We now provide the proof of Lemma \ref{lemma 02}.

\begin{proof}[Proof of Lemma \ref{lemma 02}]
The proof is based on a `mass transportation' argument. 
By homogeneity, it suffices to assume $s=1.$
We first verify \eqref{020} when $G$ is itself a clique of size $m$. In other words, we claim that  if $\sum_{i=1}^m  f_i=1$ and $f_i\geq 0$, then
\begin{align}\label{genbound}
\sum_{1\leq i<j\leq m} f_if_j  \leq \frac{m-1}{2m}.
\end{align}
This follows from the simple equation $2\sum_{i<j} f_if_j=(\sum_i f_i )^2-\sum f_i^2$ and that $\sum_i f_i^2 \ge \frac{1}{m}$ (by Cauchy-Schwarz inequality).

We now prove  \eqref{020} for  the general graphs $G$. 
Assuming that $G$ is not a clique of size $k$, one can  choose two vertices $v_1$ and $v_2$ such that $v_1\not\sim v_2$.   Without loss of generality,  we assume $\sum_{i\sim v_1}  f_i \geq \sum_{j\sim v_2}  f_j$. This allows us to transport mass from $v_2$ to $v_1$ without decreasing the objective function. Namely,
since
\begin{align*}
\sum f_if_j =\Big ( \sum_{i \sim v_1} f_i \Big  ) f_{v_1} + \Big ( \sum_{j \sim v_2} f_j  \Big  ) f_{v_2}  + \sum_{ i,j\neq v_1,v_2, i\sim j } f_if_j
\end{align*}
is linear in $f_{v_1}$ and $f_{v_2}$,
$f$ does not decrease when $f=(\cdots, f_{v_1},\cdots, f_{v_2},\cdots,)$ is replaced by $ f^{(1)}=( \cdots, f_{v_1}+f_{v_2}, \cdots, 0,\cdots ) $. After removing the zero at $v_2$, we obtain a new vector $\tilde{f}  ^{(1)}$ on the new graph $G_1$ obtained by deletion of the vertex $v_2$ and the edges incident on it. 

We repeat this procedure to get a series of vectors $  \tilde{f}  ^{(1)},\cdots,\tilde{f}  ^{(\ell)}$ and graphs $G_1,\cdots,G_\ell$ such that $G_{i+1}$ is  obtained by deletion of some vertex $w_{i+1}$ and edges incident on $w_{i+1}$ in the graph $G_i$.    This procedure is finished once every pair of vertices in $G_\ell$  are connected, i.e. $G_\ell$ is a clique of size $m\leq k$. This along with \eqref{genbound} finishes the proof.
\end{proof}

We end this section with a related short technical lemma which we will need later. The reader can choose to ignore this for the moment and only come back to it when it is later used. 

\begin{lemma} \label{lemma 03}
Suppose that $G$ is a tree with a vertex set $[n]$ and  $s,\eta$ are positive numbers.  Let $v = (v_1,\cdots,v_n)$ be a vector with $\sum_i v_i = s$ and $0\leq v_i\leq \eta$. Then,
\begin{align} \label{030}
\sum_{i<j, i\sim j} v_iv_j \leq 
\begin{cases}
\frac{1}{4}s^2  & s< 2\eta, \\
\eta  (s-\eta)  &s\geq 2\eta.
\end{cases}
\end{align}
\end{lemma}

\begin{proof}
{Let $\rho=\argmax_{i} v_i$. Now think of the tree as rooted at $\rho$ and orient every edge towards $\rho$. Thus $\sum_{i<j, i\sim j}v_i v_j\le \sum_{i\neq \rho} v_{\rho}v_i=v_{\rho}(s-v_{\rho}).$ Now since the function $x(s-x)$ is monotonically increasing in $x$ for $x\le s/2$ and since $v_{\rho}\le \eta,$ \eqref{030} follows. } 
\end{proof}

\section{Upper tail large deviations: lower bound} \label{section 3}
 To begin with, we state a well known estimate for the tail behavior of the maximum of Gaussian random variables which is a straightforward consequence of the following classical bound (We provide the proofs in the appendix.):
For the standard Gaussian random variable $X$, for any $t>0$, 
\begin{align} \label{tail}
 \frac{1}{\sqrt{2\pi}}  \frac{t}{t^2+1} e^{-t^2/2} \leq  
\mathbb{P}(X>t)  \leq   \frac{1}{\sqrt{2\pi}}  \frac{1}{t}
 e^{-t^2/2}
 \end{align}
(see \cite[Equation (A.1)]{chatterjee}).

\begin{lemma} \label{max gaussian}
Let $X_1,\cdots,X_m$ be i.i.d. standard Gaussian random variables and $m \geq cn$ for  some constant $c>0$. Then,   there exists a constant $c' = c'(c)>0,$ such that  for any $\delta>0$,
\begin{align}
\mathbb{P}( \max_{i=1,\cdots,m} X_i\geq   \sqrt{2(1+\delta) \log n} ) \geq  \frac{c'}{\sqrt{\log n}} \frac{1}{n^\delta}
\end{align}
and
\begin{align}
\mathbb{P}( \max_{i=1,\cdots,m} X_i\leq   \sqrt{2(1-\delta) \log n} ) \leq  e^{-c' \frac{n^\delta}{\sqrt{\log n}}}.
\end{align}
\end{lemma}

As indicated in Section \ref{s:iop}, we first show that the number of non-zero elements of the matrix $Z$ is at least of order $n$ with high probability.  Recall that for us $p=\frac{d}{n}$ in $\cG_{n,p}$ throughout the article and the number of non-zero elements in $X$ is twice the same as the number of edges in the underlying random graph $G$.
Let us define an event
\begin{align}\label{edges}
E_0 := \Big\{ | \{  1\leq i<j\leq n : X_{ij} \neq 0\} | >  \frac{d}{16} n \Big\}.
\end{align} 

\begin{lemma} \label{lemma non zero}
There exists a constant $c>0$ such that for sufficiently large $n$,
\begin{align*}
\mathbb{P}( E_0^c ) \leq  e^{-cn}.
\end{align*}
\end{lemma}
This follows from standard large deviation estimates and we include the proof in the appendix for completeness.

\begin{proof}[Proof of Theorem \ref{main theorem 1}: lower bound]
As indicated in Section \ref{s:iop}, there is a slight distinction between $k=2$, and $k\ge 3,$ i.e. the lower bound is governed by two related but distinct events, a large value realized on an edge, or existence of a clique of size at least $3$ with the Gaussians \emph{uniformly}  large on the edges in the clique. \\

\noindent
\textbf{Single large value:}
We first deal with the former case and prove
\begin{align} \label{115}
\limsup_{n\ri}   -\frac{1}{\log n} \log  \mathbb{P}(\lambda_1 \geq \sqrt{2(1+\delta) \log n}  ) \leq \delta.
\end{align}
Since the matrix $Z$ is Hermitian,
\begin{align} \label{bound by max}
\lambda_1 \geq \max_{1\leq i<j\leq n} Z_{ij}.
\end{align}
Thus, 
\begin{align} \label{113}
\mathbb{P}(\lambda_1 \geq \sqrt{2(1+\delta) \log n} )  & \geq    \mathbb{P}( \max_{1\leq i<j\leq n} Z_{ij} \geq \sqrt{2(1+\delta) \log n}  ) \nonumber \\
&\geq  \mathbb{E} \left(\mathbb{P}( \max_{1\leq i<j\leq n} Z_{ij} \geq \sqrt{2(1+\delta) \log n}  \mid X ) \1_{E_0}\right).
\end{align}
By Lemma \ref{max gaussian}, on the event $E_0$, 
\begin{align} \label{114}
\mathbb{P}( \max_{1\leq i<j\leq n} Z_{ij} \geq \sqrt{2(1+\delta) \log n}   \mid X ) \geq C\frac{1}{\sqrt{\log n}} \frac{1}{n^\delta}.
\end{align}
Thus, by \eqref{113}, \eqref{114} and Lemma \ref{lemma non zero}, we obtain \eqref{115}.\\

\noindent
\textbf{Clique construction:}
We now move on to the clique construction. To this end, fix a positive integer $m$ and let $G$ be a network on the clique of size $m$, $K_m$, whose conductances $\{Y_{ij}: 1\leq i<j\leq m\}$ are  i.i.d. standard Gaussians. We denote by {$\lambda(Y)$}  the largest eigenvalue of the adjacency/conductance matrix $Y = (Y_{ij})$ of the network.

By \eqref{tail}, for some constant $C = C(\delta)>0$,
\begin{align} \label{116}
\mathbb{P}( \lambda(Y) \geq  \sqrt{2(1+\delta) \log n} ) &\geq \mathbb{P}\Big(Y_{ij} \geq  \frac{1}{k-1}\sqrt{2(1+\delta)\log n},\ \forall 1\leq i<j\leq k \Big) \nonumber  \\
&\geq {        \left(\frac{C}{\sqrt{\log n}}  n^{-\frac{1+\delta }{(k-1)^2}}\right)^{{k \choose 2}}      }.
\end{align}
Next, we need an estimate of the probability that  a graph contains  a clique of size $k$. This is provided in the next lemma which along with 
 \eqref{116} imply that for any $k\geq 3$,
\begin{align}  \label{118}
\mathbb{P}(\lambda_1 \geq  \sqrt{2(1+\delta) \log n} ) &  \geq  C n^{  - {k \choose 2} +k}   \left(\frac{C}{\sqrt{\log n}}  n^{-\frac{1+\delta }{(k-1)^2}}\right)^{{k \choose 2}} \nonumber \\
&=             { C \left(\frac{C}{\sqrt{\log n}}  \right)^{{k \choose 2}}    n^{ -  \frac{k(k-3)}{2} - \frac{1+\delta}{2} \frac{k}{k-1}}    }          =  n^{-\phi_\delta(k)+o(1)}.
\end{align}
{ 
Since $\phi_\delta(2) = \delta$,  putting  \eqref{115} and \eqref{118} together,
we are done.
}
\end{proof}

\begin{lemma} \label{lemma 31}
Let $k\geq 3$ be a positive integer. Then, there exists a constant $C = C(k,d)>0$ such that
the probability that $\cG_{n,\frac{d}{n}}$ contains a clique of size $k$ is up to universal constants 
\begin{align}  \label{117}
\frac{1}{n^{ {k \choose 2} - k} } .
\end{align}
\end{lemma}

\begin{proof}
Note that the expected number of cliques is indeed up to constants $\frac{1}{n^{ {k \choose 2} - k}}.$ which implies the upper bound. Thus to lower bound the probability of existence of at least one clique we use the familiar second moment method.
{However as has been used several times in the probabilistic combinatorics literature (see e.g., \cite[Theorem 2.3]{random}), to control the second moment, it will be useful to work with the number of  cliques which are also their respective connected components.} To this end, let us denote their number by $N_k.$  Then,
\begin{align}
\mathbb{E} N_k =  {n \choose k} p ^{ {k\choose 2}}   (1-p)^{k(n-k)}\geq \frac{e^{k-1}d^{ {k \choose 2} } }{k^{k+\frac{1}{2}}} \Big(1-\frac{k}{n}\Big)^k\Big (1-\frac{d}{n}\Big)^{k(n-k)}  \frac{1}{n^{ {k \choose 2} - k}}\geq  C \frac{1}{n^{ {k \choose 2} - k}}  .
\end{align}
where above we use Stirling's formula to approximate $k!$ and we use the bound $n!/(n-k)!\ge (n-k)^k.$
Further,
\begin{align}
\mathbb{E} N_k^2 =  \mathbb{E} N_k + {n \choose k} {n-k \choose k}p ^{ 2{k\choose 2}}   (1-p)^{k^2 + 2k(n-2k)} \leq \mathbb{E}N_k + (1-p)^{-k^2} (\mathbb{E}N_k)^2.
\end{align}
Thus, by the Paley-Zygmund inequality, for sufficiently large $n$,
\begin{align}  
\mathbb{P}(N_k\geq 1) \geq \frac{(\mathbb{E}N_k)^2 }{\mathbb{E}N_k^2} \geq  \frac{1}{ (\mathbb{E}N_k)^{-1} + (1-p)^{-k^2}} \geq  \frac{1}{C n^{ {k \choose 2} - k} +2} .
\end{align}

\end{proof}


\section{Upper tail large deviations: upper bound} \label{section 3.2}
A significant fraction of the novel ideas in the paper can be found in this section which aims to implement the high level strategy outlined in Section \ref{s:iop}. Before beginning, we include a short roadmap to indicate what the different subsections achieve. 
In subsection \ref{subsec5.1} we record tail estimates for sums of squares of Gaussian variables conditioned to be large. In subsection \ref{subsec5.2} we show that with high probability the network $Z^{(2)}$ from Section \ref{s:iop} is spectrally negligible.  We then move on to analyzing the connectivity structure of the graph $X^{(1)} $ underlying the network $Z^{(1)}$, including its maximum degree, size of its connected components and the number of tree excess edges they contain in subsection \ref{subsec5.3}. In subection \ref{subsec5.4} we prove a key proposition (Proposition \ref{prop}) establishing tails for the largest eigenvalue for tree like networks in terms of the largest clique. Finally in subection \ref{subsec5.5}, we prove the upper bound in Theorem \ref{main theorem 1}.

\subsection{Chi-square tail estimates:}\label{subsec5.1} We record the following estimate that will be crucial in our applications whose proof is provided in the appendix.
\begin{lemma} \label{chi tail}
Let $\tilde{Y}$ be a standard Gaussian conditioned on $|\tilde{Y}|>\sqrt{\varepsilon \log \log  n}$, and
 denote $ \tilde{Y}_1,\cdots,\tilde{Y}_m$ by independent copies of $\tilde Y$.
Then, there exists a universal constant $C>0$ such that  for any $L>m$ and $\e>0$,
\begin{align} \label{211}
\mathbb{P}(\tilde{Y}_1^2+\cdots+\tilde{Y}_m^2 \geq   L )  \leq  C^m e^{-\frac{1}{2}L} e^{\frac{1}{2}m} \Big(\frac{L}{m}\Big)^{m}  e^{\frac{1}{2} \varepsilon m \log \log n} .
\end{align}
{
In particular, for any  $a,b,c>0$, let $ m \leq b\frac{\log n}{\log \log n}+c$ and $L = a\log n$. Then, for any $\gamma>0$,  for sufficiently large $n$,
\begin{align} \label{212}
\mathbb{P}(\tilde{Y}_1^2+\cdots+\tilde{Y}_m^2 \geq   a\log n )  \leq    n^{-\frac{a}{2} + \frac{\varepsilon b}{2} + \gamma} .
\end{align}
}
\end{lemma}

Recall from Section \ref{s:iop}, the decompositions
\begin{align*} 
Y_{ij} = Y^{(1)}_{ij} + Y^{(2)}_{ij},
\end{align*}
where $Y^{(1)}_{ij}=  Y_{ij}\1_{|Y_{ij}| > \sqrt{\varepsilon \log \log n}}$ and similarly $Y^{(2)}_{ij}=  Y_{ij}\1_{|Y_{ij}| \leq \sqrt{\varepsilon \log \log n}}.$
Thus, we can write the matrix $Z$ as $Z^{(1)}+Z^{(2)}$ with 
\begin{align}\label{decomposition1}
Z^{(1)}_{ij} = X_{ij} Y^{(1)}_{ij}, \quad  Z^{(2)}_{ij} = X_{ij} Y^{(2)}_{ij}.
\end{align}

\subsection{Spectrally negligible component}\label{subsec5.2}
We next prove an upper bound on the probability that $Z^{(2)}$ has high spectral norm.

\begin{lemma} \label{lemma 32}
For $\delta>0$,
\begin{align*} 
 \lim_{n\ri} \frac{-\log \mathbb{P}( \lambda_1 (Z^{(2)})   \geq  \sqrt{\varepsilon} (1+\delta) \sqrt{\log n}    )  }{\log n} \geq  2 \delta+ \delta^2.
\end{align*}

\end{lemma}

\begin{proof}
The proof relies on the results of the previously mentioned recent work \cite{ganguly1}.
By \cite[Theorem 1.1]{ganguly1},
\begin{align*} 
 \lim_{n\ri} \frac{-\log \mathbb{P}(  \lambda_1(X)  \geq  (1+\delta) \sqrt{\frac{\log n}{\log \log n}} )  }{\log n} = 2 \delta+ \delta^2. 
\end{align*}
Since 
$|Z^{(2)}_{ij}| \leq  X_{ij}  \sqrt{\varepsilon \log \log n},$ we have $ \lambda_1 (Z^{(2)})  \leq   \sqrt{\varepsilon \log \log n} \cdot  \lambda_1(X) $ which concludes the proof.
\end{proof}

\subsection{Connectivity structure of highly sub-critical Erd\H{o}s-R\'enyi graphs}\label{subsec5.3}
We will now shift our focus to $Z^{(1)}.$ Recall that $X^{(1)}_{ij}=X_{ij}\1_{|Y_{ij}| > \sqrt{\varepsilon \log \log n}}.$ By the tail bound for Gaussian stated in \eqref{tail}, {for large $n$,} $X^{(1)}$ is distributed as $\cG_{n, q}$ with
\begin{align}\label{newdensity}
q \leq  \frac{d}{n} \frac{1}{\sqrt{2\pi}} e^{-\frac{1}{2}\varepsilon \log \log n}  = \frac{d'}{n}\frac{1}{(\log n)^{\varepsilon/2}},
\end{align}
where $d' = \frac{d}{\sqrt{2\pi}}$.

 For any graph $G$, we denote by $d_1(G)$, the largest degree of  $G$. It is proved in \cite{KS} (see also \cite[Proposition 1.3]{ganguly1}) that the typical value of $d_1(\cG_{n,r})$  is $\frac{\log n}{\log \log n - \log (n r)}$, when
 \begin{align*}
 \log n  \gg \log (1/nr) \quad \text{and} \quad nr    \ll   \sqrt{\frac{\log n }{\log \log n}}.
 \end{align*}
Furthermore, the following large deviation result is a consequence of \cite[Proposition 1.3]{ganguly1}. 
\begin{lemma} 
 \label{lemma 33}
For $\delta_1>0$,
let  $\cD_{\delta_1}$ be an event defined by
\begin{align}\label{degsize}
\cD_{\delta_1} :=  \Big\{d_1{(X^{(1)})} \leq  (1+\delta_1) \frac{\log n}{\log \log n}\Big\}.
\end{align}
Then,
\begin{align*}
 \lim_{n\ri} \frac{-\log \mathbb{P}(  \cD_{\delta_1}^c  )  }{\log n}  \ge \delta_1.
\end{align*}
\end{lemma}
\begin{proof}
The statement, {where the inequality above is replaced with an equality,} for the case $r=\frac{d}{n}$ is obtained in \cite[Proposition 1.3]{ganguly1}, by plugging in $r=\frac{d}{n}$ in the latter and noting that in this case $$\frac{\log n}{\log \log n - \log (n r)}=\frac{\log n}{\log \log n - \log d}.$$ The above result then follows by observing that $\cG_{n,\frac{d}{n}}$ stochastically dominates $\cG_{n,q}$ and $d_{1}(G)$ is an increasing function of the graph. 
\end{proof}

We next move on to a refined analysis of the connectivity structure of the graph $X^{(1)}$. 
Towards this, let $C_1,\cdots,C_m$ be its connected components. The next lemma establishes a bound of the order of $\frac{\log n}{\log\log n}$ on the size of the largest component in contrast to the bounds of $\Theta(\log n)$, $\Theta(n^{2/3}),$  or $\Theta(n)$, that one has for $\cG_{n,\frac{d}{n}}$ depending on if $d<1, d=1$ or $d>1.$ This sub-logarithmic bound will be crucial in our application and justifies our sparsification step. 

\begin{lemma} \label{lemma 34}
For $\delta_2>0$,
let  $\cC_{\delta_2}$ be the following event.  \begin{align}\label{connsize}
\cC_{\delta_2}:=\Big \{ |C_i| \leq   \frac{2+\delta_2}{\varepsilon}  \frac{\log n}{\log \log n},\ \forall i\Big\}.
\end{align}
Then, 
\begin{align*}
 \liminf_{n\ri} \frac{-\log \mathbb{P}(  \cC_{\delta_2}^c  )  }{\log n}  \geq \frac{\delta_2}{2}.
\end{align*}
\end{lemma}

\begin{proof} 
 The proof implements the standard first moment argument, (see e.g.,  \cite[Chapter 5,6]{bollobas}).
Let  $\bar{N}_{k,-1}$ be the number of \it{connected} subgraphs having $k$ vertices and $k-1$ edges, in other words the number of trees of size $k$. Using  \eqref{newdensity} and Stirling's formula, and the fact that  the number of labelled spanning trees on $k$ vertices is $k^{k+2}$, for some large constant $c_0>0$,
\begin{align} \label{343}
 \mathbb{E}\bar{N}_{k,-1} &\leq   {n \choose k} k^{k+2} \Big(\frac{d'}{n}\frac{1}{(\log n)^{\varepsilon/2}}\Big)^{k-1}   \nonumber  \\
 &  \leq  C   e^k \frac{n^k}{k^k} k^{k+2}  \frac{(d')^{k-1}}{n^{k-1} (\log n)^{\frac{\varepsilon}{2} (k-1)} }  =  C  n    \frac{ e^{k}k^2(d')^{k-1}}{(\log n)^{\frac{\varepsilon}{2} (k-1)} } \leq  {C   n (\log n)^{\varepsilon/2} \Big(\frac{c_0}{(\log n)^{\varepsilon/2}}\Big)^k}.
\end{align}
Hence, denoting $N_k$ by the number of connected components with $k$ vertices, picking a spanning tree from each connected component, one obtains 
\begin{align*}
\mathbb{E} N_k  &\leq  \mathbb{E}\bar{N}_{k,-1}  
\leq   C n (\log n)^{\varepsilon/2} \Big(\frac{c_0}{(\log n)^{\varepsilon/2}}\Big)^k.
\end{align*}
Define $m:=\frac{2+\delta_2}{\varepsilon} \frac{\log n}{\log \log n}$, and let $N$ be the number of connected components  having  at least $m$ vertices. Then,
\begin{align*}
\mathbb{E} N = 
\mathbb{E} \sum_{k=m}^n N_k     \leq C n(\log n)^{\varepsilon/2} \Big(\frac{c_0}{(\log n)^{\varepsilon/2} }\Big)^m   \leq  C (\log n)^{\varepsilon/2} {     n^{ \frac{  ( \log c_0 ) (2+\delta_2)}{\varepsilon \log \log n}}       } n^{- \frac{\delta_2}{2} } .
\end{align*}
Since $\P(N\ge 1)\le \E(N)$, the proof is complete.
\end{proof}

For our applications, we will also need to bound the number of subgraphs having $k$ vertices and $k+\ell $ edges without the subgraph necessarily being connected. This estimate  will be  crucially used later to prove the structure theorem conditioned on $\cU_\delta$.

\begin{lemma}\label{subgraphcount}For $\ell\geq 0$,
 let $N_{k,\ell}$ be the number of  subgraphs in $X^{(1)}$ having  $k$ vertices and $k+\ell$ edges. Then, for  $0\leq \ell \leq {k\choose 2}-k$,
$$ \mathbb{E}N_{k,\ell} \le C\min \left(\Big(\frac{k}{n}\Big)^\ell, \Big(\frac{d'e^2}{(\log n)^{\varepsilon/2}} \Big)^{k+\ell} \right).$$
\end{lemma}

\begin{proof}
Denote by $C_{k,\ell}$ the number of  labelled   graphs with $k$ vertices and $k+\ell$ edges. Then, for any $-k\leq \ell \leq {k\choose 2}-k$,
using Stirling's formula
we have
\begin{align}
C_{k,\ell}  =  { {k \choose 2} \choose k+\ell} \leq {k^2 \choose k+\ell} \leq \frac{(k^2)^{k+\ell}}{(k+\ell)!} \leq  e^{k+\ell} \frac{k^{2(k+\ell)}}{(k+\ell)^{k+\ell}}.
\end{align}

Then, for  $0\leq \ell \leq {k\choose 2}-k$,
\begin{align} \label{342}
 \mathbb{E}N_{k,\ell} & \leq   {n \choose k}  C_{k,\ell} q^{k+\ell}
 \overset{\eqref{newdensity}}{\leq}   C  e^{2k+\ell} \frac{n^k}{k^k}    \frac{k^{2(k+\ell)}}{(k+\ell)^{k+\ell}}  \Big( \frac{d'}{n}\frac{1}{(\log n)^{\varepsilon/2}}\Big)^{k+\ell}  \leq C \Big (\frac{k}{n}\Big)^\ell   \Big(\frac{d'e^2}{(\log n)^{\varepsilon/2}} \Big)^{k+\ell} .   
\end{align}
where in the first inequality we use Stirling's formula again to bound $k!.$
In particular, since $k\leq n$,  
\begin{align} \label{345}
\mathbb{E}N_{k,\ell} \leq  C\Big(\frac{d'e^2}{(\log n)^{\varepsilon/2}} \Big)^{k+\ell},
\end{align}
and since $d'e^2 \leq  (\log n)^{\varepsilon/2}$  for sufficiently large $n$,
\begin{align} \label{346}
 \mathbb{E}N_{k,\ell} \leq C\Big (\frac{k}{n}\Big)^\ell.
\end{align}
\end{proof}

Having bounded the maximal component size, we next proceed to estimating how close the components are to trees by bounding the number of tree excess edges, i.e., how many edges need to be removed from such a component to obtain a tree.

\begin{lemma} \label{lemma 35}
For   $\delta_3 \geq 1$, let $\cE_{\delta_3}$ be the event defined by
\begin{align}\label{treeexcess}
\cE_{\delta_3}:= \{ |E(C_i)| < |V(C_i)| + \delta_3, \ \forall i\}.
\end{align}
Then,
\begin{align} \label{350}
 \liminf_{n\ri} \frac{-\log \mathbb{P}(  \cE_{\delta_3}^c  )  }{\log n}  \geq \delta_3.
\end{align}
In addition, define the event $\cT$ by
\begin{align*}
\cT:=\{ |E(C_i)| = |V(C_i)| -1,  \ \forall i\}.
\end{align*}
In other words, $\cT$ is the event that  all the connected components of $X^{(1)}$ are trees. Then,  
\begin{align} \label{351}
\mathbb{P}(\cT^c) \leq  \frac{C}{(\log n)^\varepsilon}.
\end{align}
\end{lemma}
\begin{proof}
For  $  \ell\geq 0$, recall the notation $N_{k,\ell}$ from Lemma \ref{subgraphcount}. 
Since the occurrence of the event $\cE_{\delta_3}^c\cap \cC_{2\delta_3}$ demands the existence of a connected component $C_i$ with $|C_i| \leq   \left \lfloor {  \frac{2+2\delta_3}{\varepsilon}\frac{\log n}{\log \log n} }   \right  \rfloor       =:m$ and  $ |E(C_i)| \geq |C_i| +\left\lceil \delta_3 \right\rceil$,  by the first moment bound,
\begin{align} \label{353}
\mathbb{P}(\cE_{\delta_3}^c\cap \cC_{2\delta_3}) \leq    \sum_{k=3}^m \sum_{\ell=\left\lceil \delta_3 \right\rceil }^{{k\choose 2}-k} \mathbb{E}N_{k,\ell}  \overset{\eqref{346}}{\leq}   C \sum_{k=3}^m    \Big (\frac{k}{n}\Big)^{\left\lceil \delta_3 \right\rceil} \leq C \frac{m^{\left\lceil \delta_3 \right\rceil+1}}{n^{\left\lceil \delta_3 \right\rceil}}.
\end{align}
Therefore, by  \eqref{353} and Lemma \ref{lemma 34} (with $\delta_2=2\delta_3$),  we obtain \eqref{350}.

Next, we prove \eqref{351}.
Let $N_{\text{cycle}}$ be the number of cycles in $X^{(1)}$. Then,
\begin{align*}
\mathbb{E}N_{\text{cycle}} = \sum_{k=3}^n  {n \choose k}  \frac{(k-1)!}{2} q^k  \leq \sum_{k=3}^n \frac{n^k}{2k} \Big( \frac{d'}{n}\frac{1}{(\log n)^{\varepsilon/2}}\Big)^k \leq \frac{C}{(\log n)^\varepsilon}.
\end{align*}
Since the occurrence of $\cT^c$  implies the existence of cycle,
by the first moment bound, we obtain \eqref{351}.
\end{proof}
 
 \subsection{Spectral tail for tree like networks.}\label{subsec5.4}
We have so far defined the events $\cD_\alpha$, $\cC_\alpha$, $\cE_\alpha$, $\cT$,  and 
in the previous series of lemmas, having established that each connected component is of size $O(\frac{\log n}{\log \log n})$ and the number of excess edges is bounded with high probability, in the following key proposition,
we control the spectral norm of such a connected component. This will be a particularly important ingredient in the proof of Theorem \ref{main theorem 1}. 

\begin{proposition} \label{prop}
Consider a connected network $G = (V,E,A)$ (where $A=(a_{ij})$ is the matrix of conductances) satisfying the following properties:
\begin{enumerate}
\item $d_1(G) \leq c_1 \frac{\log n}{\log \log n}$
\item  $|V| \leq c_2\frac{\log n}{\log \log n}$
\item $|E| \leq |V|+c_3$
\end{enumerate}
Suppose that the conductance matrix $A$ is given by i.i.d. Gaussians associated to each element of $E$, conditioned on having absolute value greater than $\sqrt{\varepsilon \log \log n}$.
Let $k$ be a maximal size of clique in $G$ { and $\lambda$ be the largest eigenvalue of $A$.
  Then,  for any ${\e},\alpha,\gamma,\eta>0$ with $\eta<\frac12$,  for sufficiently large $n$,
\begin{align}\label{keydisp}
\mathbb{P}( \lambda \geq   \sqrt{2\alpha \log n})  \leq    n^{-\frac{\alpha}{2\theta^{2}  }+ \frac{\varepsilon c_2}{2} + \gamma}   +  n^{-\frac{k}{2(k-1)} (1-\theta)^2    \alpha + \frac{c_1 \varepsilon}{2\eta^2} + \gamma},
\end{align}
where $\theta :=   (2 \eta^2 +  2\eta^4 c_3)^{1/4}.$
}
\end{proposition}

{
The expression on the right hand side is technical  but the constants $\varepsilon,\eta,\gamma$ will be suitably chosen sufficiently close to zero so that   $n^{-\frac{\alpha}{2\theta^2}+ \frac{\varepsilon c_2}{2} + \gamma}  $  and $n^{  \frac{c_1 \varepsilon}{2\eta^2} + \gamma} $ are negligible and the dominant behavior will be  $n^{-\frac{k}{2(k-1)} \alpha}.$}

{From now on, for any graph $H$, we denote by $E(H)$ and $\overrightarrow{E(H)}$ the sets of undirected and  directed edges in $H$ respectively. }
\begin{proof} The proof proceeds by analyzing the leading eigenvector.
Let $V = [\ell]$ and $f=(f_1,\cdots,f_\ell)$ be the unit (random) eigenvector associated with the largest eigenvalue $\lambda:=\lambda_{1}(G)$. 
Thus by definition, $\lambda=f^{\top}Af.$

One would have liked  to use Proposition \ref{spectral bound}  and the tail estimate \eqref{212}. However the application of the latter is useful only when the parameter $b$ in the upper bound of $m$ is small enough compared to $\frac{1}{\e}.$ On the other hand, in Lemma \ref{lemma 34}, the bound on $|C_i|$ which would be $m$ in the application is $O(\frac{1}{\e} \frac{\log n}{\log\log n})$ rendering the above straightforward strategy useless. 
To address this, the first step is to argue that entries of $f$ that are small in absolute value do not contribute much to the above quadratic form. This allows us to focus on only the large entries, of which there are not too many and hence allows an application of the above outlined strategy with a reduced value of $m.$ 
Towards this, 
for $0<\eta<1/2$, define the collection of vertices
\begin{align*}
I := \{i \in [\ell]: f_i^2  <  \eta^2\}.
\end{align*} 
Let $B_1$ be the collection of  {(directed)}  edges defined by
\begin{align*}
B_1:= \{ (i,j) \in \overrightarrow{E}: i,j\in I\},
\end{align*} 
and let $ B_2 := \overrightarrow{E} \setminus B_1$ where again each edge is considered twice {(this is done simply as a matter of convention)} Now since $f$ is a unit vector, by Markov's inequality, $|I^c| \leq \frac{1}{\eta^2}$. In addition, by the upper bound on the max-degree in condition (1), we obtain 
\begin{align} \label{362}
{ \frac{1}{2} } |B_2| \leq  \frac{c_1}{\eta^2} \frac{\log n}{\log \log n}.
\end{align}
We write
\begin{align*}
\lambda = \sum_{(i,j)\in \overrightarrow{E}}  a_{ij} f_if_j  =  \sum_{(i,j)\in B_1}  a_{ij} f_if_j  +  \sum_{(i,j)\in B_2}  a_{ij} f_if_j  =: S_1+S_2.
\end{align*}
Recall $\theta =   (2 \eta^2 +  2\eta^4 c_3)^{1/4} $,  we have
\begin{align} \label{363}
\mathbb{P}& ( \lambda \geq  \sqrt{2\alpha \log n})   \leq   \mathbb{P}(S_1\geq  \theta \sqrt{2\alpha \log n})  + \mathbb{P}( S_2\geq  (1-  \theta ) \sqrt{2\alpha \log n}) .
\end{align}
Of course, the above inequality holds for any $\theta$ and the particular choice we make is guided by our subsequent estimates of $S_1$ and $S_2.$
First, we show that
\begin{align} \label{364}
\sum_{(i,j)\in B_1} f_i^2f_j^2 \leq  2 \eta^2 +  2\eta^4 c_3 = \theta^4.
\end{align}
We will rely on Lemma \ref{lemma 03}.
Choose a spanning tree $T$ of  $G$, and define {a set of (directed) edges} $\overrightarrow{E'} := \overrightarrow{E(G)}\setminus \overrightarrow{E(T)}$. Then, by condition (3) on the number of excess edges, $\frac{1}{2}|E'| \leq c_3+1$. Now  the graph with edge set $B_1 \setminus \overrightarrow{E'}$ is necessarily a forest. Since adding more edges can only increase $\sum_{(i,j)\in B_1 \setminus \overrightarrow{E'}}f_i^2f_j^2$ we can in fact assume that  the graph with edge set $B_1 \setminus \overrightarrow{E'}$    is a tree.
Now applying Lemma \ref{lemma 03} with $s=1,$ and since $2\eta^2\le 1$, we conclude that 
\begin{equation}\label{smallvalue1}
\sum_{(i,j)\in B_1 \setminus \overrightarrow{E'}}f_i^2f_j^2\le 2\eta^2(1-\eta^2).
\end{equation}
 Hence,
\begin{align}\label{smallvalue2}
\sum_{(i,j)\in B_1 }  f_i^2f_j^2    = 
\sum_{(i,j)\in B_1 \setminus \overrightarrow{E'}}  f_i^2f_j^2  + \sum_{(i,j)\in B_1\cap   \overrightarrow{E'}}  f_i^2f_j^2    \leq 2 \eta^2(1-\eta^2 ) + 2 (c_3+1) \eta^4,
\end{align}
where for the second term, we simply use the fact that the total number of summands is at most $2(c_3+1)$ with each being at most $\eta^4.$
This  proves \eqref{364}.
Hence by the definition of $S_1$, by Cauchy-Schwarz inequality,  we immediately have 
$$
S_1  \leq  \Big(\sum _{(i,j)\in B_1} a_{ij}^2\Big)^{1/2} \Big(\sum_{(i,j)\in B_1}  f_i^2f_j^2\Big)^{1/2}    < \theta^2 \Big (\sum _{(i,j)\in B_1} a_{ij}^2\Big)^{1/2} \leq  \theta^2 \Big (\sum _{(i,j)\in E} a_{ij}^2\Big)^{1/2}.$$ Thus, for any $\gamma>0$,   for sufficiently large $n$,
\begin{align} \label{365}
\mathbb{P}( S_1 \geq \theta \sqrt{2\alpha \log n})&\leq \mathbb{P}\Big( \sum _{i<j, (i,j)\in E} a_{ij}^2 \geq   \frac{\alpha}{  \theta^2}   \log n\Big)  \leq  n^{-\frac{\alpha}{2 \theta^2  }+ \frac{\varepsilon c_2}{2} + \gamma} ,
\end{align}
where the last inequality follows by a direct application of  \eqref{212}   in  Lemma \ref{chi tail}, with $L=\frac{\alpha}{  \theta^2}   \log n$ and $M=c_2\frac{\log n}{\log \log n}+c_3.$

Next, we estimate $S_2$. Since, by hypothesis, the maximal size of  clique in the subgraph induced by edges in $B_2$ is no larger than $k$, by  Lemma \ref{lemma 02},     
\begin{align} \label{366}
S_2 \leq \Big(\sum _{(i,j)\in B_2} a_{ij}^2\Big)^{1/2} \Big(\sum_{(i,j)\in B_2}  f_i^2f_j^2\Big)^{1/2}\leq  \Big (\frac{k-1}{k}\Big)^{1/2} \Big(\sum _{(i,j)\in B_2} a_{ij}^2\Big)^{1/2}.
\end{align} 
Note that the event  $ \sum _{i<j (i,j)\in B_2} a_{ij}^2 \geq  t$ implies the existence of a {random} subset $J \in [n]$ with $|J| \leq  \left \lfloor {  \frac{1}{\eta^2} }   \right  \rfloor   $ such that  $ \displaystyle{\sum _{i\text{ or }j\in J ,  i<j, i\sim j} a_{ij}^2 \geq  t}$.
Hence,  for any $\gamma>0$, for sufficiently large $n$,
\begin{align} \label{367}
\mathbb{P}( S_2 \geq  (1-\theta ) \sqrt{2 \alpha \log n})  & \overset{\eqref{366}}{\leq}   
\mathbb{P} \Big(  \sum _{i< j,\, (i,j)\in B_2} a_{ij}^2 \geq   \frac{k}{k-1} (1 -  \theta )^2    \alpha  \log n\Big) \nonumber \\
 &\leq   |V|^{\left \lfloor {  \frac{1}{\eta^2} }   \right  \rfloor }  n^{-\frac{k}{2(k-1)} (1-\theta )^2    \alpha + \frac{c_1 \varepsilon}{2\eta^2} + \frac{\gamma}{2}}     \leq n^{-\frac{k}{2(k-1)} (1-\theta )^2    \alpha + \frac{c_1 \varepsilon}{2\eta^2} + \gamma}   .
\end{align}
{
The second inequality is obtained by a simple first moment bound, in conjunction with \eqref{212} in  Lemma  \ref{chi tail} with  $L=\frac{k}{k-1}(1-\theta)^2\alpha \log n$ and $M \leq    \frac{c_1}{\eta^2}\frac{\log n}{\log \log n} $ (see \eqref{362}). In the last inequality, we used condition (2) i.e.,  $|V|\leq  c_2\frac{\log n}{\log \log n}$, to bound the term   $ |V |^{\left \lfloor {  \frac{1}{\eta^2} }  \right \rfloor}  $  by $n^{\gamma/2}$ for sufficiently large $n$.  
}

Thus,  by \eqref{363}, \eqref{365}  and  \eqref{367}, for sufficiently large $n$,
\begin{align}
\mathbb{P}( \lambda \geq   \sqrt{2\alpha \log n})  \leq    n^{-\frac{\alpha}{2 \theta^2  }+ \frac{\varepsilon c_2}{2} + \gamma}   +  n^{-\frac{k}{2(k-1)} (1-\theta )^2    \alpha + \frac{c_1 \varepsilon}{2\eta^2} + \gamma}
\end{align} 
which finishes the proof.
\end{proof}

With all this preparation, we are now ready to prove the upper bound in Theorem \ref{main theorem 1}.

\subsection{Proof of Theorem \ref{main theorem 1}: upper bound} \label{subsec5.5}
Recall the matrices from \eqref{decomposition1} as well as the matrix $X^{(1)}$ from \eqref{newdensity}.
Let $C_1,\cdots,C_m$ be the connected components of $X^{(1)}$, and define  $\lambda_1(C_i)$ to be the largest eigenvalue of the matrix $Z^{(1)}$ restricted to $C_i$.
Let $\sf{Few-cycles}$ be the event defined by
\begin{align*}
{\sf{Few-cycles}}:=   \{|\{i: C_i  \ \textup{not tree} \}|  <  \log n \}.
\end{align*}

By Lemma \ref{lemma 35} \eqref{351}, the probability of existence of some cycle is  $\frac{C}{(\log n)^\varepsilon}.$
Since the occurrence of the event ${\sf{Few-cycles}}^c$ demands the disjoint occurrence of $\log n$ many cycles, by the above fact and  {Van-den Berg-Kesten  (BK) inequality \cite{bk}},
\begin{align} \label{121}
\mathbb{P}({\sf{Few-cycles}}^c) \leq \frac{C}{(\log n)^{\varepsilon \log n/2}}.
\end{align}
Also, since $\lambda_1(Z)  \leq  \lambda_1 (Z^{(1)}) +  \lambda_1(Z^{(2)})$,  
\begin{align} \label{122}
\mathbb{P}(\lambda_1(Z) &\geq \sqrt{2(1+\delta) \log n} ) \nonumber \\
&\leq \mathbb{P}(\lambda_1(Z^{(1)}) \geq \sqrt{2(1+\delta') \log n} )  + \mathbb{P}(\lambda_1(Z^{(2)}) \geq  \sqrt{\varepsilon}(1+\delta) \sqrt{ \log n} ) ,
\end{align} 
where $\delta'>0$ is defined by 
\begin{equation}\label{delta'}
\sqrt{2(1+\delta')} = \sqrt{2(1+\delta)} - \sqrt{\varepsilon}(1+\delta).
\end{equation}
{
Note that from this (by rearranging and multiplying both sides by $\sqrt{2(1+\delta)} +\sqrt{2(1+\delta')}$), we have  
 \begin{align} \label{610}
 \delta  - \sqrt{2\varepsilon} (1+\delta)^{3/2    }\leq  \delta'  \leq   \delta .
\end{align}
}

Using the result in subsection \ref{subsec5.2}, the second term in \eqref{122} will be negligible, so we focus on estimating the first. 
Recalling $X^{(1)}_{ij}:= X_{ij} \1_{|Y_{ij}| >  \sqrt{\varepsilon \log \log n} } $, 
let us  estimate the conditional probability $\mathbb{P}(\lambda_1(Z^{(1)}) \geq \sqrt{2(1+\delta') \log n}  | X^{(1)})$
on the high probability event  $\cD_{4\delta'}\cap \cC_{4\delta'} \cap \cE_{4\delta'} \cap {\sf{Few-cycles}}$. By definition, on this event, we have
\begin{align} \label{123}
d_1(X^{(1)}) &< (1+4\delta') \frac{\log n}{\log \log n},\\
\label{124}
|V(C_i)| &<  \frac{2+4\delta'}{\varepsilon}  \frac{\log n}{\log \log n},\quad i=1,\cdots,m,\\
 \label{125}
|E(C_i)| &< |V(C_i)| + 4\delta', \quad i=1,\cdots,m,\,\text{ and,}\\
 \label{126}
 |\{i = 1,\cdots,m & : C_i \ \textup{not tree} \}|  <  \log n .
\end{align}

From now one we will denote by $Z^{(1)}_i,$ the matrix $Z^{(1)}$ restricted to $C_i$, and by $k_i$ the size of the largest clique in $C_i$.   By \eqref{123}-\eqref{125} and  Proposition \ref{prop} with $$c_1=1+4\delta',\,c_2=\frac{2+4\delta'}{\e}, \, c_3=4\delta',\,
\alpha = 1+\delta' \text{ and }\eta = \varepsilon^{1/4},$$  setting  {    $\xi :=  (2 \varepsilon^{1/2} +  8 \varepsilon \delta' )^{1/4}$,     }  on the event  $\cD_{4\delta'}\cap \cC_{4\delta'} \cap \cE_{4\delta'} $, for any $\gamma>0$ and  sufficiently small $\varepsilon>0$,
\begin{align} \label{127}
\mathbb{P}(\lambda_1(Z^{(1)}_i) \geq \sqrt{2(1+\delta') \log n} \mid  X^{(1)})    < C n^{-\frac{k_i}{2(k_i-1)} (1- \xi)^2 (1+\delta') + \frac{1+4\delta'}{2}\varepsilon^{1/2} + \gamma},
\end{align}
by observing that for $\e$ small enough, the first term in \eqref{keydisp} is negligible compared to the second term and can be absorbed in the constant $C$. More precisely, using the bound \eqref{610}, one can take sufficiently small   $\varepsilon$ such that
\begin{align}
1+\delta'> 2(2 \varepsilon^{1/2} +  8 \varepsilon \delta' )^{1/2} (1+\delta'+ ( 1+2\delta')).
\end{align}
Then, for $k\geq 2$,  
$
\frac{1+\delta'}{2\xi^2  } - ( 1+2\delta' ) \geq 1+\delta'  \geq  
\frac{k}{2(k-1)} (1-\xi  )^2    (1+\delta') ,
$
which implies that   the first term in \eqref{keydisp} decays faster than the second term.

Define 
\begin{align}
\nonumber
I := \{i = 1,\cdots,m: &\,\, k_i \geq 3\},\quad J:= \{i = 1,\cdots, m: k_i=2\}, \text{ and,}\\
\label{maximalsize}
\bar{k} &:= \max \{k_1,\cdots,k_m\}.
\end{align}
Then, since $\frac{k}{k-1}$ is decreasing in $k,$  by \eqref{127}, under the event  $\cD_{4\delta'}\cap \cC_{4\delta'} \cap \cE_{4\delta'} $,  for any $i\in I$,
\begin{align}  \label{141}
\mathbb{P}(\lambda_1(Z^{(1)}_i) \geq \sqrt{2(1+\delta') \log n} \mid X^{(1)})  < C n^{-\frac{\bar{k}}{2(\bar{k}-1)} (1- \xi)^2 (1+\delta') + \frac{1+4\delta'}{2}\varepsilon^{1/2} +\gamma },
\end{align}
and  for any $i\in J$,
\begin{align}  \label{142}
\mathbb{P}(\lambda_1(Z^{(1)}_i) \geq \sqrt{2(1+\delta') \log n} \mid X^{(1)})   < C n^{-  (1- \xi)^2 (1+\delta') + \frac{1+4\delta'}{2}\varepsilon^{1/2} + \gamma }.
\end{align}
Also, by Lemmas \ref{lemma 33}, \ref{lemma 34}, \ref{lemma 35} and \eqref{121},   defining the event
\begin{align} \label{147}
\cF_0:=\cD_{4\delta'}\cap \cC_{4\delta'} \cap \cE_{4\delta'} \cap {\sf{Few-cycles}},
\end{align}  we have
\begin{align} \label{129}
\mathbb{P}(\cF_0^ c ) \leq \frac{C}{n^{2\delta'}}.
\end{align}
Using \eqref{newdensity}, by the first moment bound, for $k\geq 3$,
\begin{align} \label{144}
\mathbb{P}(X^{(1)}  \ \textup{contains a clique of size} \  k) \leq   {n \choose k} q^{{k \choose 2} } \leq  \frac{(d')^{{k \choose 2}}}{n^{ {k \choose 2} - k} }.
\end{align}
Also, since any connected component $C_i$ which is a tree has $k_i=2$, on the event ${\sf{Few-cycles}}$, we have $|I| < \log n$. Thus,  using \eqref{144} and the fact  $\lambda_1(Z^{(1)}) = \max_{i=1,\cdots,m} \lambda_1(Z^{(1)}_i)$,
\begin{align}  \label{143}
\mathbb{P}&(\lambda_1(Z^{(1)}) \geq \sqrt{2(1+\delta') \log n} )  \nonumber\\
&\leq 
 \sum_{k=3}^n \mathbb{E} \left[ \mathbb{P}(\max_{i\in I} \{ \lambda_1(Z^{(1)}_i) \}\geq \sqrt{2(1+\delta') \log n} \mid X^{(1)}) \1_{\cF_0 } \1_{\bar{k} = k}  \right]  \nonumber\\
 &+\mathbb{E} \left[ \mathbb{P}(\max_{i\in J} \{ \lambda_1(Z^{(1)}_i) \}\geq \sqrt{2(1+\delta') \log n} \mid X^{(1)}   ) \1_{\cD_{4\delta'}\cap \cC_{4\delta'} \cap \cE_{4\delta'} } \right] + \mathbb{P}({\cF_0}^c) \nonumber \\
  &\leq 
  C   \log n   \sum_{k=3}^n   (d')^{{k \choose 2}} n^{ -{k \choose 2} + k -\frac{k }{2(k -1)} (1- \xi)^2 (1+\delta') + \frac{1+4\delta'}{2}\varepsilon^{1/2} + \gamma} \nonumber \\
  &+ Cn\cdot  n^{-  (1- \xi)^2 (1+\delta') + \frac{1+4\delta'}{2}\varepsilon^{1/2} + \gamma}  +Cn^{-2\delta'},
\end{align}
where \eqref{141} and \eqref{142} are used to bound the first and second terms respectively. The multiplicative factors of $\log n$ and $n$ appear as a result of a union bound over the components contributing to the index sets $I$ and $J$ respectively. 
Recalling  $\xi =  (2 \varepsilon^{1/2} +  8 \varepsilon \delta' )^{1/4}$ and $\delta'$ from \eqref{delta'}, note that $\lim_{\varepsilon\ro} \delta' = \delta
 $ and $\lim_{\varepsilon\ro} \xi = 0$. Furthermore, recall from \eqref{keydef} that $\psi(\delta)=\min_{k\ge 2} \phi_k(\delta)$ where $\phi_k(\delta)=\frac{k(k-3)}{2} + \frac{1+\delta}{2} \frac{k}{k-1}.$
 
{
 Hence, by taking $\gamma = \varepsilon$ and bounding the term $\log n$ by $n^\varepsilon$, there exists $\eta_1 = \eta_1(\varepsilon)$ with $ \lim_{\varepsilon \ro} \eta_1  = 0$ such that the first term of RHS in  \eqref{143}  is bounded by
 \begin{align} \label{600}
 \sum_{k=3}^n &  (d')^{{k \choose 2}}  n^{ -{k \choose 2} + k -\frac{k }{2(k -1)} (1- \xi)^2 (1+\delta') + \frac{1+4\delta'}{2}\varepsilon^{1/2} + 2\varepsilon}   \\ \label{601}
 &\leq   \sum_{k=3}^n  (d')^{{k \choose 2}}  n^{-\frac{k(k-3)}{2}  - \frac{1+\delta}{2} \frac{k}{k-1}  + \frac{\eta_1}{2}   }    \\ \label{145}
 & \leq    C (\log n)^{1/4}  (d')^{{(\log n)^{1/4} \choose 2}}      n^{-\psi(\delta)+\frac{\eta_1}{2} } + \sum_{k= (\log n)^{1/4}}^n  n^{ \frac{1}{2} {k \choose 2} - \frac{k(k-3)}{2}  - \frac{1+\delta}{2} \frac{k}{k-1}  + \frac{\eta_1}{2}  }  <   Cn^{-\psi(\delta)+\eta_1  }.
 \end{align}
As the reader perhaps already notices, the cutoff $(\log n)^{1/4}$ is not special and any poly-log cutoff   $(\log n)^r$ with $0<r<1/2$ works.

For further applications later, we provide a quantitative bound for $\eta_1$.   Using \eqref{610} and the fact that  $ \frac{k}{2(k-1)} \leq 1$ for $k\geq 2$, one can estimate the difference between two exponents of $n$  in \eqref{600} and \eqref{601}:
 \begin{align} \label{603}
& \frac{k}{2(k-1)} (1+\delta) - \frac{k}{2(k-1)}   (1- \xi)^2 (1+\delta') + \frac{1+4\delta'}{2}\varepsilon^{1/2} + 2\varepsilon \nonumber \\
& \leq (1+\delta) - (1-2\xi) (1+ \delta  - \sqrt{2\varepsilon} (1+\delta)^{3/2    }) + \frac{1+4\delta}{2}\varepsilon^{1/2} + 2 \varepsilon  \nonumber    \\
&\leq  4( \varepsilon^{1/8}+\delta^{1/4} \varepsilon^{1/4}) (1+\delta) +   \sqrt{2\varepsilon} (1+\delta)^{3/2    }   +  \frac{1+4\delta}{2}\varepsilon^{1/2} +2 \varepsilon  =: r_\delta(\varepsilon),
\end{align} 
where we used  $\xi  =    (2 \varepsilon^{1/2} +  8 \varepsilon \delta' )^{1/4} \leq 2\varepsilon^{1/8}+2\delta^{1/4} \varepsilon^{1/4}$ in the last inequality. In addition, for any constant $\eta_1>0$, the  inequality \eqref{145} holds for sufficiently large $n$. Hence,   $\eta_1>0$ can be chosen as 
\begin{align} \label{602}
\eta_1  = 2 r_\delta(\varepsilon),
\end{align}
which obviously converges to $0$ as $\varepsilon \rightarrow 0$.

 }
Similarly, taking $\gamma = \varepsilon$ in the second term of  \eqref{143},    for some $\eta_2 =  \eta_2(\varepsilon) $ such that $ \lim_{\varepsilon \ro} \eta_2  = 0$,
 \begin{align} \label{146}
  n\cdot  n^{-  (1- \xi)^2 (1+\delta') + \frac{1+4\delta'}{2}\varepsilon^{1/2} + \varepsilon}  \leq n^{-\delta + \eta_2}\le n^{-\psi(\delta)+\eta_2}. 
 \end{align}
 Hence, applying \eqref{145} and \eqref{146} to \eqref{143}, {using the bound for $\delta'$ in \eqref{610}, for sufficiently small $\varepsilon$,}
 \begin{align} \label{128}
 \mathbb{P}&(\lambda_1(Z^{(1)}) \geq \sqrt{2(1+\delta') \log n} ) < Cn^{-\psi(\delta)+\max(\eta_1,\eta_2)} . 
 \end{align}
 Recall by Lemma \ref{lemma 32}, for all large $n,$
  \begin{align*}
 \mathbb{P}(\lambda_1 (Z^{(2)}) \geq  \sqrt{\varepsilon} (1+\delta) \sqrt{\log n}) \leq  n^{-2 \delta- \delta^2+o(1)}\le n^{-\delta+o(1)}\le n^{-\psi(\delta)+o(1)}. 
  \end{align*}
Since $\varepsilon>0$ is arbitrary small, by \eqref{122} and the above two displays, we are done. \qed

\section{Structure conditioned on $\cU_\delta$}\label{structureproof}
We prove Theorem \ref{theorem structure} in this section.
We begin by stating some facts about $\phi_\delta$. 
{
Recall that $\cM(\delta)$ is the set of of  minimizers of $\phi_\delta(\cdot)$, and by the strict convexity of $\phi_\delta(\cdot)$,  $\cM(\delta)$ is at most of size 2 containing either a single element or two consecutive numbers. {In addition, since $\delta>\delta_2$, we have $\psi(\delta) > \phi_\delta(2) = \delta$. From this, one can deduce that there exists  a constant $c(\delta) \in (0,\min( \delta - \psi(\delta) ,  1) )$ such that}
\begin{align} \label{105}
k\notin  \cM(\delta) \Rightarrow \phi_\delta(k) - \psi(\delta) \geq c(\delta)
\end{align}
(recall that $\psi(\delta) = \min_{k\geq 2} \phi_\delta(k)$).
In fact, let us define, in the case when $\cM(\delta) = \{h(\delta)\}$ is a singleton, by the strict convexity of $\phi_\delta(\cdot)$,  $$c(\delta) = \min \Big( \phi_\delta (h(\delta) - 1)   -\phi_\delta (h(\delta)), \phi_\delta (h(\delta) + 1)    -\phi_\delta (h(\delta)), \frac{1}{2}(\delta - \psi(\delta)) , \frac{1}{2} \Big),$$
and when $\cM(\delta) = \{h(\delta), h(\delta)+1\}$ (recall that $h(\delta)$ is the minimal element of $\cM(\delta)$), 
 $$c(\delta) = \min \Big( \phi_\delta (h(\delta) - 1)   -\phi_\delta (h(\delta)), \phi_\delta (h(\delta) + 2)    -\phi_\delta (h(\delta)+1), \frac{1}{2}(\delta - \psi(\delta))  , \frac{1}{2}  \Big).$$
The minimum with $1/2$ and $(\delta - \psi(\delta))/2 $ is taken for technical reasons since in later applications we will need $c(\delta)$ to be small enough, while \eqref{105} holds even without it.
Note that the quantity $c(\delta)$ can be arbitrary close to $ 0$. In fact, for any $\delta_0$ such that $|\cM(\delta_0)|=2$,  $c(\delta)$ is close to $ 0$ if  $\delta$ is  close to $\delta_0$.
}

Recall the notation $\bar k$ from \eqref{maximalsize}.
Now by the same chain of reasoning as in  \eqref{143}, setting $\xi :=  (2 \varepsilon^{1/2} +  8 \varepsilon \delta' )^{1/4}$ and $\gamma = \varepsilon$,  we obtain that for some $\eta_1,\eta_2$ with $ \lim_{\varepsilon \ro} \eta_1  = \lim_{\varepsilon \ro} \eta_2 = 0$,   
\begin{align} 
\mathbb{P}&(   \bar{k} \notin \cM(\delta), \  \lambda_1(Z^{(1)}) \geq \sqrt{2(1+\delta') \log n} ) \nonumber  \\ \label{606}
&\leq C   \log n    \sum_{k \notin \cM(\delta) }   (d')^{{k \choose 2}} n^{ -{k \choose 2} + k -\frac{k }{2(k -1)} (1- \xi)^2 (1+\delta') + \frac{1+4\delta'}{2}\varepsilon^{1/2} + \varepsilon}  \\
  &+ Cn\cdot  n^{-  (1- \xi)^2 (1+\delta') + \frac{1+4\delta'}{2}\varepsilon^{1/2} + \varepsilon}  +Cn^{-2\delta'} \nonumber \\ \label{502} 
  &\leq
 C { n^{-\psi(\delta) - c(\delta) +\eta_1} } +C n^{-\delta + \eta_2 },
\end{align}
where the bound on the first term  is obtained as follows. By \eqref{603}, for each  $k \notin \cM(\delta)$,  the exponent of $n$ in  \eqref{606} is bounded by
\begin{align*}
 & -{k \choose 2} + k -\frac{k }{2(k -1)} (1- \xi)^2 (1+\delta') + \frac{1+4\delta'}{2}\varepsilon^{1/2} + \varepsilon \\
  &\leq     -{k \choose 2} + k -\frac{k }{2(k -1)}  (1+\delta) + r_\delta(\varepsilon)    =  -  \phi_\delta(k) + r_\delta(\varepsilon)  \overset{\eqref{105}}{\leq}  - \psi(\delta ) - c(\delta) + r_\delta(\varepsilon)  .
\end{align*}
  Hence, by the argument \eqref{600}-\eqref{145}, the term \eqref{606} can be bounded by $n^{- \psi(\delta ) - c(\delta) + 2 r_\delta(\varepsilon)  }$, and since $\lim_{\varepsilon \rightarrow 0}  r_\delta(\varepsilon)  = 0$, we obtain \eqref{502}.
{Therefore, using the fact that $\psi(\delta)  + c(\delta)  <\delta$, for sufficiently small $\varepsilon>0$,
\begin{align}  \label{152}
\mathbb{P}&(   \bar{k} \notin \cM(\delta), \  \lambda_1(Z^{(1)}) \geq \sqrt{2(1+\delta') \log n} ) \leq  C    n^{-\psi(\delta) - c(\delta) +\eta_1}. 
\end{align}}
Since the statement of the theorem is about the entire graph $X$ and not just $X^{(1)},$ we will now show that superimposing $X^{(2)}$ on the latter does not alter the size of the maximal clique with high probability owing to the sparsity of $X^{(2)}$.
Recall that we use $k_X$ to denote the size of the maximal clique in $X$.
Since $k_X \geq \bar{k}$ (recall that $\bar{k}$ is the maximal clique size in $X^{(1)}$), \eqref{152}  implies
\begin{align} \label{151}
\mathbb{P}&(  k_{X} \leq h(\delta)-1,   \lambda_1(Z^{(1)}) \geq \sqrt{2(1+\delta') \log n} ) \leq C { n^{-\psi(\delta) - c(\delta) +\eta_1}. }
\end{align}

To treat the  non-trivial direction, i.e., superimposing $X^{(2)}$ does not make $k_{X}$ larger than $\bar k,$  define the event $\cF_1$, measurable with respect to  $X^{(1)}$, by
 \begin{align} \label{156}
 \cF_1:= \{|E(H)| - {\bar{k} \choose 2} \leq |V(H)|- \bar{k}: \ \text{any   subgraph} \  H \ \text{such that}  \ |H|\leq  2 h(\delta) + {2} \}.
 \end{align}
In words, under $\cF_1,$ the subgraph induced on any subset of vertices of size bigger than $\bar k,$ has significantly smaller number of edges than the clique induced on the same.

Note that, in particular, on $\cF_1,$ $X^{(1)}$ has a unique maximal clique $K:=K_{X^{(1)}}$ of size $\bar k.$ This follows from the definition of $\cF_1$ applied to the subgraph induced on $K\cup K'$ where $K'$ is another set of $\bar k$ vertices.

We will show first show that $\cF_1$ is likely, and on it,  for $X$ to have a larger clique, $X^{(2)}$ must fill in the `substantially many' edges absent in $X^{(1)}$ which will then be shown to be unlikely. \\

\noindent
\textbf{Showing $\cF_1$ is likely.}
Towards this, observe that 
 \begin{align} \label{153}
 \mathbb{P}( \{\bar{k}=k \} \cap  \cF_1^c){\overset{\eqref{346}}{\leq}}  C \sum _{i=1}^{2h(\delta)+2} \Big(\frac{i}{n}\Big)^{ {k \choose 2} - k+1}  \leq  C (2h(\delta)+2)^{ {k \choose 2} - k+2} 
 \frac{1}{n^{ {k \choose 2} - k+1} }.
 \end{align}
Hence, recalling the event $\cF_0$ in \eqref{147}, using   the above and  the argument of \eqref{143} again,  there is $\eta'_1 $ with $ \lim_{\varepsilon \ro} \eta'_1 = 0$   such that {for $\delta>\delta_2$ (recall the definition from Remark \ref{remark 1.2})},
\begin{align} \label{154} 
\mathbb{P}&(\bar{k}\in \cM(\delta)  , \cF_1^c , \lambda_1(Z^{(1)}) \geq \sqrt{2(1+\delta') \log n} ) \nonumber \\
&\leq 
 \sum_{k\in \cM(\delta)}  \mathbb{E} \left[ \mathbb{P}(\max_{i\in I} \{ \lambda_1(Z^{(1)}_i) \}\geq \sqrt{2(1+\delta') \log n} \mid X^{(1)}) \1_{\cF_0}  \1_{\cF_1^c} \1_{\bar{k} = k} \right]  \nonumber \\
 &+\mathbb{E} \left[ \mathbb{P}(\max_{i\in J} \{ \lambda_1(Z^{(1)}_i) \}\geq \sqrt{2(1+\delta') \log n} \mid X^{(1)}) \1_{\cD_{4\delta'}\cap \cC_{4\delta'} \cap \cE_{4\delta'} } \right]+ \mathbb{P}( \cF_0^ c )   \nonumber    \\
  &\leq 
  C  (\log n) n^{-1} \sum_{k\in \cM(\delta)}  (2h(\delta)+2)^{ {k \choose 2} - k+2}     n^{ -{k \choose 2} + k -\frac{k }{2(k -1)} (1- \xi)^2 (1+\delta') + \frac{1+4\delta'}{2}\varepsilon^{1/2} + \varepsilon} \nonumber \\
  &+ Cn\cdot  n^{-  (1- \xi)^2 (1+\delta') + \frac{1+4\delta'}{2}\varepsilon^{1/2} + \varepsilon}  +Cn^{-2\delta'}  \nonumber \\
  &\leq   C n^{-\psi(\delta) - 1 + \eta'_1},
\end{align}
where the extra $n^{-1}$ factor in the first term comes from \eqref{153}.
Putting the above together, letting 
\begin{align}\label{f2}
\cF_2  : = \{\bar{k}\in \cM(\delta)\} \cap  \cF_1,
\end{align}
by \eqref{152}   and  \eqref{154},  for large $\delta$,
 \begin{align} \label{155}
 \mathbb{P}(  \cF_2^c, \lambda_1(Z^{(1)}) \geq \sqrt{2(1+\delta') \log n}) \leq     n^{-\psi(\delta) - c(\delta) + \eta_1} 
 \end{align}
 (recall that $c(\delta)\in (0,1)$). {By Lemma \ref{lemma 32}, this in particular implies
 \begin{align}   \label{904}
  \mathbb{P}&(  \cF_2^c, \lambda_1(Z)  \geq \sqrt{2(1+\delta) \log n}) \nonumber \\
   & \leq   \mathbb{P}( \cF_2^c, \lambda_1(Z^{(1)}) \geq \sqrt{2(1+\delta') \log n}) +  \mathbb{P}(  \lambda_1(Z)  \geq \sqrt{2(1+\delta) \log n},  \lambda_1(Z^{(1)})  <  \sqrt{2(1+\delta') \log n}) \nonumber  \\
 &\leq C n^{-\psi(\delta) - c(\delta) +\eta_1} +   \mathbb{P}( \lambda_1(Z^{(2)}) \geq  \sqrt{\varepsilon}(1+\delta)  \sqrt{  \log n})  \leq  C n^{-\psi(\delta) -c(\delta) +\eta_1}.
 \end{align}
Combining this with \eqref{112}, since $\lim_{\varepsilon \rightarrow 0} \eta_1 = 0$, there exists $\varepsilon_ 0 = \varepsilon_0(\delta)>0$ such that for any $\varepsilon<\varepsilon_0$ (recall that $\e$ implicitly appears in the definition of $X^{(1)}$),   
 \begin{align} \label{909}
\lim_{n\ri}  \mathbb{P}(\cF_2^c \mid \cU_\delta)  = 0.
 \end{align}
 In particular, {recalling $\cF_2 \subset \cF_1$} and $\cF_1$ implies the uniqueness of maximal clique   $K$    in $X^{(1)}$, 
 \begin{align} \label{905}
\lim_{n\ri}  \mathbb{P}( \text{there is a unique maximal clique}  \  K \  \text{in} \  X^{(1)} \mid \cU_\delta)  = 1.
 \end{align}
 } 
For convenience, let us denote the above event by $\sf{Unique}.$ 
 We now proceed to showing that the unique maximal clique $K$ of $X^{(1)}$ continues to be so on superimposing $X^{(2)}$ to obtain $X.$\\
 
\textbf{Showing $K_{X}=K_{X^{(1)}}$.} 
{ We first define some notations. For two subsets of vertices $A$ and $B$, {define the set of undirected edges}
 \begin{align*}
{ {\sf{Edge}}(A,B) } := \{e = (i,j): i<j, i,j\in B \setminus A\} \cup \{e=(i,j): i\in B\setminus A, j\in A\cap B\}. 
 \end{align*}
 Note that
 \begin{align} \label{360}
  | {\sf{Edge}}(A,B)| = { |B| \choose 2} -  { |{A\cap B}| \choose 2}.
\end{align} 
Then, define the random subset of edges, measurable with respect to $X^{(1)}$, by
 \begin{align*}
 X^{(1)}(A,B) =   {\sf{Edge}}(A,B) \cap E( X^{(1)}).
 \end{align*}

 We first verify that under the event $\cF_2 = \{\bar{k}\in \cM(\delta)\} \cap  \cF_1$,  any  clique $K'$ of size  $\ell \leq \bar{k}$ satisfies
\begin{align} \label{381}
| X^{(1)} (K, K') | \leq  \ell - |K\cap K'|
\end{align} 
where as mentioned above $K$ in the unique maximal clique in $X^{(1)}$.
Since 
\begin{align*}
 |E( K\cup K')|  \geq   {\bar{k} \choose 2} + | X^{(1)} (K, K') |,
\end{align*}
   applying \eqref{156} to  $H = K\cup K'$ (note that  under the event $\bar{k}\in \cM(\delta)$, we have $|K\cup K'| \leq 2  \bar{k} \leq  2 h(\delta) + 2$),
  \begin{align*}
   {\bar{k} \choose 2} + | X^{(1)} (K, K') | -  {\bar{k} \choose 2}   \leq  |K\cup K'| - \bar{k} = \ell- |K\cap K'|,
 \end{align*}
which implies \eqref{381}.

Note that conditioning on $X^{(1)},$ the entries of $X$ are independent and satisfy 
{\begin{align*}
\mathbb{P}(X_{ij}=1 | X^{(1)}_{ij} = 0 ) &\leq \frac{2d}{n} \text{ for large }n,\\
\mathbb{P}(X_{ij}=1 | X^{(1)}_{ij} = 1 )&=1.
\end{align*}
  In fact, using the fact that $\mathbb{P}( X_{ij} = 0)  = 1-\frac{d}{n} \geq \frac{1}{2}$ for large $n$,
   \begin{align*}
  \mathbb{P}(X_{ij}=1 | X^{(1)}_{ij} = 0 )  = \frac{\mathbb{P}(|Y_{ij}|< \sqrt{\varepsilon \log \log n} , X_{ij} = 1 )}{\mathbb{P}(X^{(1)}_{ij} = 0)} \leq \frac{\mathbb{P}( X_{ij} = 1 )}{\mathbb{P}( X_{ij} = 0)} \leq  \frac{2d}{n},
\end{align*}    
and the second identity is obvious.}

We will now define two events $\cB_0$ and $\cB_1,$ which will be shown to be very likely on $\cU_\delta$ and together would imply that $K$ is the unique maximal clique in $X$ and moreover, the largest clique not fully contained in $K$ is a triangle.

{We begin with $\cB_0$ which is measurable with respect to the sigma algebra generated by $X^{(1)}$ and $X$, 
\begin{align} 
\cB_0 : = {\sf{Unique}}\cap \{ \text{there is no clique of size}  \ 4  \  \text{in} \ X \ \text{edge-disjoint from} \ K  \} .
\end{align} 
Recalling that $K$ is of size $\bar{k}$, 
by BK inequality and using Lemma \ref{lemma 31}  
\begin{align} \label{912}
\mathbb{P}(\cB_0^c \cap \{\bar{k}\in \cM(\delta)\}) \leq  C \Big(\frac{1}{n}\Big)^{ {{h}(\delta) \choose 2}   - h(\delta) }  \Big(\frac{1}{n}\Big)^{      {4 \choose 2}   - 4 } = C  \Big(\frac{1}{n}\Big)^{      {{h}(\delta) \choose 2}   - h(\delta) + 2 } ,
\end{align}
where $C>0$ is a constant depending only on $\delta$.
We write
 \begin{align}  \label{910}
 \mathbb{P}&\left( \cB_0^c,   \lambda_1(Z^{(1)}) \geq \sqrt{2(1+\delta') \log n}\right) \nonumber \\
 &\leq   \mathbb{E}  \left[ \mathbb{P}  ( \lambda_1(Z^{(1)}) \geq \sqrt{2(1+\delta') \log n}     |  X^{(1)} , X ) \1_{\cF_0} \1_{\bar{k}\in \cM(\delta) }   \1_{\cB_0^c }     \right]   \nonumber \\
 &+  \mathbb{P}\left(\big(\cF_0 \cap  \{\bar{k}\in \cM(\delta)\}\big)^c, \lambda_1(Z^{(1)}) \geq \sqrt{2(1+\delta') \log n}\right).
 \end{align}
 Since $\lambda_1(Z^{(1)}) $ and $X$ are conditionally independent given  $ X^{(1)} $, by  \eqref{141} and \eqref{142} with $\gamma =  \varepsilon$, there is $\eta_3 = \eta_3 (\varepsilon)$ with $\lim_{\varepsilon \rightarrow 0} \eta_3  = 0$ such that  for sufficiently small $\varepsilon>0$,
 \begin{align*}
  \mathbb{P}  &( \lambda_1(Z^{(1)}) \geq \sqrt{2(1+\delta') \log n}     |  X^{(1)} , X )\1_{\cF_0} \1_{\bar{k}\in \cM(\delta) } \\
  & \leq    C (\log n) n^{-\frac{\bar{k}}{2(\bar{k}-1)} (1- \xi)^2 (1+\delta') + \frac{1+4\delta'}{2}\varepsilon^{1/2} +\varepsilon } + Cn\cdot  n^{-  (1- \xi)^2 (1+\delta') + \frac{1+4\delta'}{2}\varepsilon^{1/2} + \varepsilon } \\
  & \leq   C  n^{-\frac{\bar{k}}{2(\bar{k}-1)}(1+\delta) + \eta_3 + \varepsilon } \leq   C  n^{-\frac{h(\delta)+1}{2h(\delta)} (1+\delta)+ \eta_3 + \varepsilon},
 \end{align*}
{
 where the second and last inequalities follow by observing $\frac{\bar{k}}{2(\bar{k}-1)}(1+\delta)  \leq      \phi_\delta(\bar{k})  <     \phi_\delta(2) =  \delta $  (since $\bar{k}\geq 3$ and $\delta>\delta_2$) and     $\bar{k} \leq h(\delta)+1$ respectively.}
 Hence, applying this and \eqref{912} to  \eqref{910}, using \eqref{129} and \eqref{155} to bound the last term in \eqref{910}, for sufficiently small $\varepsilon>0$,
 \begin{align} \label{911}
  \mathbb{P}& \left( \cB_0^c,   \lambda_1(Z^{(1)}) \geq \sqrt{2(1+\delta') \log n}\right) \nonumber\\
  &\leq C  n^{-\frac{h(\delta)+1}{2h(\delta)}(1+\delta)  + \eta_3 + \varepsilon }\P(\cB_0^c \cap \{\bar{k}\in \cM(\delta)\})+  Cn^{-2\delta'} + n^{-\psi(\delta) -c(\delta)+\eta_1} \nonumber\\
  &  \overset{\eqref{912}}{\leq}    C  n^{-\frac{h(\delta)+1}{2h(\delta)}(1+\delta)  + \eta_3 + \varepsilon } \Big(\frac{1}{n}\Big)^{      {{h}(\delta) \choose 2}   - h(\delta) + 2 }  + Cn^{-2\delta'} + n^{-\psi(\delta) -c(\delta)+\eta_1} \nonumber \\
  &\leq  C  n^{ -\psi(\delta) -1 + \eta_3 +\varepsilon } +  Cn^{-2\delta'}+ n^{-\psi(\delta) -c(\delta)+\eta_1}\leq  C n^{-\psi(\delta) -c(\delta)+\eta_1}
 \end{align}
 (recall that $c(\delta) \in (0,1)$),
 where   the third inequality follows from  the fact
 \begin{align*}
 \frac{h(\delta)+1}{2h(\delta)} (1+\delta)    +   {\tilde{h}(\delta) \choose 2}   - h(\delta) & =  \left(\frac{h(\delta)}{2(h(\delta)-1)} (1+\delta) +\frac{h(\delta) (h(\delta)-3)}{2 }\right)  - \frac{1}{2h(\delta) (h(\delta) -1) }   \\
&  \geq    {\psi(\delta)  -  1},
 \end{align*}
 where the last inequality follows from the observation that the term in the parentheses is exactly $\psi(\delta).$
Let us define the another event, again measurable with respect to the sigma algebra generated by $X^{(1)}$ and $X$, 
\begin{align} \label{b1}
\cB_1 : = {\sf{Unique}}\cap \{ \text{there is no clique}  \  K' \   \text{in} \ X \  \text{such that}  \  4\leq |K'| \leq \bar{k} \ \text{and} \ 2\leq |K\cap K'|\leq \bar{k}-1 \}.
\end{align} 
Thus in words, the event demands the existence of a clique of size at least $4$ which is not edge disjoint from $K$ but also is not contained in the latter.

Note that by \eqref{360} and  \eqref{381},  under the event $\cF_2$, the number of missing edges (of $X^{(1)}$)  in ${\sf{Edge}}(K,K')$ is
\begin{align*}
|{\sf{Edge}}(K,K') \setminus   X^{(1)}(K,K') | \geq {|K'| \choose 2} -  {|K\cap K'| \choose 2} - (|K'|-|K\cap K'|
).
\end{align*} Hence, 
\begin{align} \label{900}
 \mathbb{P}& ( \cB_1 ^c \mid   X^{(1)}   )   \1_{   \cF _2 } \nonumber  \\
 &\leq 
 \sum_{\ell = 4}^{\bar{k}}  \sum_{m=2}^{\ell-1} \sum_{|K'| =\ell, |K\cap K'|=m }\mathbb{P}(X_{ij}= 1  \ \text{for all edges} \  e = (i, j) \in {\sf{Edge}}(K,K') \setminus   X^{(1)}(K,K')   \mid  X^{(1)}  )   \1_{   \cF _2 }   \nonumber  \\ 
 &\leq   \sum_{\ell = 4}^{\bar{k}}  \sum_{m=2}^{\ell-1}      \bar{k}^{m}  n^{\ell-m}\Big (\frac{2d}{n}\Big)^{ {\ell \choose 2} -  {m \choose 2} - (\ell-m
) } \leq  C \sum_{\ell = 4}^{\bar{k}}  \sum_{m=2}^{\ell-1}     \Big(\frac{1}{n}\Big)^{   (   {\ell \choose 2} - 2\ell   )  - ( {m    \choose 2} - 2m) } \leq   C\frac{1}{n},
\end{align}
where $C>0$ is a constant depending only on $\delta$. Here, the last inequality follows from the fact that a function $f(k):=  {k \choose 2} - 2k$ satisfies  the following property: $f(2) = f(3) = -3$, $f(4)=-2$ and is strictly increasing for $k\geq 4$.

Hence,  observing that, {$X$ and $\lambda_1(Z^{(1)})$ are conditionally independent given $X^{(1)},$} by \eqref{128} and \eqref{155},
 \begin{align} 
 \mathbb{P}&\left( \cB_1^c,   \lambda_1(Z^{(1)}) \geq \sqrt{2(1+\delta') \log n}\right) \nonumber \\
 &\leq  \mathbb{E}  \left[ \mathbb{P}  (\cB_1^c      |  X^{(1)} ,\lambda_1(Z^{(1)})    ) \1_{\cF_2} \1_{  \lambda_1(Z^{(1)}) \geq \sqrt{2(1+\delta') \log n}}\right]  + \mathbb{P}(\cF_2^c, \lambda_1(Z^{(1)}) \geq \sqrt{2(1+\delta') \log n}) \nonumber \\
 &\leq    C n^{-\psi(\delta) -c(\delta) +\eta_1}.
 \end{align}
 Combining with \eqref{152} and \eqref{911},   
 \begin{align}
 \mathbb{P}&\left(  (\cB_0\cap  \cB_1 \cap \{\bar{k}\in \cM(\delta)\} )^c ,  \lambda_1(Z^{(1)}) \geq \sqrt{2(1+\delta') \log n}\right) \leq    C n^{-\psi(\delta) -c(\delta) +\eta_1}. 
 \end{align}
Proceeding as in \eqref{904}-\eqref{905},
there exists $\varepsilon_1>0$ such that for any $\varepsilon<\varepsilon_1$,
\begin{align} \label{400}
\lim_{n\ri}  \mathbb{P}(\cB_0\cap \cB_1 \cap \{\bar{k}\in \cM(\delta)\}   \mid  \cU_\delta) = 1.
\end{align}
Recalling that the size of clique $K$ is $\bar{k},$
  the event $\cB_0\cap \cB_1 \cap \{\bar{k}\in \cM(\delta)\}  $ implies the statements in  Theorem \ref{theorem structure} and in particular
\begin{align} \label{906}
\lim_{n\ri}  \mathbb{P}( \text{there is a unique maximal clique}  \  K_X \  \text{in} \  X \  \text{and is equal to} \  K \mid \cU_\delta)  = 1.
\end{align} 
 
}  
 \qed


\section{Optimal localization of leading eigenvector} \label{section 7}
We prove Theorem \ref{eigenvectorloc} in this section. 
Recall $v = (v_1,\cdots,v_n)$ is the unit eigenvector  associated with the largest eigenvalue $\lambda_1=\lambda_1(Z)$ and let $K_X$ be the unique maximal clique (recall that Theorem \ref{theorem structure} ensures uniqueness conditioned on $\cU_\delta$ with high probability). Then,
\begin{align} \label{162}
\lambda_1 = \sum_{1\leq i,j\leq n} Z_{ij} v_iv_j =  \sum_{1\leq i,j\leq n} Z^{(1)}_{ij} v_iv_j  +  \sum_{1\leq i,j\leq n} Z^{(2)}_{ij} v_iv_j.
\end{align}

The proof has two parts. In the first, we prove that the eigenvector allocates most of its mass on $K_X,$ while in the second part we further show that the mass is uniformly distributed.\\

\noindent
\textbf{Mass concentration.}
Let us recall $r_\delta(\varepsilon)$ and  $c(\delta)$ defined in \eqref{603} and \eqref{105} respectively.  We choose a parameter $\varepsilon$ sufficiently small so that {
\begin{align}  \label{605}
2r_\delta(\varepsilon) &< c(\delta), \\ \label{190}
\varepsilon &\leq \frac{1}{\delta^4}, \\ \label{611}
\varepsilon&<\min ( \varepsilon_0  , \varepsilon_1)
\end{align}
($\varepsilon_0$ and $\varepsilon_1$ are positive constant depending on $\delta$ such that   \eqref{905} and \eqref{906} are  satisfied for $\varepsilon<\varepsilon_0$ and $\varepsilon<\varepsilon_1$ respectively).}
Recall that by \eqref{905} and \eqref{906},   conditionally on $\cU_\delta$,  with probability tending to $1$, the following is true: the maximal cliques $K_{X^{(1)}}$ and $K_X$ are unique and equal which will be often denoted by $K$ for brevity. Hence,
throughout the proof,  we assume the occurrence  of this event.

Recall
\begin{equation}\label{1001}
\cA_1:=\Big \{ \sum_{i\in K} v_i^2 \geq 1-\kappa\Big\},
\end{equation}
where $\kappa>0$ is the parameter in the statement of the theorem.
Since
\begin{align*}
\mathbb{P}\left(\sum_{1\leq i,j\leq n} Z^{(2)}_{ij} v_iv_j \geq \sqrt{\varepsilon}(1+\delta) \sqrt{\log n}\right)\leq \mathbb{P}\left(\lambda_1 (X^{(2)}) \geq  (1+\delta) \sqrt{\frac{\log n }{\log \log n}}\right)  \leq n^{-(2\delta+\delta^2)+o(1)},
\end{align*}
by \eqref{122},{   for any event $\cA$,   }
\begin{align} \label{163}
\mathbb{P}( \cA, \lambda_1 \geq \sqrt{2(1+\delta) \log n}) \leq  \mathbb{P}(\cA, \sum_{1\leq i,j\leq n} Z^{(1)}_{ij} v_iv_j   \geq  \sqrt{2(1+\delta') \log n})  +n^{-(2\delta+\delta^2)+o(1)} ,
\end{align}
where $\delta'>0$ as before is defined to be
\begin{align} \label{191}
\sqrt{2(1+\delta')} = \sqrt{2(1+\delta)} - \sqrt{\varepsilon}(1+\delta).
\end{align} 
{Note that since $\varepsilon \leq  \frac{1}{\delta^4}$, using the bound  for $\delta'$ in  \eqref{610}, we have
\begin{align} \label{delta'bound}
  \delta'   =  \delta + o_\delta(1) \qquad \text{as} \quad
   \delta\ri.
\end{align} 
}
We will now bound the first term on the RHS of \eqref{163} {with $\cA = \cA_1$} using Proposition \ref{spectral bound} and the fact that on the high probability event $\cF_1$ defined in \eqref{156}, the largest clique outside $K$ is at most {a triangle} which would make it suboptimal in a large deviation theoretic sense for the eigenvector to allocate mass off of $K$.
We now proceed to make this precise. The arguments will bear similarities with those appearing in the proof of Proposition \ref{prop}.

Let $C_1,\cdots,C_m$ be connected components of $X^{(1)}$, and let without loss of generality $C_1$ contain the clique $K$ of size $\bar{k}$.  Let $k_i$ be the maximum clique size in $C_i$. 

We will now work with the high probability event $\cF_0$ from \eqref{147}. As in the proof of Proposition \ref{prop},
 define $B_{1}$ to be the collection of  (directed)  edges defined by
\begin{align}\label{lowvalues12}
B_1:= \{e = (i,j) \in  \overrightarrow{E(C_1)}: v_i^2 , v_j^2 < \bar{\eta}^2 \}
\end{align} 
(recall that for any graph $H$, $\overrightarrow{E(H)}$ denotes the  set of \emph{directed}  edges in $H$),
{where the parameter $ \bar{\eta}$ is chosen to be
 \begin{align} \label{192}
 \bar{\eta}= \varepsilon^{1/4}.
\end{align} 
Define the set of (directed) edges  $B_2:= {\overrightarrow{E(C_1)}} \setminus B_1$.     Since  $|\{i: v_i^2 \geq \eta^2\}| \leq \frac{1}{ \bar{\eta}^2}$, under the event $\cF_0$, using the definition of $\cD_{4\delta'},$
\begin{align}   \label{184}
 \frac{1}{2} |B_2| \leq  \frac{1+4\delta'}{ \bar{\eta}^2} \frac{\log n}{\log \log n},
\end{align}
}by the same reasoning as preceding \eqref{362}.
We write
\begin{align} \label{310}
\sum_{1\leq i,j\leq n} Z^{(1)}_{ij} v_iv_j = \sum_{(i,j)\in  \overrightarrow{E(C_1)} } Z^{(1)}_{ij} v_iv_j   
& =  \sum_{(i,j)\in B_1} Z^{(1)}_{ij} v_iv_j    +  \sum_{(i,j)\in B_2} Z^{(1)}_{ij} v_iv_j    = : S_1 + S_2.
\end{align}
{By the same reasoning as in \eqref{smallvalue2}, and following the same notation as in the latter, for $2\bar{\eta}^2\leq 1$, under the event $\cF_0$, 
\begin{align} \label{161}
\sum_{(i,j)\in B_1 }  v_i^2v_j^2    \leq 2 \bar{\eta}^2(1-\bar{\eta}^2 ) + 2(4 \delta'+1) \bar{\eta}^4 =: \theta^4.
\end{align}
 Note that since $\bar{\eta} =\varepsilon^{1/4} \leq \frac{1}{\delta}$ (see \eqref{190}), by \eqref{delta'bound}, we have
\begin{align} \label{theta}
\theta = O\Big(\frac{1}{\delta^{1/2}}\Big) \qquad \text{as} \quad   \delta \ri.
\end{align} 
 }
By Cauchy-Schwarz inequality and \eqref{161},
 \begin{align} \label{164}
 S_1 \leq 
\Big(\sum_{(i,j)\in B_1}  v_i^2v_j^2\Big)^{1/2}  \Big(\sum _{(i,j)\in B_1} (Z^{(1)}_{ij})^2\Big)^{1/2}  \leq  \theta^2  \Big  (\sum _{(i,j)\in \overrightarrow{E(C_1)}} (Z^{(1)}_{ij})^2 \Big)^{1/2}  .
 \end{align}

Next, we estimate $S_2$.   Define $x: = \sum_{i \in K} v_i^2$ and $y:= \sum_{C_1 \setminus K} v_i^2$.  Recalling the definition of event $\cF_1$ in \eqref{156}, on the latter, 
\begin{align} \label{149}
\text{ the maximum size of clique in} \   K^c  \  \text{is at most} \ 3.
\end{align} 
 In fact, if  $ K^c $ contains a clique $K'$ of size 4, then $|E(K\cup K')| \geq   {\bar{k} \choose 2} + 6 $ and $|V(K\cup K')| = \bar{k}+4$, which contradicts \eqref{156}.    Hence, under the event $\cF_1$,
\begin{align} \label{140}
\sum _{(i,j)\in B_2} v_i^2 v_j^2 \leq  \sum _{(i,j)\in \overrightarrow{E(C_1)}} v_i^2 v_j^2 &\leq     \sum_{i\neq j, i,j\in K} v_i^2v_j^2   +2 \sum_{i\in K, j\in C_1 \setminus K} v_i^2v_j^2 + \sum_{i\neq j, (i,j)\in  \overrightarrow{E(C_1\setminus K)} } v_i^2v_j^2 \nonumber \\
& \leq 2 \Big(\frac{\bar{k}-1}{2\bar{k}} x^2 +xy + \frac{1}{3}y^2 \Big) .
\end{align}
The final bound follows from \eqref{020}.
Thus, under the event $\cF_1$,
\begin{align} \label{165}
S_2 \leq   \Big(\sum_{(i,j)\in B_2}  v_i^2v_j^2\Big)^{1/2} \Big(\sum _{(i,j)\in B_2} (Z^{(1)}_{ij})^2\Big)^{1/2}  \leq   \Big(\frac{\bar{k}-1}{\bar{k}} x^2 +2xy + \frac{2}{3}y^2 \Big)^{1/2}  \Big  (\sum _{(i,j)\in B_2}  (Z^{(1)}_{ij}  ) ^2 \Big)^{1/2}  .
\end{align}
Note that using 
the fact  $$1 = \sum_{i=1}^n v_i^2 =x+y+ \sum_{\ell=2}^m  \Big (\sum_{i\in C_\ell} v_i^2 \Big ),$$
we have
{
\begin{align} \label{166}
 \sum_{1\leq i,j\leq n} Z^{(1)}_{ij} v_iv_j &  = S_1+S_2 +  \sum_{\ell=2}^m  \sum_{(i,j)\in  \overrightarrow{E(C_\ell)} } Z^{(1)}_{ij} v_iv_j  \nonumber   \\
&  \leq   S_1   +  S_2 + \sum_{\ell=2}^m \lambda(Z^{(1)}_\ell)  \Big(\sum_{i\in C_\ell} v_i^2 \Big) \leq S_1   +  S_2 +   (1-x-y)   \max_{\ell=2,\cdots,m}\lambda(Z^{(1)}_\ell).
\end{align}
}

We now estimate the following conditional probability
\begin{align} \label{167}
\mathbb{P}&\Big(x< 1- \kappa ,  \sum_{1\leq i,j\leq n}  Z^{(1)}_{ij} v_iv_j   \geq  \sqrt{2(1+\delta') \log n}\mid X^{(1)}  \Big)\1_{\cF_0\cap \cF_2} \nonumber   \\
&= \mathbb{P}\Big(x< 1- \kappa , y \geq  \frac{\kappa}{2},  \sum_{1\leq i,j\leq n}  Z^{(1)}_{ij} v_iv_j   \geq  \sqrt{2(1+\delta') \log n}\mid X^{(1)}  \Big)\1_{\cF_0\cap \cF_2} \nonumber \\   
&
+ \mathbb{P}\Big(x< 1- \kappa , y<  \frac{\kappa}{2},  \sum_{1\leq i,j\leq n}  Z^{(1)}_{ij} v_iv_j   \geq  \sqrt{2(1+\delta') \log n}\mid X^{(1)}  \Big)\1_{\cF_0\cap \cF_2}  \nonumber  \\
&=: R_1+R_2.
\end{align}
We next bound $R_1$ and $R_2$ in turn.\\

\noindent
\textbf{Bounding $R_1$.}
By \eqref{166},  
 \begin{align} \label{168}
R_1 
  &\leq      \mathbb{P}(S_1 \geq   \theta \sqrt{2(1+\delta') \log n} \mid X^{(1)}) \1_{\cF_0\cap \cF_2}   \nonumber   \\
  &+  \mathbb{P}\Big(x< 1- \kappa, y \geq  \frac{\kappa}{2}, S_2 \geq   (x+y-\theta ) \sqrt{2(1+\delta') \log n} \mid X^{(1)}  \Big)  \1_{\cF_0\cap \cF_2}  \nonumber   \\ 
 & +  \mathbb{P}\Big(  \max_{\ell=2,\cdots,m}\lambda(Z^{(1)}_\ell) \geq \sqrt{2(1+\delta') \log n}\mid X^{(1)}  \Big)  \1_{\cF_0\cap \cF_2}  \nonumber \\
 &=: R_{1,1} + R_{1,2}+R_{1,3}.
 \end{align}
{By \eqref{164} and   \eqref{212} in  Lemma \ref{chi tail}  with $\gamma = \varepsilon$,  $L =  \frac{1}{\theta^2} (1+\delta') \log n $ and $m\leq  \frac{2+4\delta'}{\varepsilon} \frac{\log n}{\log \log n} + 4\delta' $ (see \eqref{124} and \eqref{125}),   for sufficiently large $n$,
\begin{align}   \label{169}
R_{1,1}& \leq    \mathbb{P}  \Big  (\sum _{i<j, (i,j)\in \overrightarrow{E(C_1)}} (Z^{(1)}_{ij})^2  \geq \frac{1}{\theta^2} (1+\delta') \log n  \mid X^{(1)}  \Big)  \1_{\cF_0\cap \cF_2}    \leq   n^{-\frac{1+\delta'}{2 \theta^2  } + (1+2\delta' )   + \varepsilon }.
\end{align}
}
To bound $R_{1,2},$ we first need the following technical bound.
There exists a constant $\lambda=\lambda(\kappa)\in (0,\frac{1}{100})$ such that for sufficiently large $\delta$, under the event $\bar{k}\in \cM(\delta)$,
 for $x< 1- \kappa , y \geq  \frac{\kappa}{2}$, 
\begin{align}\label{183}
\frac{\bar{k}-1}{\bar{k}} x^2 +2xy + \frac{2}{3}y^2  <\Big (  \frac{\bar{k}-1}{\bar{k}}  - \lambda\Big) (x+y-\theta)^2.
\end{align}
{ In fact, rearranging, this inequality holds if
$$
\lambda x^2 + 2 \Big(\frac{1}{\bar{k}} + \lambda\Big) xy + 2 \Big(  \frac{\bar{k}-1}{\bar{k}}  - \lambda\Big)\theta (x+y) < \Big (  \frac{\bar{k}-1}{\bar{k}}  - \lambda - \frac{2}{3}\Big)y^2.
$$   Recall that  by \eqref{argmaxloc}, $\bar{k}\in \cM(\delta)$ implies 
\begin{align} \label{185}
    \Big (\frac{1+\delta}{2}\Big )^{1/3} - 1\leq    \bar{k} \leq  \Big (\frac{1+\delta}{2}\Big )^{1/3} + 3.
\end{align}  
Hence, there is  $\lambda = \lambda(\kappa) >0 $ such that 
for $x< 1- \kappa , y \geq  \frac{\kappa}{2}$ and sufficiently large $\delta$, under the event $\bar{k}\in \cM(\delta)$,
$$\lambda x^2 < \frac{1}{10}y^2,\,\, 2\Big(\frac{1}{\bar{k}} + \lambda\Big) xy   <  \frac{1}{10}y^2,\,\, 2\Big (  \frac{\bar{k}-1}{\bar{k}}  - \lambda\Big)\theta (x+y)  <  2\theta < \frac{1}{10}y^2$$ (see \eqref{theta} for the bound of $\theta$). If {$\lambda$ is small enough, say $\lambda\in (0,\frac{1}{100})$,} then  for  sufficiently large $\delta$,   under the event $\bar{k}\in \cM(\delta)$,  $\frac{3}{10} < \frac{\bar{k}-1}{\bar{k}}  - \lambda - \frac{2}{3}$, and thus we obtain \eqref{183}.}

Thus, by    \eqref{165} and \eqref{183}, using  the fact $ (\frac{\bar{k}-1}{\bar{k}} - \lambda )^{-1}\ge \frac{\bar{k}}{\bar{k}-1}  + \lambda$,
\begin{align} \label{1711}
 R_{1,2}& \leq  \mathbb{P}   \Big (   \sum_{i<j, (i,j)\in B_2} (Z^{(1)}_{ij}  ) ^2  \geq   \Big ( \frac{\bar{k}}{\bar{k}-1}  + \lambda\Big ) (1+\delta') \log n \mid X^{(1)}\Big  )  \1_{\cF_0\cap \cF_2} .
  \end{align} 
 {Note that the event  $ \sum _{i<j,(i,j)\in B_2} (Z^{(1)}_{ij}  ) ^2 \geq  t$ implies the existence of a random subset $J \in V(C_1)$ with $|J| \leq  \left \lfloor {  \frac{1}{\bar{\eta}^2} }   \right  \rfloor   $ such that  $ \displaystyle{\sum _{i<j, i\text{ or }j\in J, {i\sim j}}(Z^{(1)}_{ij}  ) ^2  \geq  t}$.
 Hence,    by the union bound and   \eqref{212} in   Lemma \ref{chi tail} with
 \begin{align*}
 \gamma = \varepsilon, \ L =     \Big (\frac{\bar{k}}{\bar{k}-1}  + \lambda \Big ) (1+\delta') \log n, \ m\leq   \frac{1+4\delta'}{\bar{\eta}^2} \frac{\log n}{\log \log n}
\end{align*}  (see \eqref{184}), recalling  $\bar{\eta} =\varepsilon^{1/4}$, for large enough $\delta$, for  sufficiently large $n$,
\begin{align}  \label{171}
 R_{1,2}&\leq   |V(C_1)|^{\left \lfloor {  \frac{1}{\bar{\eta}^2} }  \right \rfloor}   n^{-\frac{1}{2 } (\frac{\bar{k}}{\bar{k}-1}  + \lambda)  (1+\delta') + \frac{\varepsilon}{2} \frac{1+4\delta'}{\bar{\eta}^2} + \varepsilon  } \leq  n^{-\frac{1}{2 } \frac{\bar{k}}{\bar{k}-1}   (1+\delta') -\frac{1}{2 }  \lambda (1+\delta') + \frac{1}{2}   (1+4\delta') \varepsilon^{1/2} +2\varepsilon} .
\end{align}
Here, we used the fact $|V(C_1)|\leq (1+4\delta')\frac{\log n}{\log \log n}$ to bound the term   $ |V(C_1)|^{\left \lfloor {  \frac{1}{\bar{\eta}^2} }  \right \rfloor}  $  by $n^\varepsilon$.

Recall from \eqref{149} that  the size of maximal clique in $C_\ell$, $\ell=2,\cdots,m$, is  at most 3 under the event $\mathcal{F}_2 = \{\bar{k}\in \cM(\delta)\} \cap \cF_1$. Hence, by Proposition   \ref{prop} with
\begin{align*}
 \alpha = 1+\delta', \   k\leq 3, \  \gamma = \varepsilon, \  \eta = \varepsilon^{1/4},  \  c_1 = 1+4\delta', \  c_2  = \frac{2+4\delta'}{\varepsilon}, \   c_3 = 4\delta',
\end{align*} setting     $\xi :=  (2 \eta^2  +  8 \eta^4 \delta')^{1/4} $,     for sufficiently large $\delta$, 
\begin{align} \label{172} 
  R_{1,3}&\leq  n (  n^{-\frac{1+\delta'}{2 \xi^2 }+ (1+2\delta') + \varepsilon} + n^{-\frac{3}{4} (1- \xi)^2 (1+\delta')+  \frac{1}{2}   (1+4\delta')\varepsilon^{1/2} + \varepsilon} ) \nonumber \\
  & \leq Cn^{-\frac{3}{4} (1- \xi)^2 (1+\delta')+ \frac{1}{2}    (1+4\delta')\varepsilon^{1/2} + \varepsilon+1}  .
\end{align}
Here, we used the following comparison between exponents:  as $\delta \ri$,
\begin{align*}
\frac{1+\delta'}{2 \xi^2 } - (1+2\delta') - \varepsilon & = \Omega ( \delta^2), \\
 \frac{3}{4} (1- \xi)^2 (1+\delta') - \frac{1}{2}    (1+4\delta')\varepsilon^{1/2} - \varepsilon   & =   \Big( \frac{3}{4} + o_\delta(1)\Big)\delta .
\end{align*}    This follows from  $\varepsilon \leq \frac{1}{\delta^4}$ and the fact
\begin{align} \label{xi}
\xi  = O \Big (   \frac{1}{\delta^{1/2}} \Big) \qquad \text{as} \quad \delta\ri,
\end{align}
which is a consequence of  $\xi =  (2 \eta^2  +  8 \eta^4 \delta')^{1/4}  =  (2 \varepsilon^{1/2}  +  8 \varepsilon\delta')^{1/4}   $, $\varepsilon \leq \frac{1}{\delta^4}$ and the bound for $\delta'$ in \eqref{delta'bound}.

}

{

Thus, applying the above bounds for $R_{1,1}$, $R_{1,2}$ and $R_{1,3}$ (see \eqref{169}, \eqref{171} and \eqref{172} respectively) to \eqref{168}, for sufficiently large $\delta$, 
\begin{align} \label{173}
R_1 \leq  Cn^{-\frac{1}{2 } \frac{\bar{k}}{\bar{k}-1}   (1+\delta') -\frac{1}{2 }  \lambda (1+\delta') + \frac{1}{2}    (1+4\delta') \varepsilon^{1/2} +2\varepsilon} .
\end{align}
This follows from the fact that   for sufficiently large $\delta$, under the event $\bar{k}\in \cM(\delta)$, RHS of  \eqref{171} is the slowest decaying term among itself, \eqref{169} and  \eqref{172}. In fact, using     $\varepsilon \leq  \frac{1}{\delta^4}$ and the bound  for $\theta$, $\bar{k}$ and $\xi$ in \eqref{theta}, \eqref{185}  and  \eqref{xi} respectively, we have
\begin{align} \label{410}
\frac{1+\delta'}{2\theta^2} - (1+2\delta') - \varepsilon  &= \Omega (\delta^2) , \\
\label{411}
 \frac{1}{2 } \frac{\bar{k}}{\bar{k}-1}   (1+\delta') +\frac{1}{2 }  \lambda (1+\delta')   - \frac{1}{2}    (1+4\delta') \varepsilon^{1/2}- 2\varepsilon  
 & =\Big (\frac{1+\lambda}{2} + o_\delta(1)\Big) \delta, \\
 \label{412}
 \frac{3}{4} (1- \xi)^2 (1+\delta')   -  \frac{1}{2}  (1+4\delta')\varepsilon^{1/2} -  \varepsilon - 1   &=\Big (  \frac{3}{4}+o_\delta(1)\Big)\delta.
\end{align}  
{Since $\lambda\in (0,\frac{1}{100})$,  for large $\delta$, \eqref{411} is smaller than the other two terms.}
 }\\

\noindent
\textbf{Bounding $R_2$.}
For $\upsilon>0$ to be chosen later, we write
 \begin{align} \label{174}
 R_2&\leq    \mathbb{P}(S_1 \geq   \theta \sqrt{2(1+\delta') \log n} |X^{(1)})   \1_{\cF_0\cap \cF_2} \nonumber \\ 
  &+ \mathbb{P}\Big (S_2 \geq  (1+\upsilon)\Big  (x+ \frac{\bar{k}}{\bar{k}-1} y\Big ) \sqrt{2(1+\delta') \log n} |X^{(1)} \Big  )  \1_{\cF_0\cap \cF_2} \nonumber \\ 
 & + \mathbb{P}\Big (x< 1- \kappa , y <  \frac{\kappa}{2}, \nonumber \\
 &  (1-x-y)\max_{\ell\geq 2}  \lambda_1(Z^{(1)}_\ell) \geq \Big (1   - (1+\upsilon) \Big (x+ \frac{\bar{k}}{\bar{k}-1} y\Big )-\theta \Big ) \sqrt{2(1+\delta') \log n}|X^{(1)}  \Big )  \1_{\cF_0\cap \cF_2} \nonumber   \\
 &=:   R_{2,1}+  R_{2,2}+  R_{2,3}.
 \end{align}
{
We take $\upsilon>0$ such that for sufficiently large $\delta>0$ and small $\kappa>0$, under the event $\bar{k}\in \cM(\delta)$,  for  any $x< 1- \kappa $ and $ y <  \frac{\kappa}{2}$,
\begin{align} \label{181}
 1   - (1+\upsilon) \Big(x+ \frac{\bar{k}}{\bar{k}-1} y \Big)- \theta  \ge  1- (1+\upsilon)(x+y) -  \frac{1+\upsilon}{2(\bar{k}-1)} \kappa   - \theta
 > \frac{9}{10}(1-x-y).
\end{align}
Here, the last inequality follows from the  bound $x+y \leq   1- \frac{\kappa}{2}$ and the bound for $\theta$ in   \eqref{theta}.
}

By \eqref{169},  for sufficiently large $n$,
 \begin{align}   \label{170}
R_{2,1}\leq   n^{-\frac{1+\delta'}{2 \theta^2  } + (1+2\delta') +\varepsilon}   .
\end{align}
Note that for sufficiently large $\delta$, under the event $\bar{k}\in \cM(\delta)$, {by the bound for $\bar{k}$ in \eqref{185},} we have $
  2 (\frac{\bar{k}-1}{2\bar{k}} x^2 +xy + \frac{1}{3}y^2 )  <   \frac{\bar{k}-1}{\bar{k}}    (x+ \frac{\bar{k}}{\bar{k}-1} y)^2
$.
Thus, by \eqref{165},
 \begin{align*} 
S_2 \leq     \Big(  \frac{\bar{k}-1}{\bar{k}}  \Big)^{1/2}   \Big  (x+ \frac{\bar{k}}{\bar{k}-1} y\Big)  \Big  (\sum _{(i,j)\in B_2}  (Z^{(1)}_{ij}  ) ^2 \Big)^{1/2}  .
\end{align*}
{ Hence,  by the same arguments as in \eqref{1711} and \eqref{171} (apply  \eqref{212} in  Lemma \ref{chi tail} with $\gamma = \varepsilon$, $L =    \frac{\bar{k}}{\bar{k}-1}(1+\upsilon)^2  (1+\delta') \log n $ and $m\leq   \frac{1+4\delta'}{\bar{\eta}^2} \frac{\log n}{\log \log n}$),  for large enough $\delta$, for  sufficiently large $n$,
 \begin{align} \label{175}
  R_{2,2} & \leq  \mathbb{P}   \Big (   \sum_{i<j,  (i,j)\in B_2} (Z^{(1)}_{ij}  ) ^2  \geq   \frac{\bar{k}}{\bar{k}-1}(1+\upsilon)^2  (1+\delta') \log n \mid X^{(1)} \Big  )  \1_{\cF_0\cap \cF_2}   
   \nonumber \\
  & \leq  |V(C_1)|^{\left \lfloor {  \frac{1}{\bar{\eta}^2} }  \right \rfloor}   n^{-\frac{1}{2 } \frac{\bar{k}}{\bar{k}-1} (1+\upsilon)^2 (1+\delta') +  \frac{1}{2}   (1+4\delta')\varepsilon^{1/2}+ \varepsilon } \leq  n^{-\frac{1}{2 } \frac{\bar{k}}{\bar{k}-1}   (1+\delta') - \upsilon(1+\delta') + \frac{1}{2}   (1+4\delta')\varepsilon^{1/2}  + 2\varepsilon} .
 \end{align}

Next,
by \eqref{181},    
\begin{align*}
 R_{2,3}&\leq  \mathbb{P}\Big (     \max_{\ell=2,\cdots,m}   \lambda_1(Z^{(1)}_\ell)   \geq   \frac{9}{10} \sqrt{2(1+\delta') \log n}  \mid X^{(1)}  \Big )     {   \1_{\cF_0\cap \cF_2}    } .
\end{align*}
Since  the size of maximal clique in $C_\ell$, $\ell=2,\cdots,m$, is  at most 3 under the event $\mathcal{F}_2$, by the same argument as in \eqref{172} (apply  Proposition   \ref{prop} with $\alpha = (\frac{9}{10})^2 (1+\delta')$, $k\leq 3$, $\eta = \varepsilon^{1/4}$ and $\gamma = \varepsilon$),    for sufficiently large $\delta$,   
\begin{align}  \label{176}
  R_{2,3}&\leq  n (  n^{- (\frac{9}{10})^2\frac{1+\delta'}{2 \xi^2 }+ (1+2\delta')+ \varepsilon} + n^{-\frac{3}{4}  (\frac{9}{10})^2(1+\delta')+  \frac{1}{2}    (1+4\delta')\varepsilon^{1/2}+ \varepsilon} ) \nonumber \\
  & \leq Cn^{-\frac{3}{4}  (\frac{9}{10})^2(1+\delta')+  \frac{1}{2}   (1+4\delta')\varepsilon^{1/2}+ \varepsilon+1}.
\end{align} 
}
 Thus, applying \eqref{170}, \eqref{175} and \eqref{176} to \eqref{174},  for large enough $\delta$, for  sufficiently large $n$, 
\begin{align} \label{177}
R_2 \leq  Cn^{-\frac{1}{2 } \frac{\bar{k}}{\bar{k}-1}   (1+\delta') - \upsilon(1+\delta') + \frac{1}{2}   (1+4\delta')\varepsilon^{1/2}  + 2\varepsilon} .
\end{align}
{This follows from the fact that   for sufficiently large $\delta$, under the event $\bar{k}\in \cM(\delta)$, RHS of  \eqref{175} is the slowest decaying term among itself, \eqref{170} and  \eqref{176}. This can be verified by the similar argument as in \eqref{173}, combined with  the fact that
\begin{align*}
\frac{1}{2 } \frac{\bar{k}}{\bar{k}-1}   (1+\delta') + \upsilon(1+\delta') - \frac{1}{2}   (1+4\delta')\varepsilon^{1/2}  - 2\varepsilon = \Big(\frac{1}{2} + \upsilon + o_\delta(1)\Big) \delta
\end{align*} 
and $ \frac{1}{2} < \frac{3}{4}(\frac{9}{10})^2$. }

Therefore, using the bounds in \eqref{173} and  \eqref{177} in  \eqref{167}  we get that 
\begin{align} \label{178}
\mathbb{P}\Big(x< 1- \kappa &,  \sum_{1\leq i,j\leq n}  Z^{(1)}_{ij} v_iv_j   \geq  \sqrt{2(1+\delta') \log n} \mid X^{(1)}  \Big) \1_{\cF_0\cap 
\cF_2} \nonumber \\
& \leq Cn^{-\frac{1}{2 } \frac{\bar{k}}{\bar{k}-1}    (1+\delta')  - \min( \frac{1}{2}\lambda,\upsilon) (1+\delta') + \frac{1}{2}    (1+4\delta')\varepsilon^{1/2} + 2\varepsilon} .
\end{align}

\noindent
\textbf{Finishing the proof.}
{Recall that   under the event $\cF_0$, the number of non-tree components  is less than $\log n$. Hence, by \eqref{141} and \eqref{142},  
\begin{align} \label{182}
 \mathbb{P}&\Big( \sum_{1\leq i,j\leq n}  Z^{(1)}_{ij} v_iv_j   \geq  \sqrt{2(1+\delta') \log n}\mid X^{(1)}  \Big)\1_{\cF_0 } \nonumber \\
 & \leq   \mathbb{P}\Big(   \max_{\ell=1,\cdots,m} \lambda_1 (Z_\ell^{(1)})  \geq  \sqrt{2(1+\delta') \log n}\mid X^{(1)}  \Big)\1_{\cF_0 } \nonumber \\
 &\leq C (\log n) n^{-\frac{\bar{k}}{2(\bar{k}-1)} (1- \xi)^2 (1+\delta') +  \frac{1}{2}   (1+4\delta')\varepsilon^{1/2} +\varepsilon } + Cn\cdot   n^{- (1- \xi)^2 (1+\delta') +  \frac{1}{2}   (1+4\delta')\varepsilon^{1/2} +\varepsilon }
\end{align}
(recall that  $\xi = (2\varepsilon^{1/2} + 8\varepsilon \delta')^{1/4}$).
Using the bound for $\xi$ in \eqref{xi}  and $\varepsilon \leq  \frac{1}{\delta^4}$,
 in the case $\bar{k}\geq 3$,  for   large $\delta$, for sufficiently large $n$, \eqref{182} is bounded by
\begin{align} \label{1820}
   C  n^{-\frac{\bar{k}}{2(\bar{k}-1)} (1- \xi)^2 (1+\delta') + \frac{1}{2}  (1+4\delta')\varepsilon^{1/2} +2\varepsilon },
\end{align}
and  for $\bar{k}=2$,   \eqref{182} is bounded by
\begin{align} \label{1821}
C  n^{- (1- \xi)^2 (1+\delta') +  \frac{1}{2}  (1+4\delta')\varepsilon ^{1/2}+\varepsilon +1}.
\end{align}
}
Recalling $\cF_2 = \{\bar{k}\in \cM(\delta)\} \cap \cF_1$, we write
\begin{align} \label{179}
 \mathbb{P}&\Big(x<1-\kappa, \sum_{1\leq i,j\leq n}  Z^{(1)}_{ij} v_iv_j   \geq  \sqrt{2(1+\delta') \log n}\Big)  \nonumber \\
  &\leq  \sum_{k\in \cM(\delta)} \mathbb{E} \Big[ \mathbb{P}\Big(x<1-\kappa, \sum_{1\leq i,j\leq n}  Z^{(1)}_{ij} v_iv_j   \geq  \sqrt{2(1+\delta') \log n}\mid X^{(1)}  \Big) \1_{\bar{k} = k} \1_{\cF_0} \1_{\cF_2} \Big]  \nonumber \\
  & + \sum_{k\in \cM(\delta)} \mathbb{E} \Big[ \mathbb{P}\Big(\sum_{1\leq i,j\leq n}  Z^{(1)}_{ij} v_iv_j   \geq  \sqrt{2(1+\delta') \log n}\mid X^{(1)}  \Big) \1_{\bar{k} = k} \1_{\cF_0} \1_{\cF_1^c} \Big]    \nonumber \\
&    + \sum_{k\notin \cM(\delta)} \mathbb{E} \Big[ \mathbb{P}\Big( \sum_{1\leq i,j\leq n}  Z^{(1)}_{ij} v_iv_j   \geq  \sqrt{2(1+\delta') \log n}\mid X^{(1)}  \Big) \1_{\bar{k} = k} \1_{\cF_0}   \Big] +  \mathbb{P}(\cF_0^c )   .
\end{align} 
{
 Recalling  $\varepsilon \leq \frac{1}{\delta^4}$, 
by \eqref{144} and  \eqref{178},   the first term in \eqref{179} is bounded by
\begin{align}  \label{612}
C \sum_{k\in \cM(\delta)}  n^{-\frac{1}{2 } \frac{k}{k-1}    (1+\delta')  - \min( \frac{1}{2}\lambda,\upsilon) (1+\delta') +\frac{1}{2}   (1+4\delta')\varepsilon^{1/2}  + 2\varepsilon}  \frac{(d')^{{k \choose 2}}}{n^{ {k \choose 2} - k} } .
\end{align}
We use the argument in \eqref{600}-\eqref{145} to bound this quantity. The exponent in $n$ above is less than  
\begin{align*}
\Big[ -  {k \choose 2} +  k - \frac{k}{2(k-1)}  (1-\xi)^2  (1+\delta')   +\frac{1+4\delta'}{2}   \varepsilon^{1/2}  + 2\varepsilon\Big] - \min\Big( \frac{1}{2}\lambda,\upsilon\Big) (1+\delta').
\end{align*} 
Comparing this with the exponent in \eqref{600}, we notice the    additional term $\min(\frac{1}{2}\lambda,\upsilon)(1+\delta')$. Hence,
 recalling $\eta_1$ in $\eqref{145} $ can be chosen as $\eta_1  = 2r_\delta(\varepsilon)$ (see \eqref{602}),  \eqref{612} can be bounded by
\begin{align}
 Cn^{-\psi(\delta) -  \min(\frac{1}{2}\lambda,\upsilon) (1+\delta') + 2r_\delta(\varepsilon)}  .
\end{align}
Similarly,  using   \eqref{153} and   \eqref{1820},   the second term in \eqref{179} is bounded by
\begin{align}  \label{323}
C \sum_{k\in \cM(\delta)}    n^{-\frac{k}{2(k-1)} (1- \xi)^2 (1+\delta') +  \frac{1}{2}   (1+4\delta')\varepsilon^{1/2}  +2\varepsilon } \left( \frac{2h(\delta) + 2}{n} \right)^{ {k \choose 2} - k + 1}  \leq   Cn^{-\psi(\delta) - 1+ 2r_\delta(\varepsilon)} .
\end{align}
This follows from the fact that there is an additional $n^{-1}$ term arising from $ ( \frac{2h(\delta) + 2}{n} )^{ {k \choose 2} - k + 1} $. In addition,
by \eqref{144}, \eqref{1820} and \eqref{1821},     the third term in \eqref{179} is bounded by
\begin{align}  \label{324}
C    \sum_{k\geq 3, k\notin \cM(\delta)}    n^{-\frac{\bar{k}}{2(\bar{k}-1)} (1- \xi)^2 (1+\delta') + \frac{1}{2}   (1+4\delta')\varepsilon^{1/2}  +2\varepsilon }  \frac{(d')^{{k \choose 2}}}{n^{ {k \choose 2} - k} } & + C  n^{- (1- \xi)^2 (1+\delta') +  \frac{1}{2}   (1+4\delta')\varepsilon^{1/2}   +\varepsilon +1} \nonumber \\
  &\leq
  Cn^{-\psi(\delta) -c(\delta)+2 r_\delta(\varepsilon)}. 
\end{align}
Here, the additional term $c(\delta)$  comes from the fact that in the first term, the summation is taken only over $k\in \cM(\delta)^c$ and $\phi_\delta(k) \geq \psi(\delta) +c(\delta)$ for $k\in \cM(\delta)^c$  (see  \eqref{502} for details). The second term  can be absorbed in the constant $C$ since $\psi(\delta) = (\frac{1}{2}+o_\delta(1))\delta$ (see \eqref{103}).

Finally,  by \eqref{129}, the last term in \eqref{179} is bounded by $n^{-2\delta'}$. Hence,  applying the above bounds to \eqref{179},   for sufficiently large $\delta>0$, 
\begin{align} \label{17900}
\mathbb{P}&\Big(x<1-\kappa, \sum_{1\leq i,j\leq n}  Z^{(1)}_{ij} v_iv_j   \geq  \sqrt{2(1+\delta') \log n}\Big) \leq   C   n^{-\psi(\delta)  -c(\delta)  +2r_\delta(\varepsilon) }  .
\end{align}
Above, we used the fact that  for large enough $\delta$, $c(\delta) <1  <  \min(\frac{1}{2}\lambda,\upsilon) (1+\delta') $, which follows from the bound for $\delta'$ in \eqref{delta'bound}. Since $2r_\delta(\varepsilon) < c(\delta)$ (see \eqref{605}), 
applying  \eqref{163} and then Theorem \ref{main theorem 1}, for sufficiently large $\delta$,
 \begin{align} \label{195}
\lim_{n\ri} \mathbb{P}( x<1-\kappa  \mid    \cU_\delta )    = 0.
\end{align}
Therefore,  for sufficiently large $\delta$,
\begin{align}
\lim_{n\ri} \mathbb{P}(   \cA_1 \mid \cU_\delta) =1.
\end{align}

}
 
\qed
\\

\noindent
\textbf{Uniformity of eigenvector.} We will aim to show $\sum_{i<j, i, j\in K} (v_i^2 - v_j^2)^2$ is small from which the form of uniformity appearing in the theorem statement follows immediately. 

{We first recall the parameter $\theta$  defined in \eqref{161}.}
By \eqref{166}, setting $\rho : = 16\kappa$,
\begin{align} \label{188}
\mathbb{P}&\Big(  x\geq  1-\kappa,   \sum_{i<j, i, j\in K} (v_i^2 - v_j^2)^2>   \rho ,  \sum_{1\leq i,j\leq n}  Z^{(1)}_{ij} v_iv_j   \geq  \sqrt{2(1+\delta') \log n}\mid X^{(1)}  \Big)  \1_{\cF_0\cap \cF_2}   \nonumber \\
&\leq    \mathbb{P}(S_1 \geq   \theta \sqrt{2(1+\delta') \log n} \mid X^{(1)})  \1_{\cF_0\cap \cF_2} \nonumber   \\
  &+  \mathbb{P}\Big(x\geq  1- \kappa,   \sum_{i<j, i, j\in K}  (v_i^2 - v_j^2)^2>   \rho ,  S_2 \geq   (x+y-\theta ) \sqrt{2(1+\delta') \log n} \mid X^{(1)}  \Big)   \1_{\cF_0\cap \cF_2} \nonumber   \\ 
 & +  \mathbb{P}\Big( \max_{\ell=2,\cdots,m}\lambda(Z^{(1)}_\ell) \geq  \sqrt{2(1+\delta') \log n}\mid X^{(1)}  \Big)   \1_{\cF_0\cap \cF_2} .
\end{align}
Since the first and third terms were already estimated during the analysis in the first part, we now estimate the second term. Using the identity
{\begin{align*}
\frac{\bar{k}-1}{2\bar{k}} \Big(\sum_{i\in K} v_i^2\Big)^2  - \sum_{i<j, i, j\in K} v_i^2 v_j^2 =  \frac{1}{2\bar{k}} \sum_{i<j, i, j\in K} (v_i^2 - v_j^2)^2,
\end{align*}
under the event  $\cF_1$, we have an improvement of \eqref{140}:
\begin{align*}
    \sum _{(i,j)\in B_2} v_i^2 v_j^2 &\leq  \sum_{i\neq j, i,j\in K} v_i^2v_j^2   +2 \sum_{i\in K, j\in C_1 \setminus K} v_i^2v_j^2 + \sum_{i\neq j,  (i,j)\in  \overrightarrow{E(C_1\setminus K)} } v_i^2v_j^2  \\
    & \leq  2 \Big(\frac{\bar{k}-1}{2\bar{k}} x^2 +xy + \frac{1}{3}y^2 \Big) - \frac{1}{\bar{k}} \sum_{i<j, i,j\in K} (v_i^2 - v_j^2)^2 
\end{align*}
}
where we used the above identity to bound the first term on the RHS.
Thus, under the event $ \sum_{i<j, i, j\in K} (v_i^2 - v_j^2)^2> \rho$, using $x\leq 1$, we obtain the analog of
\eqref{165}:
\begin{align} \label{189}
S_2 \leq     \Big( \Big(\frac{\bar{k}-1}{\bar{k}}  - \frac{\rho}{\bar{k}} \Big) x^2 +2xy + \frac{2}{3}y^2 \Big)^{1/2}  \Big  (\sum _{(i,j)\in B_2}  (Z^{(1)}_{ij}  ) ^2 \Big)^{1/2}  .
\end{align}
To bound the above, we need the following technical inequality. For   sufficiently  large $\delta$,   under the event $\bar{k}\in \cM(\delta)$,  for $x\geq  1- \kappa $,
\begin{align}   \label{187}
  \Big(\frac{\bar{k}-1}{\bar{k}}  - \frac{\rho}{\bar{k}} \Big) x^2 +2xy + \frac{2}{3}y^2  <  \Big(  \frac{\bar{k}-1}{\bar{k}}  -  \frac{\rho}{\bar{2k}} \Big) (x+y-\theta)^2.
\end{align} 
In fact,  by rearranging, \eqref{187} holds  for sufficiently large $\delta$ if
\begin{align} \label{1871}
2\Big(\frac{1}{\bar{k}} + \frac{\rho}{2\bar{k}}\Big)xy + 2\Big( \frac{\bar{k}-1}{\bar{k}} -  \frac{\rho}{2\bar{k} }\Big) \theta (x+y) \leq  \frac{\rho}{2\bar{k}}x^2,
\end{align}
since using \eqref{185} we know the coefficient of $y^2$ on the RHS is at least that on the LHS. 

For $x\geq 1-\kappa$, we have $y\leq \kappa$ and thus
$2(\frac{1}{\bar{k}} + \frac{\rho}{2\bar{k}})xy \leq   \frac{\rho}{4\bar{k}}x^2$ holds for    small enough $\kappa>0$ (recall $\rho  = 16\kappa$). Also, by the bounds for $\theta$ and $\bar{k}$ in  \eqref{theta} and  \eqref{185}  respectively, under the event $\bar{k}\in \cM(\delta)$, we have $ 2( \frac{\bar{k}-1}{\bar{k}} -  \frac{\rho}{2\bar{k} }) \theta (x+y) \leq 2\theta \leq \frac{\rho}{4\bar{k}}x^2$. The previous two inequalities imply \eqref{1871} and thus \eqref{187}.

Hence, by \eqref{189} and \eqref{187}, and further using $ (\frac{\bar{k}-1}{\bar{k}}  -\frac{\rho}{2\bar{k}}  )^{-1}\ge \frac{\bar{k}}{\bar{k}-1}  +  \frac{\rho}{2\bar{k}}  $,   the second term in \eqref{188} is bounded by
\begin{align} \label{197}
 \mathbb{P} &  \Big (   \sum_{i<j,  (i,j)\in B_2} (Z^{(1)}_{ij}  ) ^2  \geq   \Big(   \frac{\bar{k}}{\bar{k}-1}  +  \frac{\rho}{2\bar{k}}   \Big)  (1+\delta') \log n \mid  X^{(1)}\Big  )  \1_{\cF_0\cap \cF_2}  .
\end{align}
As before,    by the union bound and  Lemma \ref{chi tail} with $\gamma = \varepsilon$, $L =     ( \frac{\bar{k}}{\bar{k}-1}  +  \frac{\rho}{2\bar{k}} ) (1+\delta') \log n $ and $m\leq   \frac{1+4\delta'}{\bar{\eta}^2} \frac{\log n}{\log \log n}$,   the above, and thus  the second term in \eqref{188}, is bounded by
 \begin{align}
   |V(C_1)|^{\left \lfloor {  \frac{1}{\bar{\eta}^2} }  \right \rfloor}   n^{-\frac{1}{2 } (\frac{\bar{k}}{\bar{k}-1}  + \frac{\rho}{2\bar{k}}   )  (1+\delta') + \frac{1}{2}   (1+4\delta')\varepsilon^{1/2} + \varepsilon  } \leq n^{-\frac{1}{2 } \frac{\bar{k}}{\bar{k}-1}    (1+\delta')  -\frac{\rho}{4\bar{k}}  (1+\delta') +  \frac{1}{2}   (1+4\delta')\varepsilon^{1/2}  + 2\varepsilon} .
 \end{align}
Since the first and last terms in \eqref{188} are bounded by \eqref{169} and  \eqref{172} respectively, one can deduce that
\begin{align} \label{321}
\mathbb{P}\Big(  x\geq  1-\kappa,   \sum_{i<j, i, j\in K} (v_i^2 - v_j^2)^2>   \rho ,&   \sum_{1\leq i,j\leq n}  Z^{(1)}_{ij} v_iv_j   \geq  \sqrt{2(1+\delta') \log n}|X^{(1)}  \Big)  \1_{\cF_0\cap \cF_2}   \nonumber \\
&\leq  C n^{-\frac{1}{2 } \frac{\bar{k}}{\bar{k}-1}    (1+\delta')  - \frac{\rho}{4\bar{k}}  (1+\delta') + \frac{1}{2}   (1+4\delta')\varepsilon^{1/2} + 2\varepsilon} .
\end{align}

Similarly as \eqref{179}, we write
\begin{align}  \label{320}
 \mathbb{P}&\Big( x\geq  1-\kappa,   \sum_{i<j, i, j\in K} (v_i^2 - v_j^2)^2>   \rho ,  \sum_{1\leq i,j\leq n}  Z^{(1)}_{ij} v_iv_j   \geq  \sqrt{2(1+\delta') \log n}\Big)  \nonumber \\
  &\leq  \sum_{k\in \cM(\delta)} \mathbb{E} \Big[ \mathbb{P}\Big( x\geq  1-\kappa,   \sum_{i<j, i, j\in K} (v_i^2 - v_j^2)^2>   \rho ,  \sum_{1\leq i,j\leq n}  Z^{(1)}_{ij} v_iv_j   \geq  \sqrt{2(1+\delta') \log n}\mid X^{(1)}  \Big) \1_{\bar{k} = k} \1_{\cF_0} \1_{\cF_2} \Big]  \nonumber \\
  & + \sum_{k\in \cM(\delta)} \mathbb{E} \Big[ \mathbb{P}\Big(\sum_{1\leq i,j\leq n}  Z^{(1)}_{ij} v_iv_j   \geq  \sqrt{2(1+\delta') \log n}\mid X^{(1)}  \Big) \1_{\bar{k} = k} \1_{\cF_0} \1_{\cF_1^c} \Big]    \nonumber \\
&    + \sum_{k\notin \cM(\delta)} \mathbb{E} \Big[ \mathbb{P}\Big( \sum_{1\leq i,j\leq n}  Z^{(1)}_{ij} v_iv_j   \geq  \sqrt{2(1+\delta') \log n}\mid X^{(1)}  \Big) \1_{\bar{k} = k} \1_{\cF_0}   \Big] +  \mathbb{P}(\cF_0^c )   .
\end{align} 
Using  \eqref{185} and  \eqref{321}, there exist a constant $c>0$ such that   the first term in  \eqref{320} is bounded by 
\begin{align}
C \sum_{k\in \cM(\delta)} n^{-\frac{1}{2 } \frac{\bar{k}}{\bar{k}-1}    (1+\delta')  - \frac{\rho}{4\bar{k}}  (1+\delta') + \frac{1}{2}   (1+4\delta')\varepsilon^{1/2} + 2\varepsilon}   \frac{(d')^{{k \choose 2}}}{n^{ {k \choose 2} - k} } \leq  Cn^{-\psi(\delta) - c \rho \delta^{2/3}+ 2r_\delta(\varepsilon) }  .
\end{align}
Other three terms in \eqref{320} can be bounded using   \eqref{323},  \eqref{324} and  \eqref{129} respectively. Hence, combining these together, using the fact that $c(\delta)<1<c\rho \delta^{2/3}$ for large $\delta$, we have
\begin{align*}  
\mathbb{P}&\Big( x\geq  1-\kappa,   \sum_{i<j, i, j\in K} (v_i^2 - v_j^2)^2>   \rho ,  \sum_{1\leq i,j\leq n}  Z^{(1)}_{ij} v_iv_j   \geq  \sqrt{2(1+\delta') \log n}\Big) \leq   C   n^{-\psi(\delta)  -c(\delta)+ 2r_\delta(\varepsilon)}  .
\end{align*}
Since $2r_\delta(\varepsilon) < c(\delta)$,
applying  \eqref{163} and then Theorem \ref{main theorem 1},
\begin{align*}
\lim_{n\ri} \mathbb{P}\Big( x \geq 1-\kappa,   \sum_{i<j, i, j\in K} (v_i^2 - v_j^2)^2>   \rho \mid \cU_\delta  \Big)=0,
\end{align*}
and thus by  \eqref{195},
\begin{align*}
\lim_{n\ri} \mathbb{P}\Big(x \geq 1-\kappa,  \sum_{i<j, i, j\in K} (v_i^2 - v_j^2)^2\leq    \rho \mid  \cU_\delta  \Big)=1.
\end{align*}

It is now straightforward to obtain the uniformity statement in the theorem from the smallness of $\sum_{i<j,i,j \in K}(v_i^2-v_j^2)^2.$ 
To see this, note that setting $S:= \sum_{i\in K} v_i^2$,
\begin{align} \label{340}
  \sum_{i\in K} \Big(v_i^2  - \frac{1}{|K|} S\Big )^2   
  &=\sum_{i\in K} v_i^4 - 2 \frac{S}{|K|} \sum_{i\in K} v_i^2 + \frac{1}{|K|}S^2   = \sum_{i\in K} v_i^4  - \frac{1}{|K|}S^2 \nonumber   \\
&= \frac{1}{|K|}\Big((K-1) \sum_{i\in K} v_i^4 - 2  \sum_{i<j, i, j\in K}  v_i^2 v_j^2\Big)  = \frac{1}{|K|} \sum_{i<j, i, j\in K} (v_i^2 - v_j^2)^2 .
\end{align}
Hence, recalling $\rho  = 16\kappa$, for sufficiently large $\delta$,
\begin{align} \label{355}
\lim_{n\ri} \mathbb{P}\Big( x \geq 1-\kappa,    \sum_{i\in K} \Big(v_i^2  - \frac{1}{|K|} \sum_{i\in K} v_i^2 \Big)^2    \leq     \frac{16\kappa}{|K|}    \mid   \cU_\delta  \Big)=1.
\end{align}
{
Using the inequality $(a+b)^2\leq 2(a^2+b^2)$, under the event
\begin{align*}
\Big\{ \sum_{i\in K} v_i^2 \geq 1-\kappa\Big \} \cap  \Big\{\sum_{i\in K} \Big(v_i^2  - \frac{1}{|K|} \sum_{i\in K} v_i^2 \Big)^2    \leq     \frac{16\kappa}{|K|}  \Big\},
\end{align*} 
we have
\begin{align} \label{402}
\sum_{i\in K}\Big( v_i^2  - \frac{1}{|K|}  \Big)^2 &\leq  2  \sum_{i\in K}\Big( v_i^2  - \frac{1}{|K|}\sum_{i\in K}v_i^2  \Big)^2  +2 |K| \Big (\frac{1}{|K|}\sum_{i\in K}v_i^2  - \frac{1}{|K|}\Big )^2 \nonumber  \\
& \leq  \frac{32\kappa}{|K|}  + \frac{2\kappa^2}{|K|} \leq \frac{40\kappa}{|K|} = : \frac{\kappa_0}{|K|}.
\end{align}
{Recalling that $K_X=K$,
the proof is complete.}
}

\section{Uniform largeness of Gaussian weights}\label{section 8}
We prove Theorem \ref{uniform gaussian} in this section.
The proof essentially proceeds by comparing the $\ell_1$ and $\ell_2$ norms of the Gaussian variables on the edges of the clique $K_X$ by obtaining sharp estimates on each of them. The final statement then can be deduced from a quantitative version of the Cauchy-Schwarz inequality. However, as the statement of the theorem indicates, we will end up working with a set $T$ slightly smaller than $K_X.$ 
Implementing the strategy involves a few steps and in particular relies on Theorem \ref{eigenvectorloc} which is the reason we proved the latter first. \\

\noindent
\textbf{Sum of squares of the Gaussian weights.} 
{We use the same notations as in Section \ref{section 7}. Also, as in the beginning of the proof of Theorem \ref{eigenvectorloc}, we assume that  the maximal cliques $K:= K_{X^{(1)}}$ and $K_X$ are unique and equal.}

Setting $\rho := 16\kappa$,
 similarly as \eqref{187},  for   sufficiently large $\delta$, under the event $\bar{k}\in \cM(\delta)$, for  $x\geq 1-\kappa
$,
\begin{align} \label{201}
\frac{\bar{k}-1}{\bar{k}} x^2 +2xy + \frac{2}{3}y^2 \leq  \Big( \frac{\bar{k}-1}{\bar{k}} + \frac{\rho}{\bar{k}} \Big) (x+y-\theta)^2.
\end{align}
Using the above and \eqref{165},
\begin{align}  \label{204}
S_2 \leq      \Big( \frac{\bar{k}-1}{\bar{k}} + \frac{\rho}{\bar{k}}  \Big)^{1/2}   (x+y-\theta)    \Big  (\sum _{(i,j)\in B_2}  (Z^{(1)}_{ij}  ) ^2 \Big)^{1/2},
\end{align}
where $S_1$ and $S_2$ were defined in \eqref{310}.
We now define an event guaranteeing a sharp behavior of the $\ell_2$ norm of the Gaussian variables on the edges in $B_2$ where the latter was defined below \eqref{lowvalues12},
\begin{align} \label{a1}
\cA_3 := \Big\{ 2  \Big( \frac{\bar{k}}{\bar{k}-1} - \frac{\rho}{\bar{k}}  \Big) (1+\delta') \log n\leq  \sum _{(i,j)\in B_2}  (Z^{(1)}_{ij}  ) ^2  \leq  2   \Big( \frac{\bar{k}}{\bar{k}-1} +  \frac{\rho}{\bar{k}}  \Big) (1+\delta') \log n\Big\}.
\end{align} 

Thus we have
\begin{align}   \label{203}
\mathbb{P}&\Big(  \cA_3^c,   x \geq 1-\kappa, \sum_{1\leq i,j\leq n}  Z^{(1)}_{ij} v_iv_j   \geq  \sqrt{2(1+\delta') \log n}\mid X^{(1)}  \Big)  \1_{\cF_0\cap \cF_2}   \nonumber \\
&\leq    \mathbb{P}(S_1 \geq   \theta \sqrt{2(1+\delta') \log n} \mid X^{(1)})  \1_{\cF_0\cap \cF_2} \nonumber   \\
  &+  \mathbb{P}\Big( \cA_3^c,  x \geq 1-\kappa,   S_2 \geq   (x+y-\theta ) \sqrt{2(1+\delta') \log n} \mid X^{(1)}  \Big)   \1_{\cF_0\cap \cF_2} \nonumber   \\ 
 & +  \mathbb{P}\Big(  \max_{\ell=2,\cdots,m}\lambda(Z^{(1)}_\ell) \geq  \sqrt{2(1+\delta') \log n}\mid X^{(1)}  \Big)   \1_{\cF_0\cap \cF_2} .
\end{align}
 Since the  first and last terms above can be bounded using  \eqref{169} and \eqref{172} respectively, we  only bound
the second term.

\noindent
\textit{-Bounding the second term:}
  Using $(\frac{\bar{k}-1}{\bar{k}} + \frac{\rho}{\bar{k}} )^{-1}\geq  \frac{\bar{k}}{\bar{k}-1} - \frac{\rho}{\bar{k}} $, 
\begin{align} \label{214}
\mathbb{P}&\Big( \cA_3^c, x \geq 1-\kappa,   S_2 \geq   (x+y-\theta ) \sqrt{2(1+\delta') \log n} \mid  X^{(1)}  \Big)   \1_{\cF_0\cap \cF_2} \nonumber \\
&\overset{\eqref{204}}{\leq} 
 \mathbb{P}\Big(\cA_3^c,   \sum _{(i,j)\in B_2}  (Z^{(1)}_{ij}  ) ^2  \geq   2\Big(  \frac{\bar{k}}{\bar{k}-1} -  \frac{\rho}{\bar{k}}  \Big)  (1+\delta') \log n \mid X^{(1)}  \Big)   \1_{\cF_0\cap \cF_2}  \nonumber \\
 &\leq   \mathbb{P}\Big( \sum _{i<j,  (i,j)\in B_2}  (Z^{(1)}_{ij}  ) ^2  \geq   \Big(  \frac{\bar{k}}{\bar{k}-1} + \frac{\rho}{\bar{k}} \Big)  (1+\delta') \log n \mid X^{(1)}  \Big)   \1_{\cF_0\cap \cF_2},
\end{align}
where the last inequality follows from the definition of $\cA_3$.
As before,  by union bound and \eqref{212} in  Lemma \ref{chi tail} with $\gamma = \varepsilon$, $L =      (\frac{\bar{k}}{\bar{k}-1} + \frac{\rho}{\bar{k}}  ) (1+\delta') \log n  $ and $m\leq   \frac{1+4\delta'}{\bar{\eta}^2} \frac{\log n}{\log \log n}$ (using the bound on $|B_2|$ in \eqref{184}), for sufficiently large $n$, the above, and thus the second term in \eqref{203}, is bounded by
\begin{align} \label{215}
  |V(C_1)|^{\left \lfloor {  \frac{1}{\bar{\eta}^2} }  \right \rfloor}  n^{-\frac{1}{2 }  (\frac{\bar{k}}{\bar{k}-1} + \frac{\rho}{\bar{k}} )  (1+\delta') +  \frac{1}{2}   (1+4\delta')\varepsilon^{1/2}  + \varepsilon }   \leq  n^{-\frac{1}{2 }  \frac{\bar{k}}{\bar{k}-1}  (1+\delta') -  \frac{\rho}{2\bar{k}} (1+\delta') +  \frac{1}{2}   (1+4\delta')\varepsilon^{1/2}  +2 \varepsilon } .
\end{align}

{
\noindent
\textit{-Combining altogether:}
}
As mentioned above, the first and last terms in \eqref{203} can be bounded using   \eqref{169} and \eqref{172} respectively. Hence, combining these together,
\begin{align} \label{403}
\mathbb{P}\Big(  \cA_3^c,   x  \geq 1-\kappa, & \sum_{1\leq i,j\leq n}  Z^{(1)}_{ij} v_iv_j   \geq  \sqrt{2(1+\delta') \log n} \mid X^{(1)}  \Big)  \1_{\cF_0\cap \cF_2}  \nonumber \\
&    \leq   Cn^{-\frac{1}{2 }  \frac{\bar{k}}{\bar{k}-1}  (1+\delta') -  \frac{\rho}{2\bar{k}} (1+\delta') +  \frac{1}{2}   (1+4\delta')\varepsilon^{1/2}  +2 \varepsilon } .
\end{align}
{This follows from the fact that    for sufficiently   large $\delta$, under the event $\cF_0\cap \cF_2$,  \eqref{215} is the slowest decaying term among itself, \eqref{169} and  \eqref{172}. This follows from \eqref{410} and \eqref{412} and observing that $\varepsilon \leq  \frac{1}{\delta^4}$ and the bound  for    $\bar{k}$ in  \eqref{185} together,  under the event $\bar{k}\in \cM(\delta)$, implies $$\frac{1}{2 }  \frac{\bar{k}}{\bar{k}-1}  (1+\delta') +  \frac{\rho}{2\bar{k}} (1+\delta') -  \frac{1}{2}   (1+4\delta')\varepsilon^{1/2}  -  2 \varepsilon   =  \Big(\frac{1}{2} + o_\delta(1) \Big) \delta.$$ 
}

Similarly as in \eqref{179}, we write
\begin{align}   \label{404}
 \mathbb{P}&\Big( \cA_3^c  ,  x \geq 1-\kappa,  \sum_{1\leq i,j\leq n}  Z^{(1)}_{ij} v_iv_j   \geq  \sqrt{2(1+\delta') \log n}\Big)  \nonumber \\
  &\leq  \sum_{k\in \cM(\delta)} \mathbb{E} \Big[ \mathbb{P}\Big( \cA_3^c,  x  \geq 1-\kappa,  \sum_{1\leq i,j\leq n}  Z^{(1)}_{ij} v_iv_j   \geq  \sqrt{2(1+\delta') \log n}\mid X^{(1)}  \Big) \1_{\bar{k} = k} \1_{\cF_0} \1_{\cF_2} \Big]  \nonumber \\
  & + \sum_{k\in \cM(\delta)} \mathbb{E} \Big[ \mathbb{P}\Big(  \sum_{1\leq i,j\leq n}  Z^{(1)}_{ij} v_iv_j   \geq  \sqrt{2(1+\delta') \log n}\mid X^{(1)}  \Big) \1_{\bar{k} = k} \1_{\cF_0} \1_{\cF_1^c} \Big]    \nonumber \\
&    + \sum_{k\notin \cM(\delta)} \mathbb{E} \Big[ \mathbb{P}\Big( \sum_{1\leq i,j\leq n}  Z^{(1)}_{ij} v_iv_j   \geq  \sqrt{2(1+\delta') \log n}\mid X^{(1)}  \Big) \1_{\bar{k} = k} \1_{\cF_0}   \Big] +  \mathbb{P}(\cF_0^c ).
\end{align} 
{
First,
as in  \eqref{612}, one can bound the first term above using \eqref{403}. In fact,
using the bound for $\delta'$ and $\bar{k}$ in \eqref{delta'bound} and \eqref{185} respectively,  under $\bar{k}\in \cM(\delta)$, $\frac{\rho}{2\bar{k}}(1+\delta')\geq c\delta^{2/3}$  for some $c>0$. Thus,  for large $\delta$,  the first term  in  \eqref{404}  can be bounded by  $ Cn^{- \psi(\delta)  - c'\delta^{2/3}  } $ for some $c'<c$.  Combining this with  the bounds for other three terms, previously obtained in \eqref{323}, \eqref{324} and \eqref{129} respectively, using the fact that  $c(\delta)<1<c'\delta^{2/3}$ for large $\delta$,   \eqref{404}  is bounded by  $  
   C   n^{-\psi(\delta)  -c(\delta)  +2r_\delta(\varepsilon)}  .$
   }
Hence, using  \eqref{163}, for large $\delta$,
\begin{align}  
\mathbb{P}(  \cA_3^c  ,  x\geq 1-\kappa, \lambda_1 \geq \sqrt{2(1+\delta) \log n}) \leq  C   n^{-\psi(\delta)  - c(\delta)+ 2r_\delta(\varepsilon)}  .
\end{align}
Since  $2r_\delta(\varepsilon)  < c(\delta)$,  combined with Theorem \ref{main theorem 1} and \eqref{195},   
  for sufficiently large $\delta$,
\begin{align} \label{231}
\lim_{n\ri} \mathbb{P}\Big(\cA_3   \mid   \cU_\delta  \Big)=1.
\end{align}

\noindent
\textbf{Sum of absolute values of Gaussian weights.} 
We now estimate the sum of absolute values of $Z^{(1)}_{ij} $. 
Defining 
\begin{align}
\mathcal{A}' := \Big\{ \sum_{i\in K}\Big( v_i^2  - \frac{1}{\bar{k}}  \Big)^2      \leq     \frac{\kappa_0}{\bar{k}}\Big\}
\end{align}  
(recall $\kappa_0 = 40\kappa  $, see \eqref{402}),
since $\bar{k} = |K|$,  by \eqref{355} and \eqref{402},
\begin{align} \label{238}
\lim_{n\ri} \mathbb{P}\Big(\cA'   \mid   \cU_\delta \Big)=1. 
\end{align}
Recalling  that $K=K_X$ with probability going to one conditionally on $\cU_\delta,$ the events $\cA_2$ (from the statement of the theorem) and $\cA'$ are essentially the same. 

We now define the set of vertices $T$ appearing in the statement of the theorem,
\begin{align*}
T :=\Big \{i \in K: \Big\vert v_i^2 - \frac{1}{\bar{k}}\Big\vert < \frac{\kappa_0^{1/4}}{\bar{k}}\Big\}.
\end{align*}
Then, by \eqref{185},  for sufficiently large $\delta$, under the event $|\bar{k} - h(\delta)| \leq 1$,
\begin{align} \label{361}
i\neq j, i,j\in T \quad \text{implies} \quad   (i,j)\in B_2.
\end{align} This is because for $i\in T$ and large $\delta$, $v_i^2 \geq (1-\kappa_0^{1/4}) \frac{1}{\bar{k}} \overset{\eqref{185}}{\geq}   \frac{c}{\delta^{1/3}} > \frac{1}{\delta^2} \geq  \bar{\eta}^2$ where the final inequality is by our choice of $\bar{\eta}$ in \eqref{192}.
We now write
\begin{align} \label{311}
S_2 =  \sum_{(i,j)\in B_2} Z^{(1)}_{ij} v_iv_j  =     \sum_{i  \ \text{or} \ j \in T^c, (i,j)\in B_2} Z^{(1)}_{ij} v_iv_j +  \sum_{i,j\in T, (i,j)\in B_2} Z^{(1)}_{ij} v_iv_j   =:  S_{21} +  S_{22}
\end{align}
(see \eqref{310} for the definition of $S_2$).
By Cauchy-Schwarz inequality,
under the  event $\cA'$,
\begin{align} \label{221}
 \sum_{i\in K} \Big\vert v_i^2  - \frac{1}{\bar{k}}   \Big\vert     \leq     \kappa_0^{1/2}.
\end{align}
Thus, under the event $\cA'$,
\begin{align} \label{230}
{|T^c\cap K|} \leq \kappa_0^{1/4}\bar{k}, \qquad |T| \geq (1-  \kappa_0^{1/4})\bar{k}.
\end{align}
Note that \eqref{221} implies
$\sum_{v_i^2 \geq   \frac{1}{\bar{k}} (1 +\kappa_0^{1/4})    }  (v_i^2   - \frac{1}{\bar{k}}) \leq  \kappa_0^{1/2}$, and thus under the event $\cA'$,
\begin{align*}
\sum_{i\in T^c} v_i^2  = \sum_{v_i^2 \geq   \frac{1}{\bar{k}} (1 +\kappa_0^{1/4})    } v_i^2 + \sum_{v_i^2 \leq   \frac{1}{\bar{k}}(1-\kappa_0^{1/4})   } v_i^2    \leq \Big ( \kappa_0^{1/2} +  \frac{1}{\bar{k}} \kappa_0^{1/4}\bar{k}   \Big) +  \frac{1}{\bar{k}}    \kappa_0^{1/4}\bar{k}   =  \kappa_0^{1/2}+   2\kappa_0^{1/4}  .
\end{align*}
Hence, under  the event $\cA'$, 
\begin{align} \label{440}
\sum_{i  \ \text{or} \ j \in T^c, (i,j)\in B_2}  v_i^2v_j^2 \leq 2   \Big(  \sum_{i\in T^c} v_i^2  \Big)\Big(  \sum_{j\in C_1 } v_j^2 \Big)   \leq  2\kappa_0^{1/2}+   4\kappa_0^{1/4}   =: \kappa'^2,
\end{align}
and thus  
\begin{align} \label{224}
 S_{21} \leq \Big  (\sum_{i  \ \text{or} \ j \in T^c, (i,j)\in B_2} (Z^{(1)}_{ij})^2 \Big  )  ^{1/2} \Big (\sum_{i  \ \text{or} \ j \in T^c, (i,j)\in B_2}  v_i^2v_j^2 \Big)  ^{1/2}\leq  \kappa' \Big(\sum_{i  \ \text{or} \ j \in T^c, (i,j)\in B_2} (Z^{(1)}_{ij})^2 \Big )^{1/2}   .
\end{align}
In addition, {using the fact that $v_i^2 < \frac{1}{\bar{k}}(1+\kappa_0^{1/4})$ for $i\in T$,    }
\begin{align} \label{225}
| S_{22} |  \leq    \frac{1}{\bar{k}} (1+ \kappa_0^{1/4} ) \sum_{i,j\in T, (i,j)\in B_2} | Z^{(1)}_{ij} |   \leq \frac{1}{\bar{k}} (1+ \kappa_0^{1/4} ) \sum_{i\neq j,i,j\in T} | Z^{(1)}_{ij} |   .
\end{align}
Now, we define the  following event analogous to $\cA_3,$ but for the $\ell_1$ norm,
\begin{align} \label{a4}
\cA_4 :=\Big \{ \bar{k} (1-3 \kappa^{1/4} ) \sqrt{2(1 + \delta')\log n} \leq \sum_{i\neq j, i,j\in T} | Z^{(1)}_{ij} |  \leq  \bar{k} (1+3 \kappa^{1/4}) \sqrt{2(1 + \delta')\log n} \Big \}.
\end{align}
Now using the decomposition in \eqref{166} and further using \eqref{311},
\begin{align*}
 \sum_{1\leq i,j\leq n} Z^{(1)}_{ij} v_iv_j   \leq S_1   +  S_{21}  +  S_{22}+   (1-x-y)   \max_{\ell=2,\cdots,m}\lambda(Z^{(1)}_\ell),
\end{align*}
we write
\begin{align}   \label{205}
\mathbb{P}&\Big(  \cA_4^c, \cA', x \geq 1-\kappa,   \sum_{1\leq i,j\leq n}  Z^{(1)}_{ij} v_iv_j   \geq  \sqrt{2(1+\delta') \log n}\mid X^{(1)}  \Big)  \1_{\cF_0\cap \cF_2}   \nonumber \\
&\leq    \mathbb{P}(S_1 \geq   \theta \sqrt{2(1+\delta') \log n} |X^{(1)})  \1_{\cF_0\cap \cF_2}  +  \mathbb{P}(  \cA', S_{21} \geq   \sqrt{\kappa'} \sqrt{2(1+\delta') \log n} \mid X^{(1)})  \1_{\cF_0\cap \cF_2}  \nonumber   \\
  &+  \mathbb{P}\Big( \cA_4^c, x \geq 1-\kappa,   S_{22} \geq   (x+y-\theta - \sqrt{\kappa'} ) \sqrt{2(1+\delta') \log n} \mid X^{(1)}  \Big)   \1_{\cF_0\cap \cF_2} \nonumber   \\ 
 & +  \mathbb{P}\Big(   \max_{\ell=2,\cdots,m}\lambda(Z^{(1)}_\ell) \geq  \sqrt{2(1+\delta') \log n}\mid X^{(1)}  \Big)   \1_{\cF_0\cap \cF_2} 
\end{align}
 (recall that $\kappa'$ is defined in  \eqref{440}). Since we already have estimates for the first and last terms above, we  only focus on the second and third terms. \\

\noindent
\textit{-Bounding the second term:}
By \eqref{224},
\begin{align} \label{226}
\mathbb{P}&( \cA', S_{21} \geq   \sqrt{\kappa'} \sqrt{2(1+\delta') \log n} \mid X^{(1)})  \1_{\cF_0\cap \cF_2} \nonumber \\
&\leq  \mathbb{P}\Big(  \cA', \sum_{i< j,\,\, i \text{ or }  j \in T^c,(i,j)\in B_2 } (Z^{(1)}_{ij})^2 \geq  \frac{1}{\kappa'}(1+\delta')\log n\mid X^{(1)}\Big ) \1_{\cF_0\cap \cF_2}   .
\end{align}

Note that  by  \eqref{212} in   Lemma \ref{chi tail} with 
\begin{align*}
 \gamma = \varepsilon, \ L =  \frac{1}{\kappa'}(1+\delta')\log n, \ m\leq  4 \kappa^{1/4}\bar{k}   (1+4\delta')\frac{\log n}{\log \log  n}
\end{align*} (see \eqref{123}),  for sufficiently large $n$,
 the quantity \eqref{226}, and thus  the second term in \eqref{205}, is bounded by 
 \begin{align} \label{241}
 |V(C_1)|^{  \left \lfloor {4 \kappa^{1/4}\bar{k}  }   \right  \rfloor     }     n^{-\frac{1}{2 \kappa' }       (1+\delta') +  \frac{1}{2}   4 \kappa^{1/4}\bar{k}  (1+4\delta')\varepsilon + \varepsilon }   \leq   n^{-\frac{1}{2 \kappa' }  (1+\delta') +  2 \kappa^{1/4}\bar{k}   (1+4\delta')\varepsilon + 2\varepsilon }   .
 \end{align}
The above inequality follows from the bound for $|V(C_1)|$ in \eqref{124} and observing that   
  $$\Big( \frac{2+4\delta'}{\varepsilon} \frac{\log n}{\log \log n}\Big)^{  \left \lfloor {4 \kappa^{1/4}\bar{k}  }   \right  \rfloor     }   \overset{\eqref{185}}{\leq}     \Big ( \frac{2+4\delta'}{\varepsilon} \frac{\log n}{\log \log n}\Big)^{c\delta^{1/3}} \leq     n^\varepsilon $$ for large $n$ ($c>0$ is a constant depending on $\kappa$). 
The first factor in \eqref{241}, as several times before, appears due to a union bound over all possible choices {of $T^c\cap K.$} \\

 \noindent
\textit{-Bounding the third term:}
 Note that for  sufficiently small $\kappa>0$, for    large enough $\delta$ and $x\geq 1-\kappa$,    
\begin{align} \label{222}
1-3 \kappa^{1/4} \leq \frac{x+y-\theta- \sqrt{\kappa'} }{1+ \kappa_0^{1/4}}
\end{align}
(recall  $\kappa_0 = 40\kappa $).
In fact,  \eqref{222} holds if $(1-3 \kappa^{1/4})(1+  \kappa_0^{1/4}) \leq 1-\kappa - \theta- \sqrt{\kappa'}  $ for sufficiently large $\delta$ and small $\kappa>0$, which follows from the bound for $\theta$ in  \eqref{theta}.

Hence, using \eqref{225} and \eqref{222}, recalling the definition of $\cA_4$ in \eqref{a4},
the third term in \eqref{205} can be controlled by
\begin{align} \label{227}
 \mathbb{P}&\Big( \cA_4^c, x \geq 1-\kappa,   S_{22} \geq   (x+y-\theta   - \sqrt{\kappa'} ) \sqrt{2(1+\delta') \log n} \mid X^{(1)}  \Big)   \1_{\cF_0\cap \cF_2}  \nonumber  \\
 &\leq 
 \mathbb{P}\Big(\cA_4^c,   \sum _{i\neq j, i,j\in T }  |Z^{(1)}_{ij}   |  \geq  \bar{k} (1-3 \kappa^{1/4} ) \sqrt{2(1+\delta') \log n}  \mid X^{(1)}  \Big)   \1_{\cF_0\cap \cF_2}  \nonumber  \\
 &\leq   \mathbb{P}\Big(  \sum _{i\neq j, i,j\in T  }  |Z^{(1)}_{ij}   |   > \bar{k}(1+3 \kappa^{1/4}) \sqrt{2(1+\delta') \log n}    \mid X^{(1)}  \Big)   \1_{\cF_0\cap \cF_2}  \nonumber  \\
  &\leq   \mathbb{P}\Big(  \sum _{i<j,  i,j\in T }  (Z^{(1)}_{ij}   )^2   >  \frac{\bar{k}^2}{ \bar{k}(\bar{k}-1
 ) }  (1+3 \kappa^{1/4}) ^2(1+\delta') \log n    \mid X^{(1)}  \Big)   \1_{\cF_0\cap \cF_2} .
\end{align}
The second inequality follows from the definition of $\cA_4,$ (similar to \eqref{214}), while 
in the third inequality above, we used the Cauchy-Schwarz inequality and the fact $|T|\leq \bar{k}$. 
Hence, by the union bound and   \eqref{212} in    Lemma \ref{chi tail} with $\gamma = \varepsilon$,  $L =\frac{\bar{k}}{  \bar{k}-1
  }  (1+3 \kappa^{1/4}) ^2(1+\delta') \log n  $ and $m\leq   \bar{k}^2  $,  for sufficiently large $n$,
 the quantity \eqref{227}, and thus  the third term in \eqref{205}, is bounded by 
\begin{align} \label{242}
  |V(C_1)|^{\bar{k}  }     n^{-\frac{1}{2 }  \frac{\bar{k}}{\bar{k}-1} (1+3 \kappa^{1/4})  ^2 (1+\delta')  + \varepsilon }   \leq   n^{-\frac{1}{2 }  \frac{\bar{k}}{\bar{k}-1} (1+3 \kappa^{1/4}) ^2  (1+\delta') + 2\varepsilon }   .
\end{align}
{Here, we used the  bound for $|V(C_1)|$ in \eqref{124} and the upper bound for $\bar{k}$ in \eqref{185} under the event $\bar{k}\in \cM(\delta)$ .}\\

\textit{-Combining altogether:}
As mentioned already, the first and the last terms in \eqref{205} can be bounded by \eqref{169} and \eqref{172}  respectively. Thus, combining these with \eqref{241} and \eqref{242}, for sufficiently small $\kappa>0$ and  large $\delta$, for large enough $n$,
\begin{align} \label{405}
\mathbb{P}\Big(  \cA_4^c, \cA',   x \geq 1-\kappa,   & \sum_{1\leq i,j\leq n}  Z^{(1)}_{ij} v_iv_j   \geq  \sqrt{2(1+\delta') \log n}\mid X^{(1)}  \Big)  \1_{\cF_0\cap \cF_2} \nonumber \\
&\leq   Cn^{-\frac{1}{2 }  \frac{\bar{k}}{\bar{k}-1} (1+3 \kappa^{1/4})^2   (1+\delta') + 2\varepsilon } \leq   Cn^{-\frac{1}{2 }  \frac{\bar{k}}{\bar{k}-1}    (1+\delta') - 3    \kappa^{1/4}   \delta'  + 2\varepsilon }    .
\end{align}
{This follows from the fact that    for sufficiently small $\kappa>0$ and  large $\delta$, under the event $\cF_0\cap \cF_2$,  \eqref{242} is the slowest decaying term among itself, \eqref{169}, \eqref{172} and \eqref{241}. In fact,  using    $\varepsilon \leq  \frac{1}{\delta^4}$ and the bound  for    $\bar{k}$ in  \eqref{185},   under the event $\bar{k}\in \cM(\delta)$, 
\begin{align*}
\frac{1}{2 \kappa' }  (1+\delta') - 2 \kappa^{1/4}\bar{k}     (1+4\delta')\varepsilon - 2\varepsilon  &=\Big  (\frac{1}{2\kappa'} + o_\delta(1) \Big ) \delta, \\ 
 \frac{1}{2 }  \frac{\bar{k}}{\bar{k}-1} (1+3 \kappa^{1/4})^2   (1+\delta') - 2\varepsilon &= \Big ( \frac{1}{2 }  (1+3 \kappa^{1/4})^2  +o_\delta (1)\Big ) \delta   .
\end{align*}
 Hence, recalling  the definition of  $\kappa'$ in \eqref{440}, for  small  $\kappa>0$ and large $\delta$,   the  quantity   \eqref{242} slowly decays than \eqref{241}.  Also, by comparing the above asymptotic  with \eqref{410} and \eqref{412}, one can deduce that  the quantity   \eqref{242} slowly decays than   \eqref{169} and  \eqref{172} for small $\kappa$ and  large $\delta$.

 Thus, by proceeding  as  in  \eqref{179},   for sufficiently large $\delta$,  
\begin{align*}
\mathbb{P}\Big(  \cA_4^c, \cA',  & x \geq 1-\kappa,   \sum_{1\leq i,j\leq n}  Z^{(1)}_{ij} v_iv_j   \geq  \sqrt{2(1+\delta') \log n}   \Big)   \leq   n^{- \psi(\delta) - c(\delta) +2r_\delta(\varepsilon)}   .
\end{align*}
  In fact, one can bound this  quantity by the sum of four quantities via the argument of  \eqref{179}. Using  \eqref{405} and the bound for $\bar{k}$ in \eqref{185}, for large $\delta$, the corresponding first term in  \eqref{179} can be bounded by $Cn^{-\psi(\delta)  - c\delta  }$ for some $c>0$, and other three terms  can be bounded by \eqref{323}, \eqref{324} and \eqref{129} respectively.  Combining these together, using the fact that  $c(\delta)<1<c\delta$ for large $\delta$,      we obtain the above inequality. 
  }
Applying \eqref{163} and then  Theorem \ref{main theorem 1},
\begin{align*}
\lim_{n\ri} \mathbb{P}\Big(  \cA_4^c,  \cA',  & x \geq 1-\kappa \mid  \cU_\delta  \Big)   =0 .
\end{align*} 
Hence,
by  \eqref{195} and  \eqref{238},
\begin{align} \label{228}
\lim_{n\ri} \mathbb{P}\Big(\cA_4    \mid \cU_\delta\Big)=1.
\end{align}

\noindent 
\textbf{Finishing the proof.}  Finally,    using  \eqref{231} and \eqref{228}, we finish the proof.
Define the event
\begin{align*}
\cA_5 :=  \{  \bar{k}\in \cM(\delta)\}.
\end{align*}
{ By \eqref{909}, recalling $\cA_5 \subset \cF_2$,
\begin{align} \label{401}
\lim_{n\ri} \mathbb{P}\Big(\cA_5 \mid \cU_\delta \Big)=1. 
\end{align}
Now, define the event
\begin{align*}
\cA_6 :=\cA'\cap \cA_3\cap \cA_4    \cap \cA_5.
\end{align*}
Since  $\cA'$,  $\cA_3$,  $\cA_4$ and  $\cA_5$ are typical events conditioned on $\cU_\delta$ (see  \eqref{238},    \eqref{231},  \eqref{228} and \eqref{401} respectively), }
\begin{align}  \label{229}
\lim_{n\ri} \mathbb{P}\Big(\cA_6\mid \cU_\delta \Big)=1.
\end{align}

We next verify that the event $\cA_6$ implies
the desired uniformity of Gaussians claimed in the statement of the theorem. As indicated earlier, the proof involves technical manipulations involving the Cauchy-Schwarz inequality to relate the $\ell_1$ and $\ell_2$ norms. 

For the ease of reading, let us recall the events
\begin{align*}
\cA_3 &= \Big\{ 2  \Big( \frac{\bar{k}}{\bar{k}-1} - \frac{\rho}{\bar{k}}  \Big) (1+\delta') \log n\leq  \sum _{(i,j)\in B_2}  (Z^{(1)}_{ij}  ) ^2  \leq  2   \Big( \frac{\bar{k}}{\bar{k}-1} +  \frac{\rho}{\bar{k}}  \Big) (1+\delta') \log n\Big\},\\
\cA_4 &=\Big \{ \bar{k} (1-3 \kappa^{1/4} ) \sqrt{2(1 + \delta')\log n} \leq \sum_{i\neq j, i,j\in T} | Z^{(1)}_{ij} |  \leq  \bar{k} (1+3 \kappa^{1/4}) \sqrt{2(1 + \delta')\log n} \Big \}.
\end{align*}

 Note that using the fact  $\rho = 16\kappa$ and \eqref{361}, for sufficiently  large  $\delta$, the event 
$   \cA_3 \cap \cA_4\cap \cA_5$ implies that 
\begin{align*}
\frac{1}{2} \sum_{i\neq j,i'\neq j', i,j,i',j'\in T} &(|Z^{(1)}_{ij}|  - |Z^{(1)}_{i'j'} |)^2 \\
&= |T| (|T|-1)    \Big( \sum _{i\neq j, i,j\in T}  (Z^{(1)}_{ij}   )^2 \Big ) - \Big   (  \sum _{i\neq j, i,j\in T}  |Z^{(1)}_{ij}   |\Big )^2 
\\
 &\leq 2  |T| (|T|-1)       \Big( \frac{\bar{k}}{\bar{k}-1} +  \frac{\rho}{\bar{k}}  \Big) (1+\delta') \log n -2 \bar{k}^2 (1-3 \kappa^{1/4} )^2 (1+\delta') \log n \\
 &\leq  {   ( 32 ( \bar{k}-1) \kappa + 12 \kappa^{1/4} \bar{k}^2 )   (1+\delta') \log n \leq   C  \kappa^{1/4}  \bar{k}^2   (1+\delta') \log n},
\end{align*} 
{
where we used $|T|\leq \bar{k}$ and $\kappa \leq \kappa^{1/4}$  in the second  and the  last   inequality respectively.}
From this,
using the argument in  \eqref{340}, setting $S' =  \sum_{i,j\in T} |Z^{(1)}_{ij} |$, one can deduce that 
\begin{align} \label{233}
 \sum_{i\neq j, i,j\in T} \Big(|Z^{(1)}_{ij}|  -  \frac{1}{|T|(|T|-1)} S' \Big)^2 \leq  C  \kappa^{1/4}     (1+\delta') \log n .
\end{align}
We check that  under the event $\cA'\cap \cA_4\cap  \cA_5$, there exists $\iota(\kappa) $ with $\lim_{\kappa \rightarrow 0} \iota   = 0$ such that
\begin{align}  \label{234}
\Big\vert  \frac{1}{|T|(|T|-1)} S'  - \frac{1}{h(\delta)}  \sqrt{2(1 + \delta')\log n} \Big\vert   \leq \Big ( \iota(\kappa)  + \frac{C}{h(\delta)} \Big)  \frac{1}{h(\delta)}  \sqrt{2(1 + \delta')\log n}. 
\end{align}
{In fact, first note that  by \eqref{230}, $|T| \geq  (1-\kappa_0^{1/4}) \bar{k}$  under the event $\cA'$. Also, $\bar{k} \geq h(\delta)$ under $\cA_5$ and we have the upper bound for $S'$ under $\cA_4$. Hence, combining these ingredients together, under the event $\cA'\cap \cA_4\cap  \cA_5$,   }
\begin{align*}
 \frac{1}{|T|(|T|-1)}  S' & \leq    \frac{1+3 \kappa^{1/4}}{1-\kappa_0^{1/4}     } \frac{1}{ \bar{k}(1-\kappa_0^{1/4}) -1} \sqrt{2(1 + \delta')\log n} \\
 &  \leq (1+10\kappa^{1/4} )  \frac{1}{ h(\delta) (1-\kappa_0^{1/4}) -1} \sqrt{2(1 + \delta')\log n}   \\
  &  \leq (1+10\kappa^{1/4} ) \Big(1+ 2\kappa_0^{1/4}  + \frac{2}{h(\delta)}\Big) \frac{1}{h(\delta)} \sqrt{2(1 + \delta')\log n} 
\end{align*}
{(recall $\kappa_0 = 40\kappa$, see \eqref{402})},
and the similar lower bound holds. This gives \eqref{234}.

Hence,
\eqref{233} and \eqref{234}  imply that for some $\iota'(\kappa) $ with $\lim_{\kappa \rightarrow 0} \iota' = 0$, {under the event $\cA_6$ (recall that $\cA_6 =\cA'\cap \cA_3\cap \cA_4    \cap \cA_5$) },
\begin{align*}
 \sum_{i\neq j, i,j\in T} \Big(|Z^{(1)}_{ij}|  -  \frac{1}{h(\delta)}  \sqrt{2(1 + \delta')\log n}  \Big)^2 \leq  \Big( \iota'(\kappa)  +  \frac{C}{h(\delta)}\Big )    (1+\delta') \log n.
\end{align*}
By  Cauchy-Schwarz inequality,  using the fact that $|T|\leq \bar{k}$, under the event $\cA_6$,
\begin{align} \label{235}
 \sum_{i\neq j, i,j\in T} \Big\vert |Z^{(1)}_{ij}|  -  \frac{1}{h(\delta)}  \sqrt{2(1 + \delta')\log n}  \Big\vert  \leq   C h(\delta)  \sqrt{ \Big( \iota'(\kappa)  +  \frac{C}{h(\delta)}\Big )     (1+\delta') \log n }.
\end{align}
Note that  under the event $\cA_5$,
\begin{align} \label{236}
\Big\vert  \sum_{i\neq j, i,j\in T} | Z_{ij} | - \sum_{i,j\in T} | Z^{(1)}_{ij} |  \Big\vert  \leq  \sum_{i,j\in T} | Z^{(2)}_{ij} |  \leq  \bar{k} ^2 \sqrt{\varepsilon \log \log n} \leq  C h(\delta)^2 \sqrt{\varepsilon \log \log n}    .
\end{align}
Hence, by the above two inequalities,  under the event $\cA_6$, for sufficiently large $n$,
\begin{align} \label{901}
 \sum_{i\neq j, i,j\in T} \Big\vert |Z_{ij}|  -  \frac{1}{h(\delta)}  \sqrt{2(1 + \delta')\log n}  \Big\vert  \leq   C h(\delta) \sqrt{ \Big( \iota'(\kappa)  +  \frac{C}{h(\delta)}\Big )     (1+\delta') \log n }.
\end{align}
By \eqref{191} and recalling $\varepsilon \leq \frac{1}{\delta^4}$, for large enough $\delta$,
the above implies
\begin{align} \label{237}
 \sum_{i\neq j,i,j\in T} \Big\vert |Z_{ij}|  -  \frac{1}{h(\delta)}  \sqrt{2(1 + \delta)\log n}  \Big\vert  \leq   C h(\delta) \sqrt{ \Big( \iota'(\kappa)  +  \frac{C}{h(\delta)}\Big )     (1+\delta) \log n }.
\end{align}  
{In fact, by the triangle inequality, the difference between LHS of \eqref{901} and \eqref{237} is bounded by
\begin{align*}
 \frac{\sqrt{\varepsilon} (1+\delta) \sqrt{\log n}}{h(\delta)} h(\delta)^2 \overset{\eqref{190}}{\leq} C  h(\delta) \frac{\sqrt{\log n}}{\delta} \overset{\eqref{argmaxloc}}{\leq} 
 C h(\delta) \sqrt{ \Big( \iota'(\kappa)  +  \frac{C}{h(\delta)}\Big )     (1+\delta) \log n }. 
\end{align*} 
 }
{Since $ \kappa_0 = 40\kappa$, by \eqref{230}, under the event $\cA_6$,  we have $|T| \geq (1-c\kappa^{1/4})\bar{k}$.  In addition, since  $h(\delta) \geq c\delta^{1/3}$, one can simplify  the term $ \iota'(\kappa) +  \frac{C}{h(\delta)}$ 
   to $\zeta(\kappa)$ with $\lim_{\kappa \rightarrow 0} \zeta(\kappa)  = 0$ if $\delta$ is chosen large enough depending on $\kappa$. Dividing both sides by $h(\delta)^2$ and  using
   \eqref{229} completes the proof.
}

\begin{remark}Note that \eqref{237} gives a bound depending on both $\delta$ and $\kappa$ and only on taking $\delta$ large enough depending on $\kappa$ yields the theorem. Further, even though we provided   sharp bounds for both $\ell_1$ and  $\ell_2$ norms, in fact, a lower bound for the former and an upper bound for the latter suffices.
\end{remark}

\section{Lower tail large deviations} \label{section 4}

We end with the short argument establishing the  large deviation probability of the lower tail, Theorem \ref{main theorem 2}.

\begin{proof}[Proof of Theorem \ref{main theorem 2}]
The upper bound   is an easy consequence of the inequality  \eqref{bound by max}. In fact, by    Lemma \ref{max gaussian}, \ref{lemma non zero} and \eqref{bound by max},
\begin{align*}
\mathbb{P}(\lambda_1  (Z)  \leq \sqrt{2(1-\delta) \log n} )  & \leq    \mathbb{P}( \max Z_{ij}   \leq \sqrt{2(1-\delta) \log n} )  \\
&\leq  \mathbb{E} \left(\mathbb{P}( \max Z_{ij} \leq \sqrt{2(1-\delta) \log n}  \mid X ) \1_{E_0}\right) + \mathbb{P}(E_0^c)  \\
&\leq  e^{-c' \frac{n^\delta}{\sqrt{\log n}}} + e^{-cn}.
\end{align*}

We now prove a matching lower bound.    
Define an event $\cS_\delta$, measurable with respect to $X$, by
\begin{align*}
\cS_\delta := \Big\{
\lambda_1 (X) \leq (1+\delta) \sqrt{\frac{\log n }{\log \log n}} \Big\}.
\end{align*}
Notice that $\P(S_{\delta}) \to 1$ by 
Lemma  \ref{lemma 32}.
 Since $ \lambda_1 (Z^{(2)})  \leq   \sqrt{\varepsilon \log \log n}  \cdot  \lambda_1 (X) $, conditionally on $X$,  under the event $\cS_\delta$, it holds that 
\begin{align*}
 \lambda_1(Z^{(2)}) \leq \sqrt{\varepsilon}(1+\delta) \sqrt{\log n}  .
\end{align*}
Since $\lambda_1 (Z)\leq  \lambda_1(Z^{(1)}) +\lambda_1(Z^{(2)}) $,
 \begin{align*}
 \mathbb{P} (\lambda_1  (Z) \leq \sqrt{2(1-\delta) \log n} ) \geq \P(\lambda_1(Z^{(1)}) \leq   \sqrt{2(1-\delta'') \log n},\,\, \lambda_1(Z^{(2)}) \leq \sqrt{\varepsilon}(1+\delta) \sqrt{ \log n}), \ 
 \end{align*}
where $\delta''>0$ is defined by $\sqrt{2(1-\delta'')} = \sqrt{2(1-\delta)} - \sqrt{\varepsilon}(1+\delta)$.
Recalling the definition of $\cF_0=\cD_{4\delta'}\cap \cC_{4\delta'} \cap \cE_{4\delta'} \cap {\sf{Few-cycles}}$
from \eqref{147}, analogously we define $\cF_3 :=  \cD_{4\delta''}\cap \cC_{4\delta''} \cap \cE_{4\delta''} \cap {\sf{Few-cycles}}\cap \cS_\delta$, we have
\begin{align} \label{133}
 \mathbb{P} (\lambda_1  (Z) \leq \sqrt{2(1-\delta) \log n} )  \geq  \mathbb{E} \left[ \mathbb{P}(\lambda_1(Z^{(1)}) \leq   \sqrt{2(1-\delta'') \log n} \mid X, X^{(1)} ) \1_{\cF_3}\right].
\end{align}
Above we use that $\cF_3$ is measurable with respect to the sigma algebra generated by $\{X^{(1)},X\}.$
We now estimate $\mathbb{P}(\lambda_1(Z^{(1)}) \leq   \sqrt{2(1-\delta'') \log n} \mid X, X^{(1)} ) $ under the event $\cF_3$ and finally we will use that $\cF_3$ is likely. We will crucially use throughout the proof that given $X^{(1)},$ $Z^{(1)}$ and $X$ are conditionally independent. 

Let $C_1,\cdots,C_m$ be $X^{(1)}$'s connected components and denote by $k_i$ the size of maximal clique in $C_i$. 
Let  \begin{align*}
I := \{i = 1,\cdots,m: k_i \geq 3\},\quad J:= \{i = 1,\cdots, m: k_i=2\}
\end{align*} 
and define  $\xi :=  (2 \varepsilon^{1/2} +  8 \varepsilon \delta'' )^{1/4}$. By Proposition  \ref{prop} with $\gamma = \varepsilon$ and $\eta = \varepsilon^{1/4}$,  for sufficiently small  $\varepsilon>0$, under the event $\cF_3$, for $i\in I$,
\begin{align} \label{135}
\mathbb{P}(\lambda_1(Z^{(1)}_i) \geq \sqrt{2(1-\delta'') \log n} \mid X, X^{(1)})   < n^{-\frac{1}{2} (1- \xi)^2 (1-\delta'') + \frac{1+4\delta''}{2}\varepsilon^{1/2} + \varepsilon}
\end{align}
using the fact that $\frac{k}{k-1}\ge 1,$
and for $i\in J$,
\begin{align} \label{136}
\mathbb{P}(\lambda_1(Z^{(1)}_i) \geq \sqrt{2(1-\delta'') \log n} \mid  X, X^{(1)})   < n^{-  (1- \xi)^2 (1-\delta'') + \frac{1+4\delta''}{2}\varepsilon^{1/2} + \varepsilon}.
\end{align}
Since $|I|< \log n$ under the event ${\sf {Few-Cycles}}$,  by \eqref{135} and \eqref{136},
\begin{align} \label{134}
\mathbb{P}(\lambda_1(Z^{(1)}_i) & \leq \sqrt{2(1-\delta) \log n}, \ \forall i \mid   X, X^{(1)}   )  \nonumber \\
& > (1- n^{-  (1- \xi)^2 (1-\delta'') + \frac{1+4\delta''}{2}\varepsilon^{1/2} + \varepsilon})^n (1-n^{-\frac{1}{2} (1- \xi)^2 (1-\delta'') + \frac{1+4\delta''}{2}\varepsilon^{1/2} + \varepsilon})^{\log n} \nonumber \\
&\geq \frac{1}{2} \exp(-n^{1-(1- \xi)^2 (1-\delta'') + \frac{1+4\delta''}{2}\varepsilon^{1/2} + \varepsilon}).
\end{align}
Since $\mathbb{P}(\cF_3) \geq \frac{1}{2}$ and $\varepsilon>0$ is arbitrary small,  by \eqref{133} and \eqref{134}, proof is concluded.
\end{proof}

\appendix

\section{Key estimates}

In this appendix, we include the outstanding proofs of basic properties about Gaussian random variables.
as well as the proof of Lemma \ref{lemma non zero} involving a straightforward application of Chernoff's bound.

\begin{proof}[Proof of Lemma \ref{max gaussian}]
Recalling the basic tail bounds from \eqref{tail},
 for some constant $c_1>0$,
\begin{align*}
\mathbb{P}( \max_{i=1,\cdots,m} X_i\geq   \sqrt{2(1+\delta) \log n} )   &= 1- (1- \mathbb{P}(X_1 \geq  \sqrt{2(1+\delta) \log n} )) ^m  \geq  c_1 \frac{1}{n^{\delta} \sqrt{\log n}}.
\end{align*}
Similarly, for some constant $c_2>0$,
\begin{align*}
\mathbb{P}( \max_{i=1,\cdots,m} X_i\leq   \sqrt{2(1-\delta) \log n} )   &= ( 1 -  \mathbb{P}(X_1 \geq   \sqrt{2(1-\delta) \log n} ) )^m  
\leq   e^{-c_2 \frac{n^\delta}{\sqrt{\log n}}}.
\end{align*}
\end{proof}

\begin{proof}[Proof of Lemma \ref{lemma non zero}]
We use the Chernoff's bound for Bernoulli variables for $q>p$:
\begin{align}
\mathbb{P}( \text{Bin}(m,p) \geq mq) \leq e^{-m I_p(q)},
\end{align}
where $I_p(x): = x\log \frac{x}{p}+(1-x) \log \frac{1-x}{1-p}$ is the relative entropy function.
Thus,
\begin{align} \label{312}
\mathbb{P}\Big( \text{Bin}\Big( \frac{n(n-1)}{2}, 1-  \frac{d}{n} \Big) \geq  \frac{n(n-1)}{2} \Big(1-\frac{d}{4n}\Big)\Big) \leq e^{-\frac{n(n-1)}{2} I_{1- \frac{d}{n} }(1-\frac{d}{4n} )},
\end{align}
Using $\log (1+x) \geq \frac{x}{2}$ for small positive $x$,
\begin{align} \label{313}
I_{1-\frac{d}{n}}\Big(1-\frac{d}{4n} \Big)    \geq \Big( 1-\frac{d}{4n} \Big)  \frac{3d}{8(n-d)}  + \frac{d}{4n} \log \frac{1}{4} \geq  \frac{C_1}{n-d} - \frac{C_2}{n^2}.
\end{align}
Hence,  by \eqref{312}  and \eqref{313}, there exists a constant $c>0$ such that for sufficiently large $n$, 
\begin{align*}
\mathbb{P}\Big( \text{Bin}\Big( \frac{n(n-1)}{2}, 1-  \frac{d}{n} \Big) \geq  \frac{n(n-1)}{2} \Big(1-\frac{d}{4n}\Big)\Big)  \leq e^{-cn}.
\end{align*}
This implies that
\begin{align*}
\mathbb{P}\Big( \text{Bin}\Big( \frac{n(n-1)}{2}, \frac{d}{n} \Big) \leq  \frac{n(n-1)}{2} \frac{d}{4n}\Big)  \leq e^{-cn},
\end{align*}
which concludes the proof.

\end{proof}

\begin{proof}[Proof of Lemma \ref{chi tail}]
Recall that we are aiming to show \begin{align*}
\mathbb{P}(\tilde{Y}_1^2+\cdots+\tilde{Y}_m^2 \geq   L )  \leq  C^m e^{-\frac{1}{2}L} e^{\frac{1}{2}m} \Big(\frac{L}{m}\Big)^{m}  e^{\frac{1}{2} \varepsilon m \log \log n},
\end{align*}
and  in particular, for any  $a,b,c>0$, if $ m \leq b\frac{\log n}{\log \log n}+c$ and $L = a\log n$, then, for any $\gamma>0$,  for sufficiently large $n$,
\begin{align} \label{212appendix}
\mathbb{P}(\tilde{Y}_1^2+\cdots+\tilde{Y}_m^2 \geq   a\log n )  \leq    n^{-\frac{a}{2} + \frac{\varepsilon b}{2} + \gamma} .
\end{align}

By exponential Chebyshev's  bound, for any $t>0$,
\begin{align} \label{500}
\mathbb{P}(\tilde{Y}_1^2+\cdots+\tilde{Y}_m^2 \geq    L )   \leq e^{-  tL }  (\mathbb{E}e^{t\tilde{Y}_1^2})^m.
\end{align}
{
Using the lower bound for the tail \eqref{tail}, the probability density function of  $\tilde{Y}$, denoted by $\tilde{f}(x)$ for $|x| \geq \sqrt{\varepsilon \log \log  n}  $, satisfies 
\begin{align*}
 \tilde{f}(x) \leq  \frac{C}{  ( \sqrt{\varepsilon \log \log  n}  ) ^{-1}  e^{-\frac{1}{2}\varepsilon \log \log n} }  e^{-\frac{1}{2}x^2}  = C\sqrt{\varepsilon \log \log  n}   e^{\frac{1}{2}\varepsilon \log \log n}   e^{-\frac{1}{2}x^2}  .
\end{align*}
Hence, using the upper bound for the tail \eqref{tail}, by making a change of variable $x = \frac{1}{\sqrt{1-2t}}y$,
\begin{align*} 
\mathbb{E}e^{t\tilde{Y}_1^2} &\leq  C \sqrt{\varepsilon \log \log  n} e^{\frac{1}{2}\varepsilon \log \log n} \int_{\sqrt{\varepsilon \log \log  n}}^\infty  e^{tx^2} e^{-\frac{1}{2}x^2} dx \\
&= C \sqrt{\varepsilon \log \log  n} e^{\frac{1}{2}\varepsilon \log \log n}  \frac{1}{\sqrt{1-2t}}  \int_{ \sqrt{1-2t} \sqrt{\varepsilon \log \log  n}}^\infty   e^{-\frac{1}{2}y^2} dy    \\
& \leq  C \sqrt{\varepsilon \log \log  n} e^{\frac{1}{2}\varepsilon \log \log n}  \frac{1}{\sqrt{1-2t}}   \frac{1}{\sqrt{1-2t} \sqrt{\varepsilon \log \log  n}  }     e^{-\frac{1}{2}(1-2t) \varepsilon \log \log n} =C      \frac{1}{1-2t}e^{t \varepsilon \log \log n}.
\end{align*}
Applying this to \eqref{500},
\begin{align*}
\mathbb{P}(\tilde{Y}_1^2+\cdots+\tilde{Y}_m^2 \geq    L )   \leq C^m e^{-tL}  \frac{1}{(1-2t)^m} e^{t\varepsilon m \log \log n}.
\end{align*}
We take  $t = \frac{1}{2} (1- \frac{m}{  L   } ) < \frac{1}{2}$ (recall that $L>m$) in order to balance two terms  $e^{-tL}$ and $  \frac{1}{(1-2t)^m} $. We conclude the proof of \eqref{211}.
}

{We now show \eqref{212appendix}. We first check that for any $L>0$, a function $x\mapsto (\frac{L}{x})^x$ is increasing on $(0, \frac{L}{e})$.  This is because the derivative of $x \log (\frac{L}{x})$, which  is given by $\log (\frac{L}{x}) - 1$,  is positive for $x\in (0, \frac{L}{e})$.
Hence,  for any  $\gamma>0$,  for sufficiently large $n$, the LHS of \eqref{212appendix} is bounded by 
\begin{align*}
C^{b\frac{\log n}{\log \log n} +c}  n^{-\frac{a}{2} + \frac{\varepsilon b}{2}} n^{\frac{b}{2\log \log n}}  \Big (  \frac{a }{b} \log \log n  \Big)^{b\frac{\log n}{\log \log n} +c} \leq    n^{-\frac{a}{2} + \frac{\varepsilon b}{2}+\gamma}.
\end{align*}
Here, we used the fact that for large $n$,
$
(c_1 \log \log n)^{c_2 \frac{\log n}{\log \log n}} \leq n^{\frac{\gamma}{2}}.
$
 }

\end{proof}


 \bibliography{ldp_ref}
\bibliographystyle{plain}

\end{document}